\documentclass[1 leqno,11pt]{amsart}
\usepackage{amssymb, amsmath,latexsym,amsfonts,amsbsy, amsthm,mathtools,graphicx,color}
\usepackage{float}
\usepackage{hyperref}

\numberwithin{equation}{section}

\allowdisplaybreaks

\let\al=\alpha

\let\e=\varepsilon

\let\f=\frac

\let\th=\theta
\let\pa=\partial

\def\no{\noindent}

\newcommand{\beq}{\begin{equation}}
\newcommand{\eeq}{\end{equation}}
\newcommand{\ben}{\begin{eqnarray}}
\newcommand{\een}{\end{eqnarray}}
\newcommand{\beno}{\begin{eqnarray*}}
\newcommand{\eeno}{\end{eqnarray*}}

\newtheorem{theorem}{Theorem}[section]

\newtheorem{lemma}[theorem]{Lemma}
\newtheorem{proposition}[theorem]{Proposition}

\newtheorem{Theorem}{Theorem}[section]

\newtheorem{Corollary}[Theorem]{Corollary}

\begin{document}

\title{Tollmien-Schlichting waves near neutral stable curve}

\author[Qi Chen]{Qi Chen}
\address[Q. Chen]{School of Mathematical Sciences, Zhejiang University, Hangzhou, 310058,  P. R. China}
\email{chenqi123@zju.edu.cn}

\author{Di Wu}
\address[D. Wu]{School of Mathematics, South China University of Technology, Guangzhou, 510640,  P. R. China}
\email{wudi@scut.edu.cn}

\author{Zhifei Zhang}
\address[Z. Zhang]{School of Mathematical Sciences, Peking University, 100871, Beijing, P. R. China}
\email{zfzhang@math.pku.edu.cn}

\date{\today}

\maketitle

\begin{abstract}
	In this paper, we study the linear stability of boundary layer flows over a flat plate. Tollmien, Schlichting, Lin et al. found that there exists a neutral curve, which consists of two branches: lower branch $\al_{low}(Re)$ and upper branch $\al_{up}(Re)$. Here, $\al$ is the wave number and $Re$ is the Reynolds number. For any $\al\in(\al_{low},\al_{up})$, there exist unstable modes known as Tollmien-Schlichting (T-S) waves to the linearized Navier-Stokes system. These waves play a key role during the early stage of boundary layer transition. In a breakthrough work (\textit{Duke math Jour, 165(2016)}), Grenier, Guo, and Nguyen provided a rigorous construction of the unstable T-S waves. In this paper, we confirm the existence of the neutral stable curve. To achieve this, we develop a more delicate method for solving the Orr-Sommerfeld equation by borrowing some ideas from the triple-deck theory. This approach allows us to construct the T-S waves in a neighborhood of the neutral curve.
\end{abstract}

\section{Introduction}

In this paper, we are concerned with a classical problem in fluid mechanics: the stability and transition of laminar flow. It has been observed in Reynolds's experiment that the flow is in the laminar state when the Reynolds number ($Re$) is small. When the Reynolds number exceeds a certain critical value, the flow becomes unstable and can transition to turbulence \cite{Drazin-Reid}. To understand this phenomenon, we consider the flow over a flat plate. The motion of the fluid can be described by the incompressible Navier-Stokes equations in a
domain $\Omega=[0,+\infty)\times\mathbb{R}_+$:
\begin{align}\label{NS}
\left\{
\begin{aligned}
&\partial_t u^\nu+u^\nu\cdot\nabla u^\nu+\nabla p^\nu-\nu\Delta u^\nu=0\quad \text{in $[0,T]\times\Omega$},\\
&\nabla\cdot u^\nu=0\quad \text{in $[0,T]\times\Omega$},\\
&u^\nu|_{y=0}=0.
\end{aligned}
\right.
\end{align}
Here $u^\nu=\big(u^\nu_1, u^\nu_2\big)$ is the velocity field and $p^\nu$ is the pressure. When $\nu$ is very small, a thin boundary layer will be formed in a neighborhood of plate. In this layer, the behavior of the solution could be described by the Prandtl equation:
 \begin{eqnarray}\label{equ:P}
\left\{\begin{aligned}
&u^p\pa_x u^p+v^p\pa_Y u^p+\pa_xp^e|_{y=0}=\pa^2_{Y}u^p,\\
&\pa_xu^p+\pa_Y v^p=0,\\
&u^p|_{Y=0}=v^p|_{Y=0}=0,\quad\lim_{Y\rightarrow+\infty}u^p(x,Y)=u^e_1(t,x,0).
\end{aligned}\right.
\end{eqnarray}
For the outer flow $U^e=1,P^e=0$, the Prandtl equation \eqref{equ:P} has a self-similar solution  called Blasius solution with the form 
\begin{align*}
	(u_B,v_B)=\Big(f'(\zeta),\f{1}{2\sqrt{x+x_0}}(\zeta f'(\zeta)-f(\zeta))\Big),
\end{align*}
where $\zeta=\f{Y}{\sqrt{x+x_0}}$ with $x_0>0$ as a free parameter, and the profile $f$ satisfies
\begin{align*}
	\f12ff''+f'''=0,\quad f(0)=f'(0)=0,\quad\lim_{\zeta\to\infty}f'(\zeta)=1.
\end{align*}
The Blasius solution has been experimentally confirmed with remarkable accuracy as a basic flow over a flat plate \cite{Sch}. In \cite{GI}, Guo and Iyer justified the local Prandtl expansion with the Blasius solution as a leading-order approximation of the steady Navier-Stokes equations, see also \cite{GZ}. Iyer and Masmoudi \cite{IM} justified the global Prandtl expansion. In a recent work \cite{GWZ}, Guo, Wang, and Zhang proved that the Blasius solution is a stable solution of the unsteady Prandtl equation. This provides another indication that the Blasius solution is reasonable as a basic flow. See \cite{GN, GI, Iyer, WZ-aihp, WZ} and references therein for more related works.

Considering a local region in the streamwise direction, the flow $u^\nu_s=\Big(u(\f{y}{\sqrt{\nu}}),0\Big)$ with a Blasius profile $u(Y)$ can be seen as an approximation of the Blasius solution. To study the instability of the flow at high Reynolds number under the perturbations, we consider the linearized Naiver-Stokes(NS) system around the flow $u_s^\nu$, which takes as follows,
    \begin{align}\label{eq:NS-linear}
    	\left\{
\begin{aligned}
&\partial_t v^\nu+u(\frac{y}{\sqrt{\nu}})\partial_x v^\nu+\frac{1}{\sqrt{\nu}}v_2^\nu(\partial_Y u(\frac{y}{\sqrt{\nu}}),0) +\nabla q^\nu-\nu\Delta v^\nu=0,\\
&\nabla\cdot u^\nu=0,\\
&v^\nu|_{y=0}=0.
\end{aligned}
\right.
\end{align}
A traditional method studying the linear stability is the spectral analysis, that is, finding the
solution of \eqref{eq:NS-linear} in the form $e^{\lambda t}V$.  Then $V$ solves an eigenvalue problem. We say that the linearized NS system is spectrally unstable if there exists a solution $V$ with $\mathrm{Re}\lambda>0$. Otherwise, the linearized NS system is spectrally stable.

To proceed, we reformulate the problem. We first introduce the following rescaled variables:
\begin{align*}
	\tau=\frac{t}{\sqrt{\nu}},\quad X=\frac{x}{\sqrt{\nu}},\quad Y=\frac{y}{\sqrt{\nu}}.
\end{align*}
Then the linearized NS system \eqref{eq:NS-linear} is reduced to 
\begin{align}\label{eq:NS-linear-scale}
	\left\{
	\begin{aligned}
		&\partial_\tau v-\sqrt{\nu}\Delta_{X,Y}v+(v_2\partial_Y u,0)+u\partial_X v+\nabla_{X,Y}q=0,\\
		&\mathrm{div}_{X,Y}v=0,\quad v|_{Y=0}=0.
	\end{aligned}
	\right.
\end{align}

\subsection{Main result}

We try to find the solution to the linearized NS system \eqref{eq:NS-linear} in the form
\begin{align}\label{eq:TS-velocity}
	v^{\nu}_{TS}=(v_1^\nu,v_2^\nu)=e^{-\mathrm icnt}e^{\mathrm inx}(\partial_Y\phi(\frac{y}{\sqrt{\nu}}),-\mathrm i\nu^\f12n\phi(\frac{y}{\sqrt{\nu}})),
\end{align}
Let $\alpha:=\nu^\f12 n$. Then
\begin{align}
	v_{TS}=(v_1,v_2)=e^{-\mathrm i\alpha c\tau}e^{\mathrm i\alpha X}(\partial_Y\phi(Y),-\mathrm i\alpha\phi(Y))
\end{align}
is a solution to the rescaled linearized system \eqref{eq:NS-linear-scale}, where $\phi$ is a solution to the Orr-Sommerfeld equation 
\begin{align}\label{eq:OS}
\left\{\begin{aligned}
	&\varepsilon(\partial_Y^2-\alpha^2)^2\phi-(u-c)(\partial_Y^2-\alpha^2)\phi+u''\phi=0,~~\lim_{Y\to\infty}\phi(Y)=0,\\
	&\phi(0)=\partial_Y\phi(0)=0,
	\end{aligned}
\right.
\end{align}
where $\varepsilon:=\frac{1}{\mathrm i n}=\frac{\nu^\f12}{\mathrm i\alpha}$. We know from \eqref{eq:TS-velocity} that
\begin{itemize}
	\item If $\mathrm{Im}(c)<0$, $|e^{-\mathrm icnt}|=e^{-|\mathrm{Im}(c)|n t}\to0$ as $t\to\infty$, then $v_{TS}^\nu$ is stable.
	\item If $\mathrm{Im}(c)=0$, $|e^{-\mathrm icnt}|\sim 1$ as $t\to\infty$, then $v_{TS}^\nu$ is neutral stable.
	\item If $\mathrm{Im}(c)>0$, $|e^{-\mathrm icnt}|=e^{|\mathrm{Im}(c)|n t}\to+\infty$ as $t\to\infty$, then $v_{TS}^\nu$ is unstable.
\end{itemize}

The unstable (growing) modes are known as Tollmien-Schlichting (T-S) waves, which play a key role during the early stage of boundary layer transition. This instability was first predicted by Tollmien. Later on, Schlichting and Lin carried out a detailed calculation to reproduce Tollmien's neutral curve via asymptotic analysis  \cite{L}. Roughly speaking, there exist lower and upper marginal stability branches  $\al_{low}(\nu),\al_{up}(\nu)$ that depend on the profile $u(Y)$ such that when $\al\in(\al_{low},\al_{up})$, the Orr-Sommerfeld equation exist unstable solutions with $\mathrm{Im}(c)>0$. In the breakthrough work \cite{GGN}, Grenier, Guo and Nyugen constructed the T-S waves
for $\al$ lying in an interval $[a,b]$ with $\al_{low}\ll a<b\ll\al_{up}$ for generic flows $u(Y)$. Moreover, they developed a robust Rayleigh-Airy iteration method to solve the Orr-Sommerfeld equation. For nonlinear stability results of boundary layer flows, we refer to \cite{GMM, GMM-apde, GMa, GN-arma, GNg,GN-jmpa, CWZ1, CWZ2, BG-siam, BG-2024}.
For further linear stability results concerning channel flow or pipe flow, we refer to \cite{Rom, GGN-AM, AH, AH-arxiv, CWZ-CPAM, CWZ-CMP}.

For subsonic flow, Yang and Zhang \cite{YZ} provided the first rigorous construction of compressible Tollmien-Schlichting (T-S) waves of temporal mode by developing a quasi-compressible and Stokes iteration method. Recently, Masmoudi, Wang, and the last two authors constructed spatial and temporal unstable T-S waves for the entire subsonic regime by introducing a new quasi-incompressible-compressible iteration scheme \cite{MWWZ}. They also constructed multiple unstable acoustic modes, known as Mack modes, in the supersonic boundary layer \cite{MWWZ2}.

Before stating our main result, we first outline the structural assumptions on the background shear flow:
\begin{eqnarray}\label{eq:Hyper-u}
\begin{split}
&u(0)=0,\quad \lim_{Y\to\infty}u(Y)=1,\quad u'(0)>0,\quad u''(0)<0,\\
&\sup_{Y\geq0}|\partial_Y^k(u(Y)-1)e^{\eta_0 Y}|<+\infty,\quad k=0,\dots,4\text{ with }\eta_0\geq C>0.
\end{split}
\end{eqnarray}
For this kind of flows equipped with analyticity, Grenier, Guo and Nyugen \cite{GGN} constructed the unstable T-S waves for $\al\in(A\nu^\f18,B\nu^{\f{1}{12}})$ with $B\ll 1\ll A$, which matches the physical results quite well. However, physical results also suggest that the transition of linear stability occurs on the neutral curve. Recently, Bian and Grenier \cite{BG-2023} made some progress on the existence of the neutral curve for the analytic basic flow.

The main goal of this paper is to confirm the existence of the neutral stable curve. To this end, we will show that the stable and neutrally stable T-S waves actually exist near the neutral
curve for sufficiently small viscosity $\nu$.  
Moreover, we will also attempt to elucidate the transition mechanism of linear stability for boundary flows by studying the stability transition of T-S waves.
\smallskip

Our main result is stated as follows.

\begin{theorem}\label{them:main}
	Let $u(Y)$ be a shear flow profile satisfying \eqref{eq:Hyper-u}. Then there exist  $\nu_0\ll1$ and $A\leq 1\leq B$ such that for any $\nu\leq \nu_0$ the following statement holds true. For any $\al\in(A\nu^{\f18},B\nu^{\frac{1}{12}})$, we can find $c(\al)$ such that there exists a solution $v_{TS}^\nu\in W^{2,\infty}$ to the linearized Navier-Stokes equation \eqref{eq:NS-linear} in the form of 
	\begin{align*}
		v^{\nu}_{TS}(t,x,y)=e^{-\mathrm ic\al \nu^{-\f12}t}e^{\mathrm i\al \nu^{-\f12}x}(\partial_Y\phi(\frac{y}{\sqrt{\nu}}),-\mathrm i\nu^\f12n\phi(\frac{y}{\sqrt{\nu}})).
	\end{align*}
  Moreover,
	\begin{itemize}
		\item There exist $A_1\geq A_c\geq A_0>A$ such that 
		      \begin{align*}
		      	&\mathrm{Im}(c(A_0\nu^\f18))<0,\quad \mathrm{Im}(c(A_c\nu^\f18))=0,\quad\mathrm{Im}(c(A_1\nu^\f18))>0.
		      \end{align*}
		\item There exsit $B_0\leq B_c\leq B_1<B$ such that 
		       \begin{align*}
		      	&\mathrm{Im}(c(B_0\nu^\frac{1}{12}))>0,\quad \mathrm{Im}(c(B_c\nu^\frac{1}{12}))=0,\quad\mathrm{Im}(c(B_1\nu^\frac{1}{12}))<0.
		      \end{align*}
		\item If $\nu^{\f18}\ll \al\ll\nu^{\frac{1}{12}}$, then 
		      \begin{align*}
		      	\alpha\mathrm{Im}(c)\sim\nu^\f14.
		      \end{align*}
	\end{itemize}
	\end{theorem}

    Let us now provide some remarks regarding our results.
    
    	\begin{itemize}
    	\item  In Theorem \ref{them:main}, we prove the existence of the stable Tollmien-Schlichting waves near $\al\sim\nu^{\f18}$ and $\al\sim\nu^{\frac{1}{12}}$. These results suggest that the lower branch of the neutral curve should be located at $\al\sim\nu^{\f18}$, and the upper branch of the neutral curve should be located at  $\al\sim\nu^{\frac{1}{12}}$.
	
    	\item We remove the analytic regularity assumption of the profile $u(Y)$ in \cite{GGN} by introducing a modified Langer transformation.
	
    	\item  The assumption $u''(0)<0$ is unnecessary for the construction of T-S wave when $\al$ is around the lower branch. This allows the Blasius profile. The upper branch depends
    	on the profile of $u(Y)$, with $\al_{up}\sim\nu^{\f{1}{20}}$ for the Blasius profile. In this case, our proof still works for the construction of the unstable T-S waves. However, the proof of the transition for the upper branch would be much involved.
	
    	\item Inspired by \cite{GGN}, the main strategy for constructing Tollmien-Schlichting waves $v^\nu_{TS}$ is to use a modified Rayleigh-Airy iteration. This method constructs two linearly independent solutions(slow mode $\phi_s$ and fast mode $\phi_f$) to the homogeneous Orr-Sommerfeld equation without matching the boundary conditions. Subsequently, we obtain the eigenvalue $c$ and its properties by solving dispersion equation
    	\begin{align*}
    		\frac{\phi_s(0)}{\partial_Y\phi_s(0)}=\frac{\phi_f(0)}{\partial_Y\phi_f(0)}.
    	\end{align*} 
    	However, compared with \cite{GGN},  the  Rayleigh equation with the same eigenvalue $c$ is singular and can not approximate the OS equation in the upper and main deck when $\al$ is on the neutral curve. \textbf{One of main improvements and achievements in this paper is that we observe a refined approximation of the OS equation in the upper and main deck when $\al$ is on the neutral curve.} This ensures the above dispersion relation exists for $c_i\leq 0$ and leads us to find the transition mechanism of linear stability near the neutral curve. 	
    \end{itemize}
    
	\subsection{Sketch of the proof}
	We introduce the following notations of three important operators
	\begin{align}
		\begin{split}
			&OS[\cdot]=\varepsilon(\partial_Y^2-\alpha^2)^2-(u- c)(\partial_Y^2-\alpha^2)+u'',\\
			&Ray_c[\cdot]=(u- c)(\partial_Y^2-\alpha^2)-u'',\\
			&Airy[\cdot]=\varepsilon(\partial_Y^2-\alpha^2)-(u-c).
		\end{split}
	\end{align}
	The key to constructing a non-zero solution to \eqref{eq:OS} is to construct two linearly independent solutions $\phi_s$ and $\phi_f$, to the following homogeneous OS equation
	\begin{align*}
		OS[\phi]=0,\quad\lim_{Y\to\infty}\phi(Y)=0,
	\end{align*}
for $\nu^\f18\lesssim\al\lesssim\nu^{\f{1}{12}}$, $c_r\sim\al$ and $|c_i|\ll\al$.
	Moreover, the slow mode $\phi_s$ and fast mode $\phi_f$ are constructed around $\varphi_{Ray}$ and $\psi_a$ respectively, where
	\begin{align*}
		Ray_c[\varphi_{Ray}]=0,\quad Airy[(\partial_Y^2-\alpha^2)\psi_a]=0,\quad\lim_{Y\to\infty}\varphi_{Ray}(Y)=\lim_{Y\to\infty}\psi_a(Y)=0.
	\end{align*}
	\subsubsection{The triple-deck structure and modified Rayleigh equation} In the case where $c_r\geq C|\varepsilon|^\f13$, there exist three crucial scales such that the homogeneous OS equation can be approximated by different simpler equations. In details,
	\begin{align*}
		OS[\phi]=0\sim
		\left\{
		\begin{aligned}
			&(\partial_Y^2-\alpha^2)\phi=0,\quad\phi\sim e^{-\alpha Y},\quad Y\sim \alpha^{-1}:\text{ upper deck},\\
			&(u-c)\partial_Y^2\phi-u''\phi=0,\quad\phi\sim(u-c)\log(u-c),\quad |Y-Y_c|\sim Y_c:\text{ main deck,}\\
			&Airy[(\partial_Y^2-\alpha^2)\phi]=0,\quad\phi\sim Ai(\kappa(Y-Y_c)),\quad |Y-Y_c|\sim|\varepsilon|^\f13:\text{ lower deck},
		\end{aligned}
		\right.
	\end{align*}
	where $u(Y_c)=c_r$, $\kappa=|\varepsilon|^{-\f13}u'(Y_c)$ and $Ai(\cdot)$ is the classical Airy function.  The lower deck, also known as the sublayer, appears when $Y_c\geq |\varepsilon|^\f13$. From the triple-deck structure, we formally have that	
	\begin{align}\label{eq:app-os-ray}
		OS[\phi]\sim -Ray_c[\phi]\text{ in upper and main deck},
	\end{align}
	and 
	\begin{align*}
		OS[\phi]\sim Airy[(\partial_Y^2-\alpha^2)\phi]\text{ in lower deck(sublayer)}.
	\end{align*}
	However, unlike the construction of the unstable mode ($c_i>0$) in \cite{GGN}, the approximation \eqref{eq:app-os-ray} fails on the neutral curve. In fact, on the neutral curve $c_i=0\Longrightarrow u(Y_c)-c=0$, the operator $Ray_c[\cdot]$ is singular.  
	
	\textbf{The main difficulty is finding a good approximation of $OS[\cdot]$ in the upper and main deck around the neutral curve.} To address this, our key observation is that
	\begin{align}\label{eq:idea-diff}
		\text{Diffusion effect}\Longrightarrow \text{Regularity}\Longrightarrow\text{Push $c$ to the left in $\mathbb C$}.
	\end{align}
	This means that we can find some $\hat c=c+\mathrm i c_0$ with $c_0>0$ such that 
	\begin{align*}
		Ray_{\hat c}[\varphi]=(u-\hat c)(\partial_Y^2-\alpha^2)\varphi-u''\varphi
	\end{align*}
	is a good approximation of $OS[\cdot]$ for $|Y-Y_c|\geq CY_c$. Here the size of $c_0$ is critical because we need to ensure that  $\hat c$ contains a strong enough regularized effect 
	from the diffusion term($c_0\gg|\varepsilon|$) and that the extra errors $\mathrm{i}c_0(\partial_Y^2-\al^2)\varphi$ generated by adding $\mathrm{i}c_0$ can be treated as a small perturbation. We observe that the decay solution $\varphi_{Ray}$ to the homogeneous Rayleigh solution
	\begin{align*}
		\varphi_{Ray}(Y)\sim \underbrace{(u-c)e^{-\alpha Y}}_{\text{upper deck}}+\underbrace{u'(0)^{-1}\alpha e^{-\eta_0 Y}}_{\text{main deck}}\text{ for $Y_c\lesssim|Y-Y_c|$},
	\end{align*}
	which implies that 
	\begin{align*}
		(\partial_Y^2-\alpha^2)\varphi_{Ray}(Y)\sim \underbrace{u''(Y)e^{-\alpha Y}}_{\text{upper deck}}+\underbrace{\frac{u''(Y)\alpha e^{-\eta_0 Y}}{u'(0)(u-c)}}_{\text{main deck}}\text{ for $Y_c\lesssim|Y-Y_c|$}.
	\end{align*}
	The influence of the diffusion term for  $Y_c\lesssim|Y-Y_c|$  is similar to
	\begin{align*}
		&\varepsilon(\partial_Y^2-\alpha^2)^2\varphi_{Ray}(Y)\sim\varepsilon(\partial_Y^2-\alpha^2)(u''(Y)e^{-\alpha Y}+\frac{u''(Y)\alpha e^{-\eta_0 Y}}{u'(0)(u-c)})\sim\underbrace{\mathrm i|\varepsilon|}_{\text{upper main}} +\underbrace{\mathrm i|\varepsilon|\alpha^{-2}}_{\text{main deck}},
	\end{align*}
   from which, we can guess $c_0\in[|\varepsilon|\al^{-2},|\varepsilon|]$. On the other hand, it requires $c_0\ll|\varepsilon|^\f13$ to ensure that the extra errors $\mathrm{i}c_0(\partial_Y^2-\al^2)\varphi$ can be treated as small perturbation. By the relation $|\varepsilon|=\nu^\f12\al^{-1}$ and $\al\geq C\nu^\f18$, we can choose $c_0=|\varepsilon|\al^{-s}$ with $0<s<2$. In our construction, we always take $c_0=|\varepsilon|\al^{-\f32}$. When $\al$ is on the neutral curve, the approximation of $OS[\cdot]$ in the upper and main deck  is 
  \begin{align}\label{eq:modify-ray}
  	-Ray_{\hat c}[\cdot]=-(u-\hat c)(\partial_Y^2-\al^2)+u'',\quad \hat c=c+\mathrm{i}|\varepsilon|\al^{-\f32}.
  \end{align}
  \textbf{This is the key observation for constructing slow/fast mode.} The modified Rayleigh operator $Ray_{\hat c}[\cdot]$ is a perturbation of the singular $Ray_c[\cdot]$ by incorporating the regularized effect of the diffusion term $\varepsilon(\partial_Y^2-\alpha^2)^2$. By \eqref{eq:idea-diff} we can say that  $Ray_{\hat c}[\cdot]$ not just approximates the OS operator $OS[\cdot]$ at the quantitative level but also at the spectral level.
  
\subsubsection{The modified Rayleigh-Airy iteration}
Inspired by \cite{GGN}, we apply a modified Rayleigh-Airy iteration to construct the slow mode $\phi_s$ and fast mode $\phi_f$. Compared with the Rayleigh-Airy iteration introduced in \cite{GGN},  we use the modified Rayleigh operator \eqref{eq:modify-ray} to approximate the OS operator $OS[\cdot]$, as the original Rayleigh operator $Ray_c[\cdot]$ is singular and cannot approximate $OS[\cdot]$ on the neutral curve. 

We introduce the modified Rayleigh-Airy iteration through the following non-homogeneous OS equation:
\begin{align*}
	OS[\phi_{non}]=f,\quad\lim_{Y\to \infty}\phi_{non}(Y)=0.
\end{align*}
We define 
	\begin{align*}
		Ray_{\hat c}[\varphi^{(0)}]=-f,\quad Airy[(\partial_Y^2-\alpha^2) \psi^{(0)}]=-\varepsilon(\partial_Y^2-\alpha^2)^2\varphi^{(0)}-\mathrm i|\varepsilon|\alpha^{-\f32}(\partial_Y^2-\alpha^2)\varphi^{(0)},
	\end{align*}
and for any $j\geq 1$,
\begin{align*}
	 &Ray_{\hat c}[\varphi^{(j)}]=u''\psi^{(j-1)}, \\
	 &Airy[(\partial_Y^2-\alpha^2)\psi^{(j)}]=-\varepsilon(\partial_Y^2-\alpha^2)^2\varphi^{(j)}-\mathrm i|\varepsilon|\alpha^{-\f32}(\partial_Y^2-\alpha^2)\varphi^{(j)}.
\end{align*}
Then we can formally obtain that
\begin{align*}
	\phi_{non}(Y)=\sum_{k=0}^\infty(\varphi^{(k)}+\psi^{(k)}).
\end{align*}
Then the slow mode $\phi_s$ and  the fast mode $\phi_f$ can be constructed in a suitable function space $\mathcal X\subset W^{4,\infty}$ as follows:
\begin{align}\label{eq:sp-slow-mode}
	\phi_s=\varphi_{Ray,\hat c}+\psi_s^{(1)}+\tilde\phi_s,\quad\phi_f=\psi_a+\tilde\phi_f,
\end{align}
where 
\begin{align*}
	&Ray_{\hat c}[\varphi_{Ray,\hat c}]=0,\quad Airy[(\partial_Y^2-\alpha^2)\psi_s^{(1)}]=-\varepsilon(\partial_Y^2-\alpha^2)^2\varphi_{Ray,\hat c}-\mathrm i|\varepsilon|\alpha^{-\f32}(\partial_Y^2-\alpha^2)\varphi_{Ray,\hat c},\\
	&Airy[(\partial_Y^2-\alpha^2)\psi_a]=0
\end{align*}
and $\tilde\phi_s$, $\tilde\phi_f$ are constructed via the modified Rayleigh-Airy iteration by solving the non-homogeneous OS equations
\begin{align*}
	OS[\tilde\phi_s]=-u''\psi_s^{(1)},\quad OS[\tilde\phi_f]=-u''\psi_a.
\end{align*}
Moreover, to ensure the iterative scheme converges, we construct solutions to the non-homogeneous Airy equation with three kinds of source terms matching the triple-deck structure. In detail, we show the estimates for the non-homogeneous Airy equations
\begin{align*}
	Airy[(\partial_Y^2-\alpha^2)\psi_{non}]=F\sim\left\{
	\begin{aligned}
		&e^{-\eta_0 Y},\quad\text{upper deck},\\
		&(u-\hat c)^{-1}e^{-\eta_0 Y},\quad\text{main deck},\\
		&(u-\hat c)^{-3}e^{-\eta_0 Y},\quad\text{lower deck(sublayer)}.
	\end{aligned}
	\right.
\end{align*}

\subsubsection{Dispersion relation}

With the slow mode $\phi_s$ and the fast mode $\phi_f$ in hand, we construct the solution $\phi$ to  \eqref{eq:OS} by taking 
\begin{align*}
	\phi=C_s\phi_s+C_f\phi_f,\text{ with }C_s\phi_s(0)+C_f\phi_f(0)=C_s\partial_Y\phi_s(0)+C_f\partial_Y\phi_f(0)=0.
\end{align*}
The existence of non-zero $C_s$ and $C_f$ is equivalent to the following dispersion relation holding for certain $(\alpha,c)$:
\begin{align*}
	\frac{\phi_s(0)}{\partial_Y\phi_s(0)}=\frac{\phi_f(0)}{\partial_Y\phi_f(0)},
\end{align*}
from which, we infer that
\begin{align}\label{eq:intro-dr}
	c\sim u'(0)^{-1}\alpha\underbrace{-\frac{u'(0)\tilde{\mathcal A}(2,0)}{\tilde{\mathcal A}(1,0)}}_{\text{related to Airy function}} ,\quad c_i\sim\frac{\alpha^2u''(0)\pi}{u'(0)^2}\underbrace{-\mathrm{Im}\Big(\frac{u'(0)\tilde{\mathcal A}(2,0)}{\tilde{\mathcal A}(1,0)}\Big)}_{\text{related to Airy function}} .
\end{align}
The dispersion relation \eqref{eq:intro-dr} holds for the case $-\delta_1(\nu)\leq c_i$ with $\delta_1(\nu)>0$. This is guaranteed by constructing the slow mode $\phi_s$ and the fast mode $\phi_f$ via the modified Rayleigh-Airy iteration. This allows  us to construct the neutral stable mode by solving \eqref{eq:intro-dr}. \textbf{This is our main novelty and improvement compared with \cite{GGN}}

From the above dispersion relation, we can infer that for $\al\in(A_0\nu^\f18,B_0\nu^{\f{1}{12}})$ with $A_0\ll 1\ll B_0$, $c_r\sim\nu^\f18$ for $\alpha\sim\nu^\f18$ and $c_r\sim\alpha$ for $\alpha\gg\nu^\f18$. Moreover,
\begin{itemize}
	\item for $\al\in(A_0\nu^\f18,B\nu^{\f{1}{12}})$ with $B\ll 1$
	      \begin{align*}
	      	c_i\sim-\mathrm{Im}\Big(\f{u'(0)\tilde{\mathcal{A}}(2,0)}{\tilde{\mathcal{A}}(1,0)}\Big)
	      	\left\{
	      	\begin{aligned}
	      		&<0,\quad\al=A\nu^\f18\text{ with }A\sim 1,\\
	      		&>0,\quad\al\gg\nu^\f18.
	      	\end{aligned}
	      	\right.
	      \end{align*}
      \item for $\al=B\nu^{\f{1}{12}}$ with $B\gg 1$,
      \begin{align*}
      	c_i\sim\frac{\alpha^2u''(0)\pi}{u'(0)^2}<0.
      \end{align*}
\end{itemize}
Hence, there exist $A_c$ and $B_c$ such that $c_i=0$ for $\al=A_c\nu^\f18$ or $\al=B_c\nu^{\f{1}{12}}$. From the above discussion, we see that the instability originates from the sublayer. However, the stable effects near the lower branch and the upper branch are quite different. Around the lower branch, the transition mechanism is the competition between the stabilizing effect generated by diffusion and the instability caused by the sublayer.  Around the upper branch,  the stability due to
the structure of the background flow $u''(0)<0$ suppresses the instability generated by the sublayer when $\al$ crosses the upper branch of the neutral curve.

	\subsection{Organization}
	The rest of this paper is organized as follows. Section 2 is devoted to the Rayleigh equation, in which we construct the solutions to the Rayleigh equation and show some key estimates. Section 3 contains the constructions of solutions to the Airy equation with three kinds of source terms related to the Rayleigh-Airy iteration. In Section 4, we construct the slow mode and fast mode of the Orr-Sommerfeld equation by a modified Rayleigh-Airy iteration, respectively. In Section 5, we present the proof of Theorem \ref{them:main} by solving the dispersion equation.


\section{The Rayleigh equation}

In this section, we consider the non-homogeneous Rayleigh  equation on $\mathbb{R}_+$:
\begin{align}\label{eq:Ray-force}
\left\{
\begin{aligned}
&(u(Y)-c)(\partial_Y^2-\alpha^2)\varphi (Y)-u''\varphi (Y)=F(Y),\quad Y>0,\\
&\lim_{Y\to\infty}\varphi(Y)=0,
\end{aligned}
\right.
\end{align}
and the homogeneous Rayleigh equation:
\begin{align}\label{eq:Ray-homo}
\left\{
\begin{aligned}
	&(u(Y)-c)(\partial_Y^2-\alpha^2)\varphi_{Ray} (Y)-u''\varphi_{Ray}(Y)=0,\quad Y>0,\\
&\lim_{Y\to\infty}\varphi(Y)=0.
\end{aligned}
\right.
\end{align} 
For convenience, we denote 
\begin{align*}
Ray[\cdot]=(u-c)(\partial_Y^2-\al^2)-u''.
\end{align*}
We construct the solution to the Rayleigh equation in the following functional spaces:
\begin{align}\label{def: Y_theta}
\|f\|_{\mathcal{Y}_\theta}=&\|(u-c)^3\pa_Y^4 f\|_{L^\infty_{\theta}}+\|(u-c)^2\pa_Y^3 f\|_{L^\infty_{\theta}}+\|(u-c)\pa_Y^2 f\|_{L^\infty_{\theta}}+\|\pa_Y f\|_{L^\infty_{\theta}}+\| f\|_{L^\infty_{\theta}},
\end{align}
and
\begin{align}\label{def: Y_bar-theta}
\|f\|_{\tilde{\mathcal Y}_\theta}=\|(u-c)\partial_Y^2 f\|_{L^\infty_\theta}+\|\pa_Y f\|_{L^\infty_{\theta}}+\| f\|_{L^\infty_{\theta}},
\end{align}
where
\begin{align*}
\|f\|_{L^\infty_\theta}=\|e^{\th Y} f\|_{L^\infty}.
\end{align*}

We denote
\begin{align}\label{def:H1set}
\mathbb H_1=\big\{(\alpha,c)\in\mathbb R_+\times \mathbb C:(|\alpha|+|c|)|\log c_i|\ll1, c_r\geq C_1c_i>0\big\}.
\end{align}
In this section, we always assume that $(\al, c)\in \mathbb H_1$.

\subsection{Non-homogeneous Rayleigh equation}
For the non-homogeneous Rayleigh equation \eqref{eq:Ray-force}, we have the following result.

\begin{proposition}\label{pro: phi_non}
	If  $e^{\theta Y}F\in L^\infty$, then there exists a solution $\varphi_{non}\in \tilde{\mathcal Y}_\theta$ to \eqref{eq:Ray-force} with 	
	\begin{align*}
	&\|\varphi_{non}\|_{\tilde{\mathcal Y}_\theta}\leq C|\log c_i|\|F\|_{L^\infty_\theta}.
	\end{align*}	
	In addition, if $e^{\theta Y} \partial_Y^jF\in L^\infty$ for $j=1,2$, then 	\begin{align}\label{eq:ray-force-Y}
	\|\varphi_{non}\|_{\mathcal Y_\theta}\leq C|\log c_i|\|F\|_{L^\infty_\theta}+C\|(u-c)\pa_YF\|_{L^\infty_\theta}+C\|(u-c)^2\pa_Y^2F\|_{L^\infty_\theta}.
	\end{align}
\end{proposition}
\begin{proof}
Applying Proposition 3.2 in \cite{MWWZ} by taking $m=0$, we can obtain 
\begin{align}\label{eq:ray-non-ty}
	&\|\varphi_{non}\|_{\tilde{\mathcal Y}_\theta}\leq C|\log c_i|\|e^{\theta Y}F\|_{L^\infty}.
	\end{align}

From \eqref{eq:Ray-force} and \eqref{eq:ray-non-ty}, we infer that 
\begin{align*}
	|e^{\theta Y}(u-c)^2\partial_Y^3\varphi_{non}(Y)|\leq& C|\log c_i|\|F\|_{L^\infty_\theta}+C\|(u-c)\pa_YF\|_{L^\infty_\theta},\\
	|e^{\theta Y}(u-c)^3\partial_Y^4\varphi_{non}(Y)|\leq& C|\log c_i|\|F\|_{L^\infty_\theta}+C\|(u-c)\pa_YF\|_{L^\infty_\theta}+C\|(u-c)^2\pa_Y^2F\|_{L^\infty_\theta},
	\end{align*}
which implies \eqref{eq:ray-force-Y}. 
\end{proof}

\subsection{Homogeneous Rayleigh equation}
In this part, we consider the Rayleigh equation on $\mathbb{R}_+$:
\begin{align}\label{eq:ray}
(u(Y)-c)(\partial_Y^2-\alpha^2)\varphi(Y)-u''(Y)\varphi(Y)=0,\quad Y>0.
\end{align}
Our task is to construct a non-zero decay solution $\varphi_{Ray}(Y)$ to \eqref{eq:ray}. In fact, we construct such solution around $\varphi_{Ray}^{(0)}$ defined by 
\begin{align}\label{eq:ray-main-def}
\varphi_{Ray}^{(0)}:=2\alpha e^{\alpha Y}(u-c)\int_Y^{+\infty}\frac{1}{e^{2\alpha Z}(u(Z)-c)^2}dZ.
\end{align}
By a direct calculation, we find that 
\begin{align*}
Ray[\varphi_{Ray}^{(0)}]=2\alpha u'\varphi^{(0)}_{Ray}.
\end{align*}
We define that $\varphi_{R}$ is the solution to the equation 
\begin{align*}
	Ray[\varphi_R]=-2\alpha u'\varphi^{(0)}_{Ray},
\end{align*}
which is constructed as in Proposition \ref{pro: phi_non}. Then 
\begin{align*}
	\varphi_{Ray}(Y)=\varphi_{Ray}^{(0)}(Y)+\varphi_{R}(Y),
\end{align*}
and Proposition \ref{pro: phi_non} suggests that $\varphi_{R   }(Y)$ is a $\mathcal O(\alpha)$ perturbation. In order to obtain the point-wise estimates  and the asymptotic expansion on the boundary for $\varphi_{Ray}$, we first show such estimates for $\varphi_{Ray}^{(0)}$.
\begin{lemma}\label{lem: asymp-phi_0}
	There exists $Y_0>0$ such that 
	\begin{itemize}
		\item For $Y\geq Y_0$,
	\begin{align*}
	\left|\varphi_{Ray}^{(0)}(Y)-\frac{e^{-\alpha Y}(u-c)}{(1-c)^2}\right|+\left|\partial_Y\varphi_{Ray}^{(0)}(Y)-\partial_Y\Big(\frac{e^{-\alpha Y}(u-c)}{(1-c)^2}\Big) \right|\leq C\alpha e^{-\eta_0 Y}.
	\end{align*}
		
		\item For $0\leq Y\leq Y_0$,
	\begin{align*}
	\left|\varphi_{Ray}^{(0)}(Y)-\frac{e^{-\alpha Y}(u-c)}{(1-c)^2}-\frac{2\alpha e^{-\eta_0 Y}}{u'(Y_c)} \right|\leq C\alpha|u(Y)-c||\log|c_i||,
	\end{align*}
	\begin{align*}
	\left|\partial_Y\varphi_{Ray}^{(0)}(Y)-\frac{e^{-\alpha Y}u'(Y)}{(1-c)^2} \right|\leq C\alpha|\log|c_i||.
	\end{align*}
		In particular, we have 
		\begin{align*}
		&\varphi_{Ray}^{(0)}(0)=-c+\frac{2\alpha}{u'(Y_c)}+\mathcal O\Big((|\alpha|+|c|)|c||\log c_i|\Big),\\
		&\pa_Y\varphi_{Ray}^{(0)}(0)=u'(0)+\mathcal O\Big((|\al|+|c|)|\log c_i|\Big).
		\end{align*}
	\end{itemize}
	Moreover, we have 
		\begin{align}\label{est: varphi^(0)_Ray,R}
		\big\|\varphi_{Ray}^{(0)}-(1-c)^{-2}(u-c)e^{-\alpha Y}\big\|_{\mathcal Y_{\eta_0}}\leq C|\alpha||\log c_i|.
		\end{align}
\end{lemma}
\begin{proof}
	According to Lemma 3.3 in \cite{MWWZ}, we can obtain the results in Lemma \ref{lem: asymp-phi_0} except for \eqref{est: varphi^(0)_Ray,R}. We denote 
	\begin{align*}
		\tilde\varphi^{(0)}(Y)=\varphi_{Ray}^{(0)}-(1-c)^{-2}(u-c)e^{-\alpha Y}.
	\end{align*}
	Then we have for $j=0,1,2$ 
	\begin{align*}
		\partial_Y^jRay_c[\tilde\varphi^{(0)}]=2\alpha\partial_Y^j( u''\tilde\varphi^{(0)}),\quad\|\tilde\varphi^{(0)}\|_{L^\infty_{\eta_0}}+\|\partial_Y\tilde\varphi^{(0)}\|_{L^\infty_{\eta_0}}\leq C\alpha|\log c_i|,
	\end{align*}
	which gives \eqref{est: varphi^(0)_Ray,R}.
	\end{proof}
	
 We provide a more precise expansion of the boundary value of $\varphi_{Ray}^{(0)}$ up to the second order,  as the second-order term plays a crucial role in the transition of stability near the upper branch of the neutral curve.

\begin{lemma}\label{lem: asymp-phi_02}
Let  $|c_i|\ll \min\{\alpha, c_r\}$. Then we have 

      \begin{align*}
      	&\mathrm{Im}\left(\varphi_{Ray}^{(0)}(0)\right)=-c_i-\frac{2\alpha c_ru''(Y_c)}{u'(Y_c)^3}\arg(-c)+o(\alpha(\alpha+c_r)),\\
      &\mathrm{Im}\left(\partial_Y\varphi_{Ray}^{(0)}(0) \right)=\frac{2\al u''(Y_c)}{u'(Y_c)^2}\arg(-c) +o(\alpha+c_r).
      \end{align*}
  
\end{lemma}

\begin{proof}
\textbf{Estimates of  $\varphi_{Ray}^{(0)}(0)$}. By a direct calculation, we can obtain 
\begin{align}\label{eq:ray-app-1}
\varphi_{Ray}^{(0)}(0)=-c-2\alpha c\int_0^{Y_0}e^{-2\alpha Z}\frac{1}{(u(Z)-c)^2}dZ+\mathcal{B}_0,
\end{align}
where 
\begin{align*}
\mathcal{B}_0=&c(1-\f{1}{(1-c)^2})+ \frac{2\alpha c}{(1-c)^2}\int_0^{Y_0}e^{-2\alpha Z}dZ-2\alpha c\int_{Y_0}^{+\infty}e^{-2\alpha Z}\left(\frac{1}{(u-c)^2}-\frac{1}{(1-c)^2} \right)dZ.
\end{align*}
We notice that
\begin{align}\label{eq:ray-main-Re}
\begin{split}
|\mathrm{Re}(\mathcal{B}_0)|\lesssim&|c|^2+\alpha c_r+\al|c|\int_{Y_0}^{+\infty}e^{-\eta_0Z}dZ+|c_i|c_r+\alpha|c|\lesssim\alpha c_r,
\end{split}
\end{align}
and 
\begin{align}\label{eq:ray-app-b1}
\begin{split}
|\mathrm{Im}(\mathcal{B}_0)|\lesssim&c_r|c_i|+\alpha|c_i|+\alpha|c_i|\int_{Y_0}^{+\infty}e^{-\eta_0 Z}dZ+c_r|c_i|\ll\al(\al+c_r).
\end{split}
\end{align}
Therefore, we need to show an asymptotic formula for $2\alpha c\int_0^{Y_0}e^{-2\alpha Z}\frac{1}{(u(Z)-c)^2}dZ$. In details, we have 
\begin{align}
\begin{split}
2\alpha c\int_0^{Y_0}e^{-2\alpha Z}\frac{1}{(u(Z)-c)^2}dZ=-\frac{2\al}{u'(Y_c)}+\frac{2\alpha c}{u'(Y_c)}\int_0^{Y_0}\frac{(u'(Y_c)-u'(Z))dZ}{e^{2\alpha Z}(u(Z)-c)^2}+\mathcal{B}_{1},
\end{split}
\end{align}
where 
\begin{align*}
\mathcal{B}_1=-\frac{2\alpha c}{e^{2\alpha}u'(Y_c)(u(Y_0)-c)}-\frac{4\alpha^2 c}{u'(Y_c)}\int_0^{Y_0}\frac{dZ}{e^{2\alpha Z}(u(Z)-c)}.
\end{align*}
Moreover, we notice that 
\begin{align}\label{eq:ray-main-key1-Re}
\begin{split}
|\mathrm{Re}(\mathcal{B}_1)|\lesssim\al c_r+\left|\frac{\alpha^2 c}{u'(Y_c)}\int_0^{Y_0}\frac{dZ}{e^{2\alpha Z}(u(Z)-c)}\right|\lesssim\al c_r+\al^2 c_r|\log|c_i||\lesssim\al c_r,
\end{split}
\end{align}
and 
\begin{align}\label{eq:ray-main-key1-Im}
\begin{split}
|\mathrm{Im}(\mathcal{B}_1)|\lesssim\al|c_i|+\left|\frac{\alpha^2 c}{u'(Y_c)}\int_0^{Y_0}\frac{dZ}{e^{2\alpha Z}(u(Z)-c)}\right|\lesssim\al c_i+\al^2 c_r|\log|c_i||\ll\al(\al+c_r).
\end{split}
\end{align}	
Therefore, by \eqref{eq:ray-app-1}-\eqref{eq:ray-main-key1-Im}, we obtain  
\begin{align}\label{eq:ray-app-2}
\varphi_{Ray}^{(0)}(0)=-c+\frac{2\al}{u'(Y_c)}-\frac{2\alpha c}{u'(Y_c)}\int_0^{Y_0}\frac{(u'(Y_c)-u'(Z))dZ}{e^{2\alpha Z}(u(Z)-c)^2}+\mathcal{B}_2,
\end{align}	
where $\mathcal{B}_2=\mathcal{B}_0+\mathcal{B}_1$ with 
\begin{align*}
|\mathrm{Re}(\mathcal{B}_2)|\lesssim\al(\al+c_r),\quad|\mathrm{Im}(\mathcal{B}_2)|\ll\al(\al+c_r).
\end{align*}

To give more precise expansion of $\varphi_{Ray}^{(0)}(0)$, we introduce 
	\begin{align*}
     g(Y)=u'(Y)-u'(Y_c)-\frac{u''(Y_c)}{u'(Y_c)^2}u'(Y)(u(Y)-c_r),
    \end{align*}
    which satisfies  
    \begin{align*}
    g(Y_c)=0,\quad g'(Y_c)=0\quad\text{and}\quad |g''(Y)|\leq C.
    \end{align*}
Then we have 
\begin{align}\label{eq:ray-main-key}
\begin{split}
\frac{2\alpha c}{u'(Y_c)}\int_0^{Y_0}\frac{(u'(Y_c)-u'(Z))dZ}{e^{2\alpha Z}(u(Z)-c)^2}=\frac{2\al cu''(Y_c)}{u'(Y_c)^3}\log(-c)+\mathcal{B}_3,
\end{split}
\end{align}
where 
\begin{align*}
\mathcal{B}_3=&-\frac{2\alpha cu''(Y_c)}{u'(Y_c)^3}e^{-2\alpha Y_0}\log(u(Y_0)-c) -\frac{4\alpha^2cu''(Y_c)}{u'(Y_c)^3}\int_0^{Y_0}e^{-2\alpha Z}\log(u(Z)-c)dZ\\
&-\frac{2\alpha cu''(Y_c)}{u'(Y_c)^3}\int_0^{Y_0}\frac{\mathrm i c_i u'(Z)}{e^{2\alpha Z}(u(Z)-c)^2}dZ-\frac{2\alpha c}{u'(Y_c)}\int_0^{Y_0}\frac{g(Z)dZ}{e^{2\alpha Z}(u(Z)-c)^2}\\
=&\mathcal{B}_{3,1}+\mathcal{B}_{3,2}+\mathcal{B}_{3,3}+\mathcal{B}_{3,4}.
\end{align*}
We also notice that 
    \begin{align}\label{eq:ray-main-key-err}
    	\begin{split}
    		&\left|\mathrm{Re}(\mathcal{B}_{3,1})\right|\lesssim \alpha|c|,\quad\left|\mathrm{Im}(\mathcal{B}_{3,1})\right|\leq C\alpha|c||c_i|\ll\alpha(\alpha+c_r),\\
    		&\left|\mathcal{B}_{3,2}\right|\leq C\alpha^2|c||\log|c_i||\ll\alpha(\alpha+c_r),\quad\left|\mathcal{B}_{3,3}\right|\leq C\alpha|c_i|\ll\alpha(\alpha+c_r).
    	\end{split}
    \end{align}
    By the definition of $g(Y)$, we find that for any $Y\in[0,Y_0]$, $|g(Y)|\leq C|Y-Y_c|^2$, from which we notice that for any $Z\in[0,Y_0]$,
    \begin{align*}
    	\frac{g(Z)}{(u(Z)-c)^2}=&\frac{g(Z)}{(u(Z)-c_r)^2}+g(Z)\left(\frac{1}{(u(Z)-c)^2} -\frac{1}{(u(Z)-c_r)^2}\right)\\
    	=&\frac{g(Z)}{(u(Z)-c_r)^2}+g(Z)\left(\frac{c_i^2}{(u-c_r)^2(u-c)^2}+\frac{2\mathrm i c_i}{(u-c_r)(u-c)^2} \right).
    \end{align*}
    Then by the property of $g$, we know that for any $Z\in[0,Y_0]$,
    \begin{align*}
    	&\left|g(Z)\left(\frac{c_i^2}{(u-c_r)^2(u-c)^2}+\frac{2\mathrm i c_i}{(u-c_r)(u-c)^2} \right)\right|\\
    	&\leq C|c_i|\frac{1}{|u(Z)-c|}\leq \frac{C|c_i|}{|Z-Y_c|+|c_i|},
    \end{align*}
    which implies that for any $Z\in[0,Y_0]$,
    \begin{align*}
    	&\left|\mathrm{Re}\left(\frac{g(Z)}{(u(Z)-c)^2}\right) \right|\leq C+\frac{C|c_i|}{|Z-Y_c|+|c_i|},\quad\left|\mathrm{Im}\left(\frac{g(Z)}{(u(Z)-c)^2}\right) \right|\leq \frac{C|c_i|}{|Z-Y_c|+|c_i|}.
    \end{align*}
    As a consequence, we have 
    \begin{align*}
    	&\left|\mathrm{Re}\left(\int_0^1\frac{g(Z)dZ}{e^{2\alpha Z}(u(Z)-c)^2}\right)\right|\leq C+C|c_i|\int_0^1\frac{1}{|Z-Y_c|+|c_i|}dZ\leq C(1+|c_i||\log|c_i||),\\
    	&\left|\mathrm{Im}\left(\int_0^1\frac{g(Z)dZ}{e^{2\alpha Z}(u(Z)-c)^2}\right)\right|\leq C\int_0^1\frac{1}{|Z-Y_c|+|c_i|}dZ\leq C|c_i|\log|c_i|.
    \end{align*}
    From the above bounds, we get
    \begin{align*}
    	&\left|\mathrm{Re}(\mathcal{B}_{3,4})\right|\leq C\alpha c_r,\quad \left|\mathrm{Im}(\mathcal{B}_{3,4})\right|\leq C\alpha c_r|c_i||\log|c_i||\ll\alpha(\alpha+c_r).
    \end{align*}
    which along with \eqref{eq:ray-main-key} and \eqref{eq:ray-main-key-err} imply that 
    \begin{align}\label{eq:ray-main-key-RI}
   	\left|\mathrm{Re}(\mathcal{B}_{3})\right|\leq C\alpha c_r,\quad\left|\mathrm{Im}(\mathcal{B}_{3}) \right|\leq C\alpha |c_i|\ll\alpha(\alpha+c_r).
    \end{align}
    Then by \eqref{eq:ray-app-2}, \eqref{eq:ray-main-key} and \eqref{eq:ray-main-key-RI}, we obtain  
    \begin{align*}
    \varphi_{Ray}^{(0)}(0)=-c+\frac{2\al}{u'(Y_c)}-\frac{2\al cu''(Y_c)}{u'(Y_c)^3}\log(-c)+\mathcal{B}_4,
    \end{align*}
    where $\mathcal{B}_4=\mathcal{B}_2+\mathcal{B}_3$ with
    \begin{align*}
    	&\left|\mathrm{Re}(\mathcal{B}_4)\right|\leq C\alpha(\al+c_r),\quad\left|\mathrm{Im}(\mathcal{B}_4)\right|\ll\alpha(\alpha+c_r).
    \end{align*}
    In particular, we have the following asymptotic expansion for $\varphi_{Ray}^{(0)}(0)$:
   \begin{align}\label{eq:ray-app-3}
   \begin{split}
   &\mathrm{Im}\left(\varphi_{Ray}^{(0)}(0)\right)=-c_i-\frac{2\alpha c_ru''(Y_c)}{u'(Y_c)^3}\arg(-c)+o(\alpha(\alpha+c_r)).
   \end{split}
   \end{align}
	
	\no\textbf{Estimates of $\partial_Y\varphi_{Ray}^{(0)}(0)$.} We first have 
	\begin{align}\label{eq:ray-main-pboundary}
	\begin{split}
	\partial_Y\varphi_{Ray}^{(0)}(0)=&u'(0)+2\al u'(0)\int_0^{Y_0}e^{-2\al Z}(u-c)^{-2}dZ+2\al c^{-1}+\mathcal{D}_0,
	\end{split}
	\end{align}
	where 
	\begin{align*}
	\mathcal{D}_0=&u'(0)((1-c)^{-2}-1)+\al\varphi_{Ray}^{(0)}(0)-\f{2\al u'(0)}{(1-c)^2}\int_0^{Y_0}e^{-2\al Z}dZ\\
	&+2\al u'(0)\int^{Y_0}_0e^{-2\al Z}\Big((u-c)^{-2}-(1-c)^{-2}\Big)dZ.
	\end{align*}
	By a similar argument as in \eqref{eq:ray-app-b1}, we have  
	\begin{align}\label{eq:ray-main-pboundary1}
	|\mathrm{Re}\mathcal{D}_0|\lesssim\al+c_r,\quad|\mathrm{Im}\mathcal{D}_0|\lesssim |c_i|+\al(\al+c_r)\ll\al.
	\end{align}
	We  find that 
	\begin{align}\label{eq:ray-main-d-key}
	2\alpha u'(0)\int_0^{Y_0}e^{-2\alpha Z}\frac{1}{(u(Z)-c)^2} dZ+2\al c^{-1}=2\al\int_0^{Y_0}\frac{(u'(Y_c)-u'(Z))}{e^{2\alpha Z}(u(Z)-c)^2}dZ+\mathcal{D}_1,
	\end{align}
	where 
	\begin{align*}
	\mathcal{D}_1=&\frac{2\alpha(u'(Y_c)-u'(0))}{cu'(Y_c)}-\frac{2\alpha u'(0)}{e^{2\alpha}u(Y_c)(u(Y_0)-c)}-\frac{4\alpha^2u'(0)}{u'(Y_c)}\int_0^{Y_0}\frac{dZ}{e^{2\alpha Z}(u(Z)-c)}\\
	&+2\al u'(0)(u'(Y_c)^{-1}-u'(0)^{-1})\int_0^{Y_0}\frac{(u'(Y_c)-u'(Z))}{e^{2\alpha Z}(u(Z)-c)^2}dZ.
	\end{align*}
    Moreover, we notice that 
    \begin{align}
    \begin{split}
    &|\mathrm{Im}(\mathcal{D}_1)|\lesssim\al|c_i||c|^{-1}+\al|c_i|+\al^2|c_i|+\al c_r|\log|c_i||\ll\al.
    \end{split}
    \end{align}
    On the other hand, by a similar argument as in \eqref{eq:ray-main-key} and \eqref{eq:ray-main-key-RI}, we obtain 
    \begin{align}\label{eq:ray-app-p1}
    2\al\int_0^{Y_0}\frac{(u'(Y_c)-u'(Z))}{e^{2\alpha Z}(u(Z)-c)^2}dZ=\f{2\al u''(Y_c)}{u'(Y_c)^2}\log(-c)+\mathcal{D}_2,
    \end{align}
    where 
    \begin{align*}
    |\mathrm{Re}(\mathcal{D}_2)|\lesssim\al,\quad |\mathrm{Im}(\mathcal{D}_2)|\lesssim\al c_r|\log|c_i||+|c_i|\ll\al.
    \end{align*}
	Therefore, by \eqref{eq:ray-main-pboundary}-\eqref{eq:ray-app-p1}, we obtain 
	\begin{align}
	\partial_Y\varphi_{Ray}^{(0)}(0)=&u'(0)+\f{2\al u''(Y_c)}{u'(Y_c)^2}\log(-c)+\mathcal{D}_3,
	\end{align}
	where $\mathcal{D}_3=\mathcal{D}_0+\mathcal{D}_1+\mathcal{D}_2$ satisfies
	\begin{align*}
	|\mathrm{Re}(\mathcal{D}_3)|\lesssim\al,\quad |\mathrm{Im}(\mathcal{D}_3)|\lesssim\al c_r|\log|c_i||+|c_i|\ll\al.
	\end{align*}
	In particular, we have
	\begin{align*}
	\mathrm{Re}\left(\partial_Y\varphi_{Ray}^{(0)}(0) \right)=u'(0)+\mathcal O(\alpha),\quad\mathrm{Im}\left(\partial_Y\varphi_{Ray}^{(0)}(0) \right)=\frac{2\al u''(Y_c)}{u'(Y_c)^2}\arg(-c) +o(\alpha).
	\end{align*}
	
    The proof is completed.  \end{proof}
	 
Now we are in a position to construct the solution $\varphi_{Ray}$ to the homogeneous Rayleigh equation. Recall that 
\begin{align*}
Ray[\varphi_{Ray}^{(0)}]=2\al u'\varphi_{Ray}^{(0)}.
\end{align*}
As we mentioned at the beginning of this subsection, we shall construct $\varphi_{Ray}$ in the form of 
\begin{align*}
	\varphi_{Ray}(Y)=\varphi_{Ray}^{(0)}(Y)+\varphi_{R}(Y),
\end{align*}
where $\varphi_{R}=\varphi_1+\varphi_2$ with
\begin{align*}
\varphi_1(Y)=&-2\al(u-c)\int_Y^{+\infty}(u-c)^{-2}\int_{Y'}^\infty u'\varphi_{Ray}^{(0)}dY''dY',
\end{align*}
and $Ray[\varphi_2]=\al^2(u-c)\varphi_1$ constructed  in Proposition \ref{pro: phi_non}.
\begin{lemma}\label{lem: varphi_R}
	It holds that
	\begin{itemize}
		\item For $Y\geq Y_0$,
		\begin{align*}
		&|\varphi_R(Y)|\leq C|\al| e^{-\eta_0 Y},\quad |\pa_Y\varphi_R(Y)|\leq C|\al| e^{-\eta_0 Y}.
		\end{align*}
		
		\item For $0\leq Y\leq Y_0$,
		\begin{align*}
		&\varphi_R(Y)=-\f{\al e^{-\al Y}}{u'(Y_c)}+\mathcal O\Big(|\al|(|\al|+|u-c|)|\log c_i|\Big),\\
		&\pa_Y\varphi_R(Y)=\mathcal{O}(|\al| |\log c_i|).
		\end{align*}
		In particular, we have
		\begin{align*}
		&\varphi_R(0)=-\f{\al}{u'(Y_c)}+\mathcal{O}\Big((|\al|^2+|c|^2)|\log c_i|\Big),\\
		&\pa_Y\varphi_R(0)=\mathcal{O}(|\al| |\log c_i|).
		\end{align*}
		\end{itemize}
		Moreover, we have
		\begin{align}\label{est: varphi_R-Y-1}
		\|\varphi_R\|_{\mathcal Y_{\eta_0}}\leq C|\alpha||\log c_i|.
		\end{align}
\end{lemma}
\begin{proof}
	According to Lemma 3.4 in \cite{MWWZ}, we can obtain the results in the above lemma except for \eqref{est: varphi_R-Y-1}. Moreover, from Lemma 3.4 in \cite{MWWZ}, we can obtain 
	\begin{align*}
		\|\varphi_R\|_{\tilde{\mathcal Y}_{\eta_0}}\leq C\alpha|\log c_i|, 
	\end{align*}
	which along with \eqref{est: varphi^(0)_Ray,R} and the fact $\partial_Y^jRay_c[\varphi_R]=-2\alpha\partial_Y(u'\varphi^{(0)}_{Ray})$ gives \eqref{est: varphi_R-Y-1}.
\end{proof}

\begin{lemma}\label{lem:ray-R-0}
	Let $|c_i|\ll \min\{\alpha, c_r\}$. We obtain the following estimates for $\varphi_R(0)$:
	\begin{align*}
	\mathrm{Im}\left(\varphi_{R}(0)\right)=\frac{\alpha c_ru''(Y_c)}{u'(Y_c)^3}\arg(-c)+o(\alpha(\alpha+c_r)),
	\end{align*}
	and 
	 \begin{align*}
	\mathrm{Im}\left(\partial_Y\varphi_{R}(0) \right)=\frac{\alpha u''(Y_c)}{u'(Y_c)^2}\arg(-c) +o(\alpha+c_r),
	\end{align*}
\end{lemma}

\begin{proof}
According to the definition of $\varphi_R$, we know that $\varphi_R(0)=\varphi_1(0)+\varphi_2(0)$. Moreover, by Proposition  \ref{pro: phi_non} and Lemma \ref{lem: varphi_R}, 
we can obtain
\begin{align*}
|\partial_Y\varphi_2(0)|+|\varphi_2(0)|\lesssim\al^2(\al+c_r)|\log|c_i||\ll\al(\al+c_r).
\end{align*} 
Then we have 
\begin{align}\label{eq:ray-R-BV}
\begin{split}
&\mathrm{Im}(\varphi_R(0))=\mathrm{Im}(\varphi_1(0))+o(\al(\al+c_r)),\\ &\mathrm{Im}(\partial_Y\varphi_R(0))=\mathrm{Im}(\partial_Y\varphi_1(0))+o(\al(\al+c_r)).
\end{split}
\end{align}
Hence, we are left with the estimates for $\varphi_1(0)$ and $\partial_Y\varphi_1(0)$. We first notice that 
\begin{align}\label{eq:ray-R-phi1}
\begin{split}
\varphi_1(0)=&2\al c\int_0^{+\infty}(u-c)^{-2}\int_{Y}^{+\infty} u'\varphi_{Ray}^{(0)}(Y')dY'dY=I_1+I_2+I_3.
\end{split}
\end{align}
where
\begin{align*}
I_1=&\frac{2\alpha c}{u'(Y_c)}\int_0^{Y_0}\frac{u'(Y_c)-u'(Y')}{(u(Y')-c)^2}\int_{Y}^{+\infty}u'\varphi_{Ray}^{(0)}(Y')dY'dY,\\
I_2=&\f{2\al}{u'(Y_c)}\int_0^{+\infty}u'\varphi_{Ray}(Y)dY-\f{2\al c}{u'(Y_c)}\int_0^{Y_0}u'(u-c)^{-1}\varphi_{Ray}^{(0)}(Y)dY,\\
I_3=&2\al c\int_{Y_0}^{+\infty}(u-c)^{-2}\int_{Y}^{+\infty} u'\varphi_{Ray}^{(0)}(Y')dY'dY-\f{2\al c}{u'(Y_c)(u(Y_0)-c))}\int_{Y_0}^{+\infty}u'\varphi_{Ray}^{(0)}dY.
\end{align*}

By Lemma \ref{lem: asymp-phi_0}, we know that for any $Y\geq 0$,
\begin{align}\label{eq:phi-exp-r}
\varphi_{Ray}^{(0)}(Y)=(u-c)(1-c)^{-1}e^{-\al Y}+\mathcal{E}(Y),\quad|\mathcal{E}(Y)|\lesssim(\al+c_r)e^{-\eta_0Y}.
\end{align}
Moreover, we have that for $Y\geq Y_0$
\begin{align*}
\mathrm {Im}(\mathcal{E}(Y))=&2\alpha e^{\alpha Y}\mathrm{Im}\left((u-c)\int_Y^{+\infty}e^{-2\alpha Z}\left(\frac{1}{(u(Z)-c)^2}-\frac{1}{(1-c)^2} \right)dZ\right)\\
=&-2\alpha c_ie^{-\alpha Y}\int_Y^{+\infty}e^{-2\alpha Z}\mathrm{Re}\left(\frac{1}{(u(Z)-c)^2}-\frac{1}{(1-c)^2} \right)dZ\\
&+2\alpha(u-c_r)e^{-\alpha Y}\int_Y^{+\infty}e^{-2\alpha Z}\mathrm{Im}\left(\frac{1}{(u(Z)-c)^2}-\frac{1}{(1-c)^2} \right)dZ,
\end{align*}
which implies that 
\begin{align*}
|\mathrm{Im}(\mathcal{E}(Y))|\lesssim\al |c_i|e^{-\eta_0 Y},\quad\forall Y\geq Y_0.
\end{align*}
This along with the fact $|\mathrm{Im}(\mathcal{E}(Y))|\lesssim\al(\al+c_r)|\log|c_i||$ for $0\leq Y\leq Y_0$ in Lemma \ref{lem: asymp-phi_0} deduces that for any $Y\geq 0$,
\begin{align}\label{eq:phi-e-Im}
|\mathrm{Im}(\mathcal{E}(Y))|\lesssim\al(\al+c_r)|\log|c_i||e^{-\eta_0 Y}.
\end{align}
Therefore, we notice that for any $0\leq Y\leq Y_0$,
\begin{align}\label{eq:int-u'-phi-Im}
\begin{split}
&\Big|\mathrm{Im}\Big(\int_Y^{+\infty}u'(Y')\varphi_{Ray}^{(0)}(Y')dY'\Big)\Big|\\
&\lesssim\Big|\mathrm{Im}\Big((1-c)^{-1}\int_{Y}^{+\infty}u'(u-c)e^{-\al Y'}dY'\Big)\Big|+\int_Y^{+\infty}u'|\mathrm{Im}\mathcal{E}(Y'))|dY'\\
&\lesssim \al(\al+c_r)|\log|c_i||.
\end{split}
\end{align}
Moreover, for any $0\leq Y\leq Y_0$,
\begin{align}\label{eq:int-u'-phi}
\begin{split}
&\int_Y^{+\infty}u'(Y')\varphi_{Ray}^{(0)}(Y')dY'\\
&=(1-c)^{-1}\int_Y^{+\infty}u'(Y')(u-c)e^{-\al Y'}dY'+\int_{Y}^{+\infty}u'(Y')\mathcal{E}(Y')dY'\\
&=\int_Y^{\al^{-\f12}}u'(Y')(u(Y')-c)dY'+\mathcal{I},
\end{split}
\end{align}
 where 
 \begin{align*}
 \mathcal{I}=&(1-c)^{-1}\int_Y^{\al^{-\f12}}u'(Y')(u-c)e^{-\al Y'}dY'-\int_Y^{\al^{-\f12}}u'(Y')(u(Y')-c)dY'\\
 &+(1-c)^{-1}\int_{\al^{-\f12}}^{+\infty}u'(Y')(u-c)e^{-\al Y'}dY'+\int_{Y}^{+\infty}u'(Y')\mathcal{E}(Y')dY'.
 \end{align*}
By integration by parts, we have that for $Y\in[0,Y_0]$,
\begin{align*}
\int_Y^{\al^{-\f12}}u'(Y')(u(Y')-c)dY'=\f12(u(\al^{-\f12})-c)^2-\f12(u(Y)-c)^2=\f12-\f12(u(Y)-c)^2+\mathcal{O}(\al+c_r).
\end{align*}
About $\mathcal{I}$, we have that for $Y\in[0,Y_0]$,
\begin{align*}
|\mathcal{I}(Y)|\lesssim&|1-(1-c)^{-1}|\int_Y^{+\infty}e^{-\eta_0 Y'}dY'+\int_Y^{\al^{-\f12}}e^{-\eta_0 Y}|e^{\al Y'}-1|dY'\\
&+\int_{\al^{-\f12}}^{+\infty}e^{-\eta_0 Y'}dY'+\al\int_{Y}^{+\infty}e^{-\eta_0 Y'}dY'\lesssim\al+c_r.
\end{align*}
From \eqref{eq:int-u'-phi} and the above two estimates, we can deduce that for $Y\in[0,Y_0]$,
\begin{align}\label{eq:int-u'-phi-1}
\int_Y^{+\infty}u'(Y')\varphi_{Ray}^{(0)}(Y')dY'=\f12-\f12(u(Y)-c)^2+\mathcal{O}(\al+c_r).
\end{align}
 By a similar argument as in \eqref{eq:ray-main-key}-\eqref{eq:ray-main-key-RI} and applying the estimates \eqref{eq:int-u'-phi-Im} and \eqref{eq:int-u'-phi-1},  we obtain 
\begin{align}\label{eq:ray-R-I1-BV}
\mathrm{Im}(I_1)=\frac{\alpha c_ru''(Y_c)}{u'(Y_c)^3}\arg(-c) +o(\alpha(\al+c_r)).
\end{align}

About $I_{2}$, we notice that by \eqref{eq:int-u'-phi-Im},
\begin{align*}
|\mathrm{Im}(I_2)|\lesssim \al^2(\al+c_r)|\log|c_i||+\al c_r\Big|\mathrm{Im}\Big(\int_0^{Y_0}u'(u-c)^{-1}\varphi_{Ray}^{(0)}(Y)dY\Big)\Big|.
\end{align*}
On the other hand, by \eqref{eq:phi-exp-r}, we have  
\begin{align*}
\Big|\mathrm{Im}\Big(\int_0^{Y_0}u'(u-c)^{-1}\varphi_{Ray}^{(0)}(Y)dY\Big)\Big|\lesssim c_i+\Big|\int_0^{Y_0}u'(u-c)^{-1}\mathcal{E}(Y)dY\Big|\lesssim(\al+c_r)|\log|c_i||.
\end{align*}
Therefore, we obtain  
\begin{align}\label{eq:ray-R-I2-BV}
|\mathrm{Im}(I_2)|\lesssim \al^2(\al+c_r)|\log|c_i||.
\end{align}
Again by \eqref{eq:phi-exp-r} and \eqref{eq:int-u'-phi-Im}, we can obtain 
\begin{align}\label{eq:ray-R-I3-BV}
\begin{split}
|\mathrm{Im}(I_3)|\lesssim&\al c_r|c_i|\int_{Y_0}^{+\infty}\int_Y^{+\infty}|u'\varphi_{Ray}^{(0)}(Y')|dY'+\al^2 c_r\int_{Y_0}^\infty e^{-\eta_0 Y}dY+\al^2(\al+c_r)\\
\ll&\al(\al+c_r).
\end{split}
\end{align}
Therefore, by \eqref{eq:ray-R-BV}, \eqref{eq:ray-R-phi1}, \eqref{eq:ray-R-I1-BV}-\eqref{eq:ray-R-I3-BV}, we get 
  \begin{align*}
  \mathrm{Im}\left(\varphi_{R}(0)\right)=\frac{\alpha c_ru''(Y_c)}{u'(Y_c)^3}\arg(-c)+o(\alpha(\alpha+c_r)).
  \end{align*}
  
  Now we are in a position to show the estimates for $\partial_Y\varphi_R(0)$. We first notice that 
  \begin{align*}
   \partial_Y\varphi_1(0)=&-2\al u'(0)\int_0^{+\infty}(u-c)^{-2}\int_{Y}^{+\infty}u'\varphi_{Ray}^{(0)}(Y')dY'dY\\
  &+2\al c^{-1}\int_0^{+\infty}u'(Y)\varphi_{Ray}^{(0)}(Y)dY=II_1+II_2+II_3,
  \end{align*}
  where
  \begin{align*}
  II_1=&\f{2\al  u'(0)}{u'(Y_c)}\int_0^{Y_0}\f{u'(Y)-u'(Y_c)}{(u-c)^2}\int_Y^{+\infty}u'\varphi_{Ray}^{(0)}(Y')dY'dY,\\
  II_1=&2\al c^{-1}\Big(1-\f{u'(0)}{u'(Y_c)}\Big)\int_0^{+\infty}u'(Y)\varphi_{Ray}^{(0)}(Y)dY-\f{2\al u'(0)}{u'(Y_c)}\int_0^{Y_0}\f{u'}{u-c}\varphi_{Ray}^{(0)}(Y)dY,
  \end{align*}
  and 
  \begin{align*}
  II_3=&-2\al u'(0)\int_{Y_0}^{+\infty}(u-c)^{-2}\int_{Y}^{+\infty}u'\varphi_{Ray}^{(0)}(Y')dY'dY\\
  &+\f{2\al u'(0)}{u'(Y_c)(u(Y_0)-c)}\int_{Y_0}^{+\infty}u'\varphi_{Ray}^{(0)}(Y)dY.
  \end{align*}
By a similar argument in the estimates of $\mathrm{Im}(\varphi_R(0))$, we can obtain 
\begin{align*}
\mathrm{Im}(II_1)=\frac{\alpha u''(Y_c)}{u'(Y_c)^2}\arg(-c) +o(\alpha+c_r),\quad|\mathrm{Im}(II_2)|+|\mathrm{Im}(II_3)|\ll\al+c_r.
\end{align*}

The proof is completed.
\end{proof}

From Lemma \ref{lem: asymp-phi_0}-\ref{lem:ray-R-0}, we directly conclude that 

\begin{proposition}\label{prop:ray-homo}
    Let $(\al,c)\in\mathbb{H}_1$. There exists a solution $\varphi_{Ray}\in \mathcal Y_{\al}$ to \eqref{eq:ray}, for which it holds that
	\begin{align*}
	&\varphi_{Ray}(0)=-c+\f{\al}{u(Y_c)}+\mathcal{O}((|\al|^2+|c|^2) |\log |c_i||),\\
	&\pa_Y\varphi_{Ray}(0)=u'(0)+\mathcal{O}(|\al|| \log |c_i||).
	\end{align*}
	Moreover, we have 
	\begin{align*}
	\|\varphi_{Ray}-(u-c)e^{-\al Y}\|_{\mathcal Y_{\eta_0}}\leq C|\alpha||\log c_i|.
	\end{align*}
	In particular, for $|c_i|\ll \min\{\alpha, c_r\}$, we have
	 \begin{align*}
	&\mathrm{Im}\left(\varphi_{Ray}(0)\right)=-c_i-\frac{\alpha c_ru''(Y_c)}{u'(Y_c)^3}\arg(-c)+o(\alpha(\alpha+c_r)),\\
	&\mathrm{Im}\left(\partial_Y\varphi_{Ray}(0) \right)=\frac{\al u''(Y_c)}{u'(Y_c)^2}\arg(-c) +o(\alpha+c_r).
	\end{align*}

\end{proposition}

\section{The Airy equation}

In this section, we consider the following Airy equation 
\begin{align}\label{eq:Airy-eq}
\begin{split}
	&\varepsilon(\partial_Y^2-\alpha^2)^2\psi-(u-c)(\partial_Y^2-\alpha^2)\psi =F,\quad Y>0,\\
	&\lim_{Y\to\infty}\psi(Y)=0,
\end{split}
\end{align}
with several kinds of source terms $F$. We construct a solution $\psi_a$ to the corresponding homogeneous Airy equation around a modified Airy function
\begin{align*}
	&\varepsilon(\partial_Y^2-\alpha^2)^2\psi_a-(u-c)(\partial_Y^2-\alpha^2)\psi_a =0,\quad Y>0,\\
	&\lim_{Y\to\infty}\psi_a(Y)=0.
\end{align*}
We define 
\begin{align*}
	\eta_{out}(Y):=\left\{
	\begin{array}{ll}
		\left(\f32\int_{Y_c}^Y \left(\frac{u(Z)-c_r}{u'(Y_c)}\right)^\f12dZ \right)^\f23,\quad Y\geq Y_c,\\
		-\left(\f32\int^{Y_c}_Y \left(\frac{c_r-u(Z)}{u'(Y_c)}\right)^\f12dZ \right)^\f23,\quad Y\leq Y_c,
	\end{array}
	\right.
\end{align*}
and $\eta_{in}(Y):=Y-Y_c$. Then we introduce the modified Langer transformation
\begin{align}\label{eq:Langer-mod}
	\eta(Y):=\eta_r(Y)+\mathrm i\eta_i
\end{align}
with $\eta_i:=-u'(Y_c)^{-1}(|\varepsilon|\alpha^2+c_i)$ and 
\begin{align}
	\eta_r(Y):=\chi\Big(\frac{Y-Y_c}{\kappa^{-1} M}\Big)\eta_{in}(Y)+(1-\chi\Big(\frac{Y-Y_c}{\kappa^{-1} M}\Big))\eta_{out}(Y),
\end{align}
where $\chi(\cdot)$ is a smooth function in $\mathbb R_+$ satisfying $\chi(Y)\equiv1$ on $[0,1]$ and $\chi(Y)=0$ on $[2,+\infty)$. We also introduce the modified Airy function
\begin{align*}
	A_1(Y):=-\mathrm i|\varepsilon|^{-\f23}u'(Y_c)^{-\f13} Ai(e^{\mathrm i\frac{\pi}{6}}\kappa\eta(Y)),~~A_2(Y):=2\pi  Ai(e^{\mathrm i\frac{5\pi}{6}}\kappa\eta(Y))
\end{align*}
with $\kappa:=|\varepsilon|^{-\f13}u'(Y_c)^\f13$ matching the scale of sublayer. We also define 
\begin{align*}
	&A_1(1,Y):=-\int_Y^{+\infty}A_1(Z)dZ,\qquad A_1(2,Y)=-\int_Y^{+\infty}A_1(1,Z)dZ,\\
	&A_2(1,Y):=\int_{-\infty}^{Y}A_2(Z)dZ,\qquad A_{2}(2,Y)=\int_{-\infty}^{Y}A_2(1,Z)dZ.
\end{align*}
By the properties of Airy function and the above definitions, we can obtain by direct calculations that for $j=1,2$,
\begin{align}\label{eq:Airy-Err-app}
	\varepsilon(\partial_Y^2-\alpha^2)A_j(Y)-(u-c)A_j(Y)=Err_1(Y)A_j+Err_2(Y)\partial_YA_j,
\end{align}
where
\begin{align*}
	Err_1(Y):=u'(Y_c)\eta(\partial_Y\eta)^2-\varepsilon\alpha^2-(u-c)\text{ and } Err_2(Y):=\frac{\varepsilon\partial_Y^2\eta(Y)}{\partial_Y\eta(Y)}.
\end{align*}
We shall see that $Err_1$ and $Err_2$ are, in fact,  small errors.
Moreover, by the property of Airy function, we obtain 
\begin{align}\label{eq:Airy-wron}
	A_1(Y)\partial_Y A_2(Y)-(\partial_Y A_1(Y))A_2(Y)=-\varepsilon^{-1}\partial_Y\eta(Y).
\end{align}
Hence, for given source term $e^{\theta Y}F\in L^\infty$, we define
\begin{align}\label{eq:airy-app-green}
	\begin{split}
	w_{app}(Y):=&A_1(Y)\int_0^YA_2(Z)\partial_Y\eta(Z)^{-1} F(Z)dZ\\
	&+A_2(Y)\int_Y^{+\infty}A_1(Z)\partial_Y\eta(Z)^{-1}F(Z)dZ,\\
	\psi_{app}(Y):=&\int_Y^{+\infty}\int_{Y'}^{+\infty}w_{app}(Z)dZ.
	\end{split}
\end{align}
By direct calculations, we know that $\psi_{app}$ is a solution to the following equation
\begin{align}\label{eq:Airy-app-Fs}
	\begin{split}
		&\varepsilon(\partial_Y^2-\alpha^2)w_{app}-(u-c)w_{app}=F+Err_1w_{app}+Err_2\partial_Yw_{app},\\
		&\partial_Y^2\psi_{app}(Y)=w_{app}(Y),\quad\quad\lim_{Y\to\infty}\psi_{app}(Y)=w_{app}(Y)=0.
	\end{split}
\end{align}
We will show that \eqref{eq:Airy-app-Fs} is an appropriate approximation of \eqref{eq:Airy-eq}.\smallskip

In the  rest of this section, we decompose $\mathbb R_+\cup\{0\}=\mathcal N^-\cup\mathcal N\cup\mathcal N^+,$ where
\begin{align*}
	&\mathcal N^+:=\{Y\geq Y_c:|\kappa\eta(Y)|\geq M\},\quad\mathcal N:=\{Y:|\kappa\eta(Y)|\leq M\},\\
	&\mathcal N^-:=\{0\leq Y\leq Y_c:|\kappa\eta(Y)|\geq M\}.
\end{align*}
We always assume that the structure assumption \eqref{eq:Hyper-u}  holds and $(\al,c)\in\mathbb{H}_2\subset\mathbb{R}_+\times\mathbb{C}$, where
\begin{align*}
\mathbb{H}_2:=\{(\al,c):|\varepsilon|^\f13\leq C\min(\al,c_r),(\al+c_r)|\log|c_i||+(\al^2+c_r^2)|\varepsilon|^{-\f13}\ll1,|c_i|\ll\min\{\al,c_r\}\}.
\end{align*}
It is easy to check that for $(\alpha,c)\in\mathbb H_2$, we always have $-\delta_0<\mathrm{Im}(\kappa\eta(Y))<\delta_0$, where $\delta_0$ is a small constant as defined in Lemma \ref{lem:pri-Airy-decay}.

 \subsection{Estimates of the Airy function}
 
 Since we introduce the modified  Langer transformation involved in the construction of $w_{app}$, we first show some useful estimates for $\eta(Y)$ and the modified Airy function. 
 The proof of the following results is presented  in \cite{MWWZ}.
 
\begin{lemma}[Lemma 4.1 in \cite{MWWZ}]\label{lem:est-eta}
	Let $\alpha,|\varepsilon|,|c|\ll1$ and $c_r>0$. Suppose that $u(Y)$ satisfies the structure assumption \eqref{eq:Hyper-u} and $0<L\leq 1$ is a small number. Then $\eta_r(Y)\in C^{\infty}(\mathbb R_+)$ and $\eta_r(Y)$ satisfies that following properties: 
	
    \begin{enumerate}
    	\item For any $|Y-Y_c|\leq L$, 
    	      \begin{align*}
	           |\eta_r(Y)-(Y-Y_c)|\leq C|Y-Y_c|^2,
              \end{align*}
        	  and for any $|Y-Y_c|\geq L$,
        	  \begin{align*}
	           C^{-1}(1+|Y-Y_c|)^\f23\leq |\eta_r(Y)|\leq C(1+|Y-Y_c|)^\f23.
              \end{align*}
    	\item There exists $0<C_1\leq C_2<+\infty$ such that for any $Y\geq 0$,
    	      \begin{align*}
    	      	C_1(1+|Y-Y_c|)^{-\f13}\leq |\partial_Y \eta_{r}(Y)|\leq C_2(1+|Y-Y_c|)^{-\f13}.
    	      \end{align*}
    	      Moreover, for any $|Y-Y_c|\leq L$,
    	      \begin{align*}
    	      	|\partial_Y \eta_{r}(Y)-1|\leq C|Y-Y_c|.
    	      \end{align*}
    	\item For any $Y\geq 0$,
              \begin{align*}
	           |\partial_Y^2\eta_{r}(Y)|\leq C(1+|Y-Y_c|)^{-\f43}.
              \end{align*}
    	\end{enumerate}
    	Here the constant $C$ is independent of $\eta_0$.
\end{lemma}

\begin{lemma}[Lemma 4.3 in \cite{MWWZ}]\label{lem:err1-err2}
	Let $\alpha,|\varepsilon|,|c|\ll1$ and $c_r>0$, $|c_i|\leq C|\varepsilon|^\f13$. Suppose that $u(Y)$ satisfies the structure assumption \eqref{eq:Hyper-u} and $0<L\leq 1$ is a small number. Then 
\begin{align*}
	&|Err_1(Y)|\leq C|\eta_r(Y)|(|c_i|+|\eta_r(Y)|),\qquad\forall|Y-Y_c|\leq 2\kappa^{-1} M,\\
		&|Err_1(Y)|\leq C|c_i||\eta_r(Y)|,\qquad 2\kappa^{-1} M\leq |Y-Y_c|\leq L,\\
		&|Err_1(Y)|\leq C|c_i|,\qquad|Y-Y_c|\geq L,
\end{align*}
and 
\begin{align*}
	|Err_2(Y)|\leq C|\varepsilon|.
\end{align*}
\end{lemma}

Now we show the asymptotic expansion of the primitives of the modified Airy function. 

\begin{lemma}[Lemma A.4 in \cite{MWWZ}]\label{lem:airy-langer-asy}
Let $\alpha,|\varepsilon|,|c|\ll1$ and $c_r>0$. Suppose that $0<\delta_0\ll1\leq M$ are the constants in Lemma \ref{lem:pri-Airy-decay}. Then there holds that for any $Y\in\mathcal N^-\cup\mathcal N^+$,
      \begin{align*}
 		A_1(1,Y)=&\frac{1}{2\sqrt{\pi}}e^{\mathrm i\frac{\pi}{3}}|\varepsilon|^{-\f13}u'(Y_c)^{-\f23}\partial_Y\eta^{-1}\Big(e^{\mathrm i\frac{\pi}{6}}\kappa\eta(Y) \Big)^{-\f34}e^{-\f23\big(e^{\mathrm i\frac{\pi}{6}}\kappa\eta(Y)\big)^\f32}(1+\mathcal O(|\kappa\eta(Y)|^{-\f32})),\\
 		A_1(2,Y)=&-\frac{1}{2\sqrt{\pi}}u'(Y_c)^{-1}e^{\mathrm i\frac{5\pi}{6}}\partial_Y\eta^{-2}\Big(e^{\mathrm i\frac{\pi}{6}}\kappa\eta(Y) \Big)^{-\f54} e^{-\f23\big(e^{\mathrm i\frac{\pi}{6}}\kappa\eta(Y)\big)^\f32}(1+\mathcal O(|\kappa\eta(Y)|^{-\f32})),\\
 		A_2(1,Y)=&-\sqrt{\pi}e^{-\mathrm i\frac{5\pi}{6}}|\varepsilon|^\f13u'(Y_c)^{-\f13}\partial_Y\eta^{-1}\Big(e^{\mathrm i\frac{5\pi}{6}}\kappa\eta(Y) \Big)^{-\f34} e^{-\f23\big(e^{\mathrm i\frac{5\pi}{6}}\kappa\eta(Y)\big)^\f32}(1+\mathcal O(|\kappa\eta(Y)|^{-\f32})),\\
 		A_2(2,Y)=&\sqrt{\pi}e^{\mathrm i\frac{\pi}{3}}|\varepsilon|^\f23u'(Y_c)^{-\f23}\partial_Y\eta^{-2}\Big(e^{\mathrm i\frac{5\pi}{6}}\kappa\eta(Y) \Big)^{-\f54}e^{-\f23\big(e^{\mathrm i\frac{5\pi}{6}}\kappa\eta(Y)\big)^\f32}(1+\mathcal O(|\kappa\eta(Y)|^{-\f32})),
 	\end{align*}
and  for any $Y\in\mathcal N$,
	      \begin{align*}
	      	&|A_1(1,Y)|\leq C|\varepsilon|^{-\f13},\quad|A_1(2,Y)|\leq C,\\
	      	&|A_2(1,Y)|\leq C|\varepsilon|^\f13,\quad|A_2(2,Y)|\leq C|\varepsilon|^\f23.
	      \end{align*}

\end{lemma}

At the end of this part, we show the estimates for the Green function of the approximate Airy equation \eqref{eq:Airy-app-Fs}.

\begin{lemma}[Lemma 4.4 in \cite{MWWZ}]\label{lem:A1A2}
	Let $0<\delta_0\ll1\leq M$ be the constants as defined in Lemma \ref{lem:pri-Airy-decay}.  Suppose $|\kappa\eta_i|<\delta_0$  and $\alpha,|\varepsilon|, |c|\ll1$. Then for any $0\leq Z\leq Y$, there exists $\gamma_0>0$ such that for $k=0,1,2$ and $j=0,1,2$,
\begin{align*}
		|\partial_Y^kA_1(Y)A_2(j,Z)|
		\leq C|\varepsilon|^{-\frac{2+k-j}{3}}\mathcal M(k,Y)\mathcal M(-j,Z)e^{-W_{\varepsilon}(Y,Z)},\\
		|\partial_Y^kA_2(Z)A_1(j,Y)|
		\leq C|\varepsilon|^{-\frac{2+k-j}{3}}\mathcal M(k,Z)\mathcal M(-j,Y) e^{-W_{\varepsilon}(Y,Z)},
	\end{align*}
	where $\mathcal M(l;Y)=|\partial_Y\eta(Y)|^l|\kappa\eta(Y)|^{-\f14+\f l2}$ for $Y\in\mathcal{N}^-\cup\mathcal{N}^+$, $\mathcal M(l;Y)=M^{-\f14+\f l2}$ for $Y\in\mathcal{N}$ and 
	\begin{align}\label{def:W-e}
	W_{\varepsilon}(Y,Z)=\gamma_0|\varepsilon|^{-\f13}|\eta(Y)-\eta(Z)|(|\kappa\eta(Y)|^\f12+|\kappa\eta(Z)|^\f12).
	\end{align}
\end{lemma}

\subsection{Approximate solution to Airy equation with good force} 
In this part, we show estimates about $\psi_{app}$ and $w_{app}$ defined in \eqref{eq:airy-app-green} satisfying the approximate Airy equation \eqref{eq:Airy-app-Fs} for the source term $F$ corresponding to the upper and main deck. In a similar way as in Proposition 4.5 in \cite{MWWZ}, we can have the following results.
\begin{proposition}\label{prop:airy-gf1}
	Let $1\lesssim\theta\leq 3\eta_0$ and $\hat c:=c+\mathrm i c_0$ with $c_0\in(|\varepsilon|^\f23,|\varepsilon|^\f13)$. Suppose $(\alpha,c)\in\mathbb H_2$ with $c_i>-c_0/2$ and $e^{\theta Y}F\in L^{\infty}$. Then it holds that
	\begin{align*}
	&|e^{\theta Y}(U_s-\hat c)w_{app}(Y)|+|\e|^\f23|e^{\theta Y}\pa_Yw_{app}(Y)|+|\e||e^{\theta Y}\pa_Y^2w_{app}(Y)|
	\leq C\|F\|_{L^\infty_\theta},\\
	&|e^{\theta Y}\partial_Y\psi_{app}(Y)|+|e^{\theta Y}\psi_{app}(Y)|\leq C |\log c_i|\|F\|_{L^\infty_\theta}.
	\end{align*}
	Moreover, we have 
	\begin{align*}
	|e^{\theta Y}Err_1(Y)w_{app}(Y)|+|e^{\theta Y}Err_2(Y)\partial_Yw_{app}(Y)|\leq C|\varepsilon|^\f13 \|F\|_{L^\infty_\theta}.
	\end{align*}	
\end{proposition}

Now we show the estimates about $w_{app}$ with the source term $F\sim(u-\hat c)^{-1}$. 

\begin{proposition}\label{prop:airy-gf2}
	Let $1\lesssim\theta\leq 3\eta_0$ and $\hat c:=c+\mathrm i c_0$ with $c_0\in(|\varepsilon|^\f23,|\varepsilon|^\f13)$. Suppose $(\alpha,c)\in\mathbb H_2$ with $c_i>-c_0/2$ and  $e^{\theta Y}(U_s-\hat c)F\in L^{\infty}$. Then it holds  that
	\begin{align*}
	&|(u-\hat c)^2w_{app}(Y)|+|\e|^\f23|(u-\hat c)\pa_Yw_{app}(Y)|+|\e||(u-\hat c)\pa_Y^2w_{app}(Y)|
	\\
	&\leq C|\log c_i|e^{-\theta Y}\|(u-\hat  c)F\|_{L^\infty_\theta},\\
	&|e^{\theta Y}(u-\hat c)\partial_Y\psi_{app}(Y)|+|e^{\theta Y}\psi_{app}(Y)|\leq C(1+|c||\varepsilon|^{-\f13}) |\log c_i|^2\|(u-\hat c)F\|_{L^\infty_\theta}.
	\end{align*}
	Moreover, we have
	\begin{align*}
	|(u-\hat c)Err_1(Y)w_{app}(Y)|+|(u-\hat  c)Err_2(Y)\partial_Yw_{app}(Y)|
	\leq C|\varepsilon|^\f13 |\log c_i|e^{-\theta Y} \|(u- \hat c)F\|_{L^\infty_\theta}.
	\end{align*}
\end{proposition}

\subsection{Approximate solution to the Airy equation with singular force}
In this part, we construct a solution to 
\begin{align}\label{eq:Airy-app-SF}
	\begin{split}
	&\varepsilon(\partial_Y^2-\alpha^2)w_{app,s}-(u-c)w_{app,s}=\partial_Y^2G(Y) +Err_1w_{app,s}+Err_2\partial_Yw_{app,s},\\
		&\partial_Y^2\psi_{app,s}(Y)=w_{app,s}(Y),\quad\quad\lim_{Y\to\infty}\psi_{app,s}(Y)=w_{app,s}(Y)=0,
	\end{split}
\end{align}
for some $G(Y)$ containing the singularity around the critical layer. By integration by parts, the solution $\psi_{app}(Y)$ to \eqref{eq:Airy-app-SF} can be written as
	\begin{align}\label{eq:psi-app-01}
	\begin{split}
	\psi_{app,s}(Y)=\sum_{j=1}^4\mathcal M_j(Y)+\sum_{j=1}^3\mathcal E_j^{(0)}(Y),
	\end{split}
	\end{align}
	where 

	\begin{align}\label{eq:airy-app-green-remain}
		\begin{split}
			\mathcal M_1(Y):=&-A_1(2,Y)A_2(0)\partial_Y\eta(0)^{-1} \partial_Y G(0)+A_1(2,Y)\partial_Y A_2(0)\partial_Y\eta(0)^{-1} G(0),\\
		\mathcal M_2(Y):=&A_1(2,Y)\int_0^Y\partial_Z^2A_2(\partial_Z\eta)^{-1}G(Z) dZ+A_2(2,Y)\int_Y^{+\infty}\partial_Z^2A_1(\partial_Z\eta)^{-1}G(Z)dZ,\\
		\mathcal M_3(Y):=&\int_Y^{+\infty}(Z-Y)\mathcal H_1(Z)\partial_Z\eta(Z)^{-1}G(Z)dZ,\\
		\mathcal M_4(Y):=&\int_Y^{+\infty}\mathcal (A_1(2,Z)\partial_Z^2A_2(Z)-A_2(2,Z)\partial_Z^2A_1(Z))\partial_Z\eta(Z)^{-1}G(Z)dZ,\\
			\mathcal E_1^{(0)}(Y):=&A_1(2,Y)\int_0^Y\frac{\partial_Z^2\eta(Z)}{\partial_Z\eta(Z)^2}\Big(A_2(Z)\partial_ZG(Z)-\partial_ZA_2(Z)G(Z) \Big)dZ\\
			&+A_2(2,Y)\int_Y^{+\infty}\frac{\partial_Z^2\eta(Z)}{\partial_Z\eta(Z)^2}\Big(A_1(Z)\partial_ZG(Z)-\partial_ZA_1(Z)G(Z) \Big)dZ,\\
			\mathcal E_2^{(0)}(Y):=&\int_Y^{+\infty}(Z-Y)\frac{\partial_Z^2\eta(Z)}{\partial_Z\eta(Z)^2}\mathcal H_2(Z)dZ,\quad\mathcal E_3^{(0)}(Y):=\int_Y^{+\infty}\frac{\partial_Z^2\eta(Z)}{\partial_Z\eta(Z)^2}\mathcal H_3(Z)dZ
		\end{split}
	\end{align}
	with
	\begin{align}\label{eq:airy-app-green-H}
		\begin{split}
			\mathcal H_1(Z):=&\partial_Z A_1(Z)A_2(Z)-\partial_Z A_2(Z)A_1(Z)+A_2(1,Z)\partial_Z^2A_1(Z)-A_1(1,Z)\partial_Z^2A_2(Z),\\
			\mathcal H_2(Z):=&(A_2(1,Z)A_1(Z)-A_1(1,Z)A_2(Z))\partial_Z G(Z)\\
			&-(A_2(1,Z)\partial_Z A_1(Z)-A_1(1,Z)\partial_Z A_2(Z)) G(Z),\\
			\mathcal H_3(Z):=&(A_1(2,Z)A_2(Z)-A_2(2,Z)A_1(Z))\partial_Z G(Z)\\
			&-(A_1(2,Z)\partial_Y A_2(Z)-A_2(2,Z)\partial_Z A_1(Z)) G(Z).
		\end{split}
	\end{align}

	\begin{proposition}\label{prop:airy-green-Y}
Let  $1\lesssim\theta\leq 3\eta_0$ and $\hat c:=c+\mathrm i c_0$ with $c_0\in(|\varepsilon|^\f23,|\varepsilon|^\f13)$.  Suppose $(\alpha,c)\in\mathbb H_2$ with $c_i>-c_0/2$. Then for any $G(Y)\in W^{2,\infty}$ with $\mathrm{supp}(G)\subset[0,\tilde Y]$, $\tilde Y\sim 1$, $\psi_{app,s}(Y)$ defined  by \eqref{eq:psi-app-01} is a solution to \eqref{eq:Airy-app-SF} and satisfies that for any $Y\geq 0$,
\begin{align*}
&|e^{\theta Y}\psi_{app,s}(Y)|
        	+|e^{\theta Y}(u-\hat c)\partial_Y\psi_{app,s}|\\
        	&
        	\leq C|\varepsilon|^{-1}|\hat c||\log|\hat c_i|| \|(u-\hat 
        	c)G\|_{L^\infty}+ C|\varepsilon|^{-\f23}|\log|\hat c_i||\|(u-\hat c)^2\partial_Y G\|_{L^\infty},\\
        	&|e^{\theta Y}(u-\hat c)^2\partial_Y^2\psi_{app,s}(Y)|\leq C|\varepsilon|^{-\f23}|\log|\hat c_i||(\|(u-\hat c)G\|_{L^\infty}+\|(u-\hat c)^2\partial_Y G\|_{L^\infty}),\\
	    	&|e^{\theta Y}(u-\hat c)\partial_Yw_{app,s}(Y)|\leq C(|\varepsilon|^{-1}|\hat c_i|^{-1}+|\varepsilon|^{-\f43})|\log|\hat c_i| |\|(u-\hat c)^2\partial_YG\|_{L^\infty}.
    	        \end{align*} 
  Moreover, for any $Y\in\mathcal N\cup\mathcal N^+$,
  \begin{align*}
  	|e^{\theta Y}\psi_{app,s}(Y)|+|(u-\hat c)\partial_Y\psi_{app,s}(Y)|\leq C|\varepsilon|^{-\f23}|\log|\hat c_i||(\|(u-\hat 
        	c)G\|_{L^\infty}+\|(u-\hat c)^2\partial_Y G\|_{L^\infty}).
  \end{align*}     
\end{proposition}
For any $Y\geq 0$, we define 
\begin{align*}
&\mathcal N_{near}(Y):=\{Z\in\mathcal{N}^+\cup\mathcal{N}^- :|\kappa(\eta(Z)-\eta(Y))|\leq 1\},\\
&\mathcal N_{far}(Y):=\{Z\in\mathcal{N}^+\cup\mathcal{N}^- :|\kappa(\eta(Z)-\eta(Y))|>1\}.
\end{align*}
For $Z, Y\geq 0$, we have
\begin{align*}
e^{\theta|Y-Z|}\leq Ce^{C|\eta_r(Y)^\f32-\eta_r(Z)^\f32|}\leq& Ce^{C|\eta_r(Y)-\eta_r(Z)|(|\eta_r(Y)|^\f12+|\eta_r(Z)|^\f12)}\\
\leq& Ce^{C|\varepsilon|^{\f12} |\varepsilon|^{-\f13}|\eta(Y)-\eta(Z)|(|\kappa\eta(Y)|^\f12+|\kappa\eta(Z)|^\f12)},
\end{align*}
which implies that
\begin{align}\label{eq:ariy-app-gf-decay1}
\begin{split}
W_\varepsilon(Y,Z)e^{\theta|Y-Z|}\leq& e^{-C|\varepsilon|^{-\f13}|\eta(Y)-\eta(Z)|(|\kappa\eta(Y)|^\f12+|\kappa\eta(Z)|^\f12)(1+|\varepsilon|^\f12)}\\
\leq& e^{-\tilde{C}|\varepsilon|^{-\f13}|\eta(Y)-\eta(Z)|(|\kappa\eta(Y)|^\f12+|\kappa\eta(Z)|^\f12)}:=\tilde{W}_\varepsilon(Y,Z).
\end{split}
\end{align}
  By the definition of $\mathcal N^-,\mathcal N$ and $\mathcal N^+$, there exist $Y_-$ and $Y_+$ such that 
	    \begin{align}
	    	\mathcal N^-=[0, Y_-],\quad\mathcal N=[Y_-,Y_+]\text{ and }\mathcal N^+=[Y_+,+\infty).
	    \end{align}

\begin{proof}[Proof of Proposition \ref{prop:airy-green-Y}]
\no\textbf{Estimates about $\psi_{app,s}(Y)$.} \smallskip

\no\underline{Estimates about $\mathcal M_1(Y)$.} By the definition of $\mathcal M_1(Y)$, we first notice that by Lemma \ref{lem:est-eta} and \ref{lem:A1A2}, for any $Y\in\mathcal N^+$, 
	\begin{align}\label{eq:airy-app-SF-L1}
		\begin{split}
	|\mathcal M_1(Y)|
\leq& Ce^{-\gamma_0|\varepsilon|^{-\f13} Y}|\partial_Y\eta(0)|^{-1} |\partial_Y G(0)|+C|\varepsilon|^{-\f13}e^{-\gamma_0|\varepsilon|^{-\f13} Y}|\partial_Y\eta(0)|^{-1}|G(0)|	\\
\leq&C|\hat c|^{-2}e^{-\gamma_0|\varepsilon|^{-\f13} Y}|\partial_Y\eta(0)|^{-1}\|(u-\hat c)^2\partial_Y G\|_{L^\infty}\\
&+C|\varepsilon|^{-\f13}|\hat c|^{-1}e^{-\gamma_0|\varepsilon|^\f13 Y}|\partial_Y\eta(0)|^{-1}\|(u-\hat c)G\|_{L^\infty}\\
\leq&C|\hat c|^{-2}e^{-\theta Y}\|(u-\hat c)^2\partial_YG\|_{L^\infty}+|\hat c|^{-1}|\varepsilon|^{-\f13} e^{-\theta Y}\|(u-\hat c)G\|_{L^\infty}.
		\end{split}
	\end{align}
Again by Lemma \ref{lem:est-eta} and \ref{lem:A1A2}, for any $Y\in\mathcal N^-\cup\mathcal N$, we have 
	\begin{align*}
		|\mathcal M_1(Y)|
		\leq&C|\partial_YG(0)|+C|\varepsilon|^{-\f13}|G(0)|
			\lesssim|\hat c|^{-2}\|(u-\hat c)^2\partial_YG\|_{L^\infty}+|\hat c|^{-1}|\varepsilon|^{-\f13} \|(u-\hat c)G\|_{L^\infty},
	\end{align*}
	which along with \eqref{eq:airy-app-SF-L1} implies that for any $Y\geq 0$,
	\begin{align}\label{eq:airy-app-sf-main1-b}
		\begin{split}
	     |\mathcal M_1(Y)|
		\leq&C|\hat c|^{-2}e^{-\theta Y}\|(u-\hat c)^2\partial_YG\|_{L^\infty}+|\hat c|^{-1}|\varepsilon|^{-\f13} e^{-\theta Y}\|(u-\hat c)G\|_{L^\infty}\\
		\leq& C|\varepsilon|^{-\f23} e^{-\theta Y}(\|(u-\hat c)G\|_{L^\infty}+\|(u-\hat c)^2\partial_YG\|_{L^\infty}).
		\end{split}
	\end{align}
	
	\no\underline{Estimates about $\mathcal M_2(Y)$.} We notice that for any $Y\in\{Y:|\kappa\eta(Y)|\geq M+1\}$,
	    \begin{align}\label{eq:airy-app-sf-main1}
	    	\begin{split}
	    	|\mathcal M_2(Y)
			\leq&C|\varepsilon|^{-\f23}\int_{\mathcal N_{near}(Y)}\f{|\kappa\eta(Z)|^{\f34}|\partial_Z\eta(Z)|}{	|\kappa\eta(Y)|^{\f54}|\partial_Y\eta(Y)|^{2}}
		 e^{-W_{\varepsilon}(Y,Z)}|G(Z)|dZ\\
			&+C|\varepsilon|^{-\f23}\int_{\mathcal N_{far}(Y)}\f{|\kappa\eta(Z)|^{\f34}|\partial_Z\eta(Z)|}{	|\kappa\eta(Y)|^{\f54}|\partial_Y\eta(Y)|^{2}}
			e^{-W_{\varepsilon}(Y,Z)}|G(Z)|dZ\\
			&+C|\varepsilon|^{-\f23}\int_{\mathcal N}|\kappa\eta(Y)|^{-\f54}|\partial_Y\eta(Y)|^{-1} e^{-W_{\varepsilon}(Y,Z)}|G(Z)|dZ=\mathcal{I}_{2,1}+\mathcal{I}_{2,2}+\mathcal{I}_{3,3}.
	    	\end{split}
	    \end{align}
	    By Lemma \ref{lem:est-eta},  for any $Y\in\{Y\geq 0:|\kappa\eta(Y)|\geq M+1, |Y-Y|\leq L\}$ and $Z\in \mathcal N_{near}(Y)$, we have
	    \begin{align*}
	    &\f{|\kappa\eta(Z)|^{\f34}|\partial_Z\eta(Z)|}{	|\kappa\eta(Y)|^{\f54}|\partial_Y\eta(Y)|^{2}|u(Z)-c|}
	    	\leq C |\kappa\eta(Y)|^{-\f12}|\eta(Z)|^{-1}\leq C|\varepsilon|^{-\f13}|\kappa\eta(Y)|^{-\f32}\leq C|\varepsilon|^{-\f13},
	    \end{align*}
	    and for any $|Y-Y_c|\geq L$ and $Z\in \mathcal N_{near}(Y)$,
	    \begin{align*}
	    &\f{|\kappa\eta(Z)|^{\f34}|\partial_Z\eta(Z)|}{	|\kappa\eta(Y)|^{\f54}|\partial_Y\eta(Y)|^{2}|u(Z)-c|}
	    	\leq C|\kappa\eta(Y)|^{-\f12}|\partial_Y\eta(Y)|^{-1}\leq C|\varepsilon|^{\f16}.
	    \end{align*}
	    Therefore, we obtain that for any $Y\in\{Y:|\kappa\eta(Y)|\geq M+1\}$,
	    \begin{align*}
	    	|\kappa\eta(Y)&|^{-\f54}|\partial_Y\eta(Y)|^{-2}|\kappa\eta(Z)|^{\f34}|\partial_Z\eta(Z)||u(Z)-\hat c)|^{-1}\leq C|\varepsilon|^{-\f13}.
	    \end{align*}
	  Then by the above estimates and \eqref{eq:ariy-app-gf-decay1}, we obtain that for any $Y\in\{Y:|\kappa\eta(Y)|\geq M+1\}$,
	    \begin{align}\label{eq:airy-app-sf-main1-1}
	    	\begin{split}
	    		\mathcal{I}_{2,1}\lesssim& |\varepsilon|^{-\f23}e^{-\theta Y}\|e^{\theta_0Y}(u-\hat  c)G\|_{L^\infty}\int_{\mathcal N_{near}(Y)}\f{|\kappa\eta(Z)|^{\f34}|\partial_Z\eta(Z)|}{	|\kappa\eta(Y)|^{\f54}|\partial_Y\eta(Y)|^{2}|u(Z)-\hat c|}e^{-\tilde{W}_\varepsilon(Y,Z)}dZ        \\
			\lesssim& |\varepsilon|^{-1} e^{-\theta Y}\|(u-\hat  c)G\|_{L^\infty}\int_{\mathcal N_{near}(Y)}1dZ\lesssim|\varepsilon|^{-\f23} e^{-\theta Y}\|(u- \hat c)G\|_{L^\infty}.
	    	\end{split}
	    \end{align}
	    By Lemma \ref{lem:est-eta}, for any $Y\in\{Y\geq 0:|\kappa\eta(Y)|\geq M+1\}$ and $Z\in \mathcal N_{far}(Y)$, we have
	    \begin{align*}
	    |\kappa\eta(Z)|^{\f34}|\partial_Z\eta(Z)|e^{-\tilde{W}_\varepsilon(Y,Z)}\leq C|\kappa\eta(Z)|^{\f34}e^{-|\kappa\eta(Z)|^\f12}\leq Ce^{-C|\kappa\eta(Z)|^\f12}.
	    \end{align*}
       Therefore, for $\mathcal{I}_{2,2}$, we have 
        \begin{align}\label{eq:airy-app-sf-main1-2}
        	\begin{split}
        	\mathcal{I}_{2,2}
			&\lesssim|\varepsilon|^{-\f23}e^{-\theta Y}\|e^{\theta_0Y}(u-\hat c)G\|_{L^\infty}\int_{\mathcal N_{far}(Y)}\f{|\kappa\eta(Z)|^{\f34}|\partial_Z\eta(Z)|}{	|\kappa\eta(Y)|^{\f54}|\partial_Y\eta(Y)|^{2}|u(Z)-\hat c|}e^{-\tilde{W}_\varepsilon(Y,Z)}dZ\\
			&\lesssim|\varepsilon|^{-1} e^{-\theta  Y}\|(u- \hat c)G\|_{L^\infty}\int_{\mathcal N_{far}(Y)}e^{-|\kappa\eta(Z)|^\f12}|\partial_Z\eta(Z)|dZ\\
			&\lesssim |\varepsilon|^{-\f23}e^{-\theta Y}\|(u- \hat c)G\|_{L^\infty}.
        	\end{split}
        \end{align}
	    For $\mathcal{I}_{2,3}$, we notice that
	    \begin{align*}
	    	\mathcal{I}_{2,3}
	    	\leq&C|\varepsilon|^{-\f23}\|(u-\hat c)G\|_{L^\infty}\int_{Z\in\mathcal N} |u(Z)-\hat c|^{-1}dZ
	    	\leq C|\varepsilon|^{-\f23}\|(u-\hat c)G\|_{L^\infty}|\log|\hat c_i||,
	    \end{align*}
	    which along with \eqref{eq:airy-app-sf-main1}, \eqref{eq:airy-app-sf-main1-1} and \eqref{eq:airy-app-sf-main1-2} deduces that for any $Y\in\{Y:|\kappa\eta(Y)|\geq M+1\}$,
	    \begin{align*}
	    	   |\mathcal M_2(Y)| \leq Ce^{-\theta Y}\|(u- \hat c)G\|_{L^\infty}|\varepsilon|^{-\f23} |\log|\hat c_i||. 
	    \end{align*}
	    
	    Now we consider the estimates of $\mathcal M_2(Y)$ for the case of $|\kappa\eta(Y)|\leq M+1$. By Lemma \ref{lem:A1A2}, we have that for any $|\kappa\eta(Y)|\leq M+1$, 
	   \begin{align*}
	   		|\mathcal M_2(Y)|\lesssim &|\varepsilon|^{-\f23}\int_{|\kappa\eta(Z)|\geq M+2}|\kappa\eta(Z)|^\f34|\partial_Z\eta(Z)|e^{-\tilde{W}_\varepsilon(Y,Z)}|G|dZ+|\varepsilon|^{-\f23}\int_{|\kappa\eta(Z)|\leq M+2}|G(Z)|dZ\\
	    		\lesssim&|\varepsilon|^{-1}\|(u- \hat c)G\|_{L^\infty}\int_{|\kappa\eta(Z)|\geq M+2}|\kappa\eta(Z)|^{-\f14}e^{-C|\kappa\eta(Z)|^\f12}dZ\\
	    		&+|\varepsilon|^{-\f23}\|(u- \hat c)G\|_{L^\infty}\int_{|\kappa\eta(Z)|\leq M+2}|u(Z)-\hat c|^{-1}dZ\\
	    		\lesssim&\|(u- \hat c)G\|_{L^\infty}|\varepsilon|^{-\f23} |\log|\hat c_i||.
	   \end{align*}
	   Therefore, we obtain that for any $Y\geq 0$,
	   \begin{align}\label{eq:airy-app-sf-main2-b}
	   	  |\mathcal M_2(Y)| \leq Ce^{-\theta Y}\|(u- \hat c)G\|_{L^\infty}|\varepsilon|^{-\f23} |\log|\hat c_i||. 
	   \end{align}
	   
	   \no\underline{Estimates about $\mathcal M_3(Y)$.} Notice that for any $Y\geq 0$,
	   \begin{align*}
	   	\mathcal H_1(Z)=\partial_Z\Big(\underbrace{\partial_ZA_1(Z)A_2(1,Z)-A_1(1,Z)\partial_ZA_2(Z)}_{\tilde{\mathcal{H}}(Z)}\Big).
	   \end{align*}
	    Hence, we have that for any $Y\geq 0$,
	    \begin{align}\label{eq:airy-app-sf-main3}
	    	\begin{split}
	    		\mathcal M_3(Y)
	    		&=-\int_{Y}^{+\infty}(Z-Y)\tilde{\mathcal{H}}(Z)\partial_Z\eta(Z)^{-1}\partial_ZG(Z)dZ-\int_{Y}^{+\infty}\tilde{\mathcal{H}}(Z)\partial_Z\eta(Z)^{-1}G(Z)dZ\\
	    		&\quad+\int_{Y}^{+\infty}(Z-Y)\tilde{\mathcal{H}}(Z)\frac{\partial_Z^2\eta(Z)}{\partial_Z\eta(Z)^2} G(Z)dZ=\mathcal{I}_{3,1}+\mathcal{I}_{3,2}+\mathcal{I}_{3,3}.
	    	\end{split}
	    \end{align}
	    We first notice that by Lemma \ref{lem:A1A2}, for any $Y\geq Y_c+L$, 
	    \begin{align*}
	    |\mathcal{I}_{3,1}(Y)|
	    	\leq&C|\varepsilon|^{-\f23}\int_{Y}^{+\infty}(Z-Y)|\kappa\eta(Z)|^{-\f12}|\partial_Z G(Z)|dZ
	    	\leq C|\varepsilon|^{-\f12}e^{-\theta Y}\|(u- \hat c)^2\partial_Y G\|_{L^\infty}.
	    \end{align*}
	    Again by Lemma \ref{lem:A1A2}, for any $Y\in\mathcal N^+\cap\{Y:|Y-Y_c|\leq L\}$,
	    \begin{align}\label{eq:airy-app-sf-main3-1}
	    	\begin{split}
	    		|\mathcal{I}_{3,1}(Y)|
	    	\leq&|\mathcal{I}_{3,1}(Y_c+L)|+\Big|(Y_c+L-Y)\int_{Y_c+L}^{+\infty}|\kappa\eta(Z)|^\f12|\partial_ZG|dZ\Big| \\
	    	&+\Big|\int_{Y}^{Y_c+L}(Z-Y)\tilde{\mathcal{H}}(Z)\partial_Z\eta(Z)^{-1}\partial_ZG(Z)dZ\Big|\\
	    	\leq&C|\varepsilon|^{-\f12}\|(u- \hat c)^2\partial_Y G\|_{L^\infty}+C|\varepsilon|^{-\f23}\int_{Y}^{Y_c+L}|Z-Y||\kappa\eta(Z)|^{-\f12}|\partial_ZG(Z)|dZ.
	    	\end{split}
	    \end{align}
	    Moreover, for any $Y\in\mathcal N^+\cap\{Y:|Y-Y_c|\leq L\}$,
	    \begin{align*}
	    \int_{Y}^{Y_c+L}&|Z-Y||\kappa\eta(Z)|^{-\f12}|\partial_ZG(Z)|dZ\\
	    \leq&\|(u-\hat c)^2\partial_Y G\|_{L^\infty}\int_Y^{Y_c+L}|Z-Y_c||\kappa\eta(Z)|^{-\f12}|u(Z)-\hat c|^{-2}dZ\\
	    \leq&C|\varepsilon|^{-\f13}\|(u-\hat c)^2\partial_Y G\|_{L^\infty}\int_Y^{Y_c+L}|\kappa\eta(Z)|^{-\f32}dZ\\
	    \leq&C\|(u-\hat c)^2\partial_Y G\|_{L^\infty},
	    \end{align*}
	   which along with \eqref{eq:airy-app-sf-main3-1} implies that for any $Y\in\mathcal N^+$,
	    \begin{align}\label{eq:airy-app-sf-main3-2}
	    	\begin{split}
	    		\Big|\mathcal{I}_{3,1}(Y)\Big|
	    	\leq C|\varepsilon|^{-\f23} e^{-\theta  Y}\|(u-\hat c)^2\partial_Y G\|_{L^\infty}.
	    	\end{split}
	    \end{align}
	     By a similar argument as above, we  can also deduce that for any $Y\in\mathcal N^+$,
        \begin{align}\label{eq:airy-app-sf-main3-5}
        	\begin{split}
        		\Big|\mathcal{I}_{3,2}(Y)\Big|
        		\leq C|\varepsilon|^{-\f23}|\log|\hat c_i||e^{-\theta  Y}\|(u-\hat c) G\|_{L^\infty}.
        	\end{split}
        \end{align}
        About $\mathcal{I}_{3,3}$, we get by Lemma \ref{lem:est-eta} that 
        \begin{align}\label{eq:langer-d2-d1}
        	\left|\frac{\partial_Z^2\eta(Z)}{\partial_Z\eta(Z)^2}\right|\leq C.
        \end{align}
        Hence, again by performing a similar argument in the estimates for $\mathcal I_{3,1}$, we have that for any $Y\in\mathcal N^+$,
        \begin{align*}
        	\Big|\mathcal{I}_{3,3}(Y)\Big|
            \leq C|\varepsilon|^{-\f13}(|\log|\hat c_i||+|c||\varepsilon|^{-\f13})e^{-\theta  Y}\|(u-\hat c) G\|_{L^\infty},
        \end{align*}
        which along with \eqref{eq:airy-app-sf-main3}, \eqref{eq:airy-app-sf-main3-2} and \eqref{eq:airy-app-sf-main3-5} implies that for any $Y\in\mathcal N^+$,
        \begin{align}\label{eq:airy-app-sf-main3-bYL}
        	|\mathcal M_3(Y)|\leq C|\varepsilon|^{-\f23}|\log|\hat c_i||e^{-\theta Y} (\|(u-\hat 
        	c)G\|_{L^\infty}+ \|(u-\hat c)^2\partial_Y G\|_{L^\infty}).
        \end{align}
       For any $Y\in\mathcal N$, we can write that 
	    \begin{align*}
	    	\mathcal M_3(Y)=&\mathcal M_3(Y_+)+(Y_+-Y)\int_{Y_+}^{+\infty}\mathcal H_1(Z)\partial_Z\eta(Z)^{-1}G(Z)dZ\\
	    	&+\int_Y^{Y_+}(Z-Y)\mathcal H_1(Z)\partial_Z\eta(Z)^{-1}G(Z)dZ=\mathcal M_3(Y_+)+\mathcal I_{3,4}^++\mathcal I_{3,5}^+.
	    \end{align*}
	    By the definition of $\mathcal H_1$ and Lemma \ref{lem:A1A2}, we infer that for any $Y\in\mathcal N$,
	    \begin{align*}
	    	|\mathcal I_{3,4}^+(Y)|\leq C|\varepsilon|^{-\f23}\|(u-\hat c)G\|_{L^\infty}\int_{Y_+}^{1}|u(Z)-\hat c|^{-1}dZ\leq C|\varepsilon|^{-\f23}\|(u-\hat c)G\|_{L^\infty}|\log|\hat c_i||,
	    \end{align*}
	    and 
	    \begin{align*}
	    	|\mathcal I_{3,5}^+(Y)|\leq& C|\varepsilon|^{-\f23}\|(u-\hat c)G\|_{L^\infty}\int_Y^{Y_+}|u(Z)-\hat c|^{-1}dZ
	    	\leq C|\varepsilon|^{-\f23}\|(u-\hat c)G\|_{L^\infty}|\log|\hat c_i||.
	    \end{align*}
	    Therefore, we obtain that for any $Y\in\mathcal N$,
	    \begin{align}\label{eq:airy-app-sf-main3-bYc}
        	|\mathcal M_3(Y)|\leq C|\varepsilon|^{-\f23}|\log|\hat c_i|| (\|(u-\hat 
        	c)G\|_{L^\infty}+ \|(u-\hat c)^2\partial_Y G\|_{L^\infty}).
        \end{align}
        For any $Y\in\mathcal N^-$, notice that 
         \begin{align*}
	    	\mathcal M_3(Y)=&\underbrace{(Y_c-Y)\int_{Y}^{+\infty}\mathcal H_1(Z)\partial_Z\eta(Z)^{-1}G(Z)dZ+\int_Y^{Y_-}(Z-Y_c)\mathcal H_1(Z)\partial_Z\eta(Z)^{-1}G(Z)dZ}_{\mathcal I_{3,4}^-}\\
	    	&+\underbrace{(Y_--Y_c)\int_Y^{+\infty}\mathcal H_1\partial_Z\eta^{-1}G(Z)dZ}_{\mathcal I_{3,5}^-}+\mathcal M_3(Y_-).
	    \end{align*}
	    By the definition of $\mathcal H_1$ and Lemma \ref{lem:A1A2}, we infer that for any $Y\in\mathcal N^-$
	    \begin{align*}
	    	|\mathcal I_{3,4}^-(Y)|\leq& C|Y-Y_c||\varepsilon|^{-1}\|(u-\hat c)G\|_{L^\infty}\int_Y^{Y_-}|u(Z)-\hat c|^{-1}dZ\\
	    	&+C|\varepsilon|^{-1}\|(u-\hat c)G\|_{L^\infty}\int_Y^{Y_-}1dZ \\
	    	\leq&C|Y-Y_c||\varepsilon|^{-1}\|(u-\hat c)G\|_{L^\infty}|\log\hat c_i|,
	    \end{align*} 
	    and 
	    \begin{align*}
	    	|\mathcal I_{3,5}^-(Y)|\leq& C|\varepsilon|^{-\f23}\|(u-\hat c)G\|_{L^\infty}\int_{Y_-}^{1}|u(Z)-\hat c|^{-1}dZ\\
	    	\leq& C|\varepsilon|^{-\f23}|\log|\hat c_i||\|(u-\hat c)G\|_{L^\infty}.
	    \end{align*}
	  	 Therefore, we conclude that for any $Y\in\mathcal N^-$,
	     \begin{align}\label{eq:airy-app-sf-main3-bYS}
        \begin{split}
        	|\mathcal M_3(Y)|
        	&\leq C|\varepsilon|^{-1}|Y-Y_c||\log|\hat c_i||\|(u-\hat 
        	c)G\|_{L^\infty}+ C|\varepsilon|^{-\f23}|\log|\hat c_i||\|(u-\hat c)^2\partial_Y G\|_{L^\infty}.
        \end{split}
        \end{align}
        By \eqref{eq:airy-app-sf-main3-bYL}, \eqref{eq:airy-app-sf-main3-bYc} and \eqref{eq:airy-app-sf-main3-bYS}, we infer that for any $Y\geq 0$,
	  \begin{align}\label{eq:airy-app-sf-main3-b}
        \begin{split}
        		|e^{\theta Y}\mathcal M_3(Y)|\leq& C|\varepsilon|^{-1}|\log\hat c_i|(|\varepsilon|^\f13+(Y-Y_c)\mathbf 1_{\mathcal N^-})\|(u-\hat 
        	c)G\|_{L^\infty}\\
        	&+C|\varepsilon|^{-\f23}|\log \hat c_i| \|(u-\hat c)^2\partial_Y G\|_{L^\infty}.
        \end{split}
        \end{align}
        
        \no\underline{Estimates about $\mathcal M_4(Y)$.}  We can obtain the following bounds for $\mathcal M_4(Y)$ by a similar argument as above
        \begin{align}\label{eq:airy-app-sf-main4-b}
        	|e^{\theta Y}\mathcal M_4(Y)|\leq C|\varepsilon|^{-\f23}|\log|\hat c_i||\|(u-\hat 
        	c)G\|_{L^\infty}.
        \end{align}
        
        \no\underline{Estimates about $\mathcal E^{(0)}_j(Y)$ for $j=1,2,3$.} By Lemma \ref{lem:A1A2} and \eqref{eq:langer-d2-d1}, we obtain that for any $Y\geq 0$,
        \begin{align}\label{eq:airy-app-sf-E-b}
        	\begin{split}
        		|e^{\theta Y}\mathcal E^{(0)}_{1}(Y)|\leq& C|\varepsilon|^{-\f13}|\log|\hat c_i||\|(u-\hat c)G\|_{L^\infty}+C|\hat c_i|^{-1}\|(u-\hat c)^2\partial_Y G\|_{L^\infty},\\
        		|e^{\theta Y}\mathcal E^{(0)}_{2}(Y)|\leq&C|\varepsilon|^{-\f13}(|\log|\hat c_i||+|c||\varepsilon|^{-\f13})(\|(u-\hat 
        	c)G\|_{L^\infty}+ \|(u-\hat c)^2\partial_Y G\|_{L^\infty}),\\
        	   |e^{\theta Y}\mathcal E^{(0)}_{3}(Y)|\leq&C|\varepsilon|^{-\f23}|\log|\hat c_i||\|(u-\hat 
        	c)G\|_{L^\infty}.
        	\end{split}
        \end{align}
        Therefore, by \eqref{eq:airy-app-sf-main1-b}, \eqref{eq:airy-app-sf-main2-b}, \eqref{eq:airy-app-sf-main3-b}, \eqref{eq:airy-app-sf-main4-b} and \eqref{eq:airy-app-sf-E-b}, we deduce that for any $Y\geq 0$,
        \begin{align*}
        	|e^{\theta Y}\psi_{app,s}(Y)|\leq& C|\varepsilon|^{-1}|\log\hat c_i|(|\varepsilon|^\f13+(Y-Y_c)\mathbf 1_{\mathcal N^-})\|(u-\hat 
        	c)G\|_{L^\infty}\\
        	&+C|\varepsilon|^{-\f23}|\log \hat c_i| \|(u-\hat c)^2\partial_Y G\|_{L^\infty}.
        \end{align*}      
        \no\textbf{Estimates about $\partial_Y\psi_{app,s}(Y)$.} By a direct calculation, we write 
        \begin{align}\label{eq:airy-app-sf-green1}
		\begin{split}
			\partial_Y\psi_{app,s}(Y)=\sum_{j=1}^3\mathcal M_{j}^{(1)}(Y)+\mathcal E^{(1)}_1(Y)+\mathcal E^{(1)}_2(Y),
		\end{split}
	\end{align}
	where
	\begin{align}
		\begin{split}
		   \mathcal M_1^{(1)}(Y):=&-A_1(1,Y)A_2(0)\partial_Y\eta(0)^{-1} \partial_Y G(0)+A_1(1,Y)\partial_Y A_2(0)\partial_Y\eta(0)^{-1} G(0),\\
		   \mathcal M_2^{(1)}(Y):=&A_1(1,Y)\int_0^Y\partial_Z^2A_2(\partial_Z\eta)^{-1}G(Z) dZ
			+A_2(1,Y)\int_Y^{+\infty}\partial_Z^2A_1(\partial_Z\eta)^{-1}G(Z)dZ,\\
			\mathcal M_3^{(1)}(Y):=&-\int_Y^{+\infty}\mathcal H_1(Z)\partial_Z\eta(Z)^{-1}G(Z)dZ,  \\
			\mathcal E_1^{(1)}(Y):=&A_1(1,Y)\int_0^Y\frac{\partial_Z^2\eta(Z)}{\partial_Z\eta(Z)^2}(A_2(Z)\partial_ZG(Z)-\partial_ZA_2(Z)G(Z))dZ\\
			&+A_2(1,Y)\int_Y^{+\infty}\frac{\partial_Z^2\eta(Z)}{\partial_Z\eta(Z)^2}(A_1(Z)\partial_ZG(Z)-\partial_ZA_1(Z)G(Z))dZ,\\
			\mathcal E_2^{(1)}(Y):=&\int_Y^{+\infty}\frac{\partial_Z^2\eta(Z)}{\partial_Z\eta(Z)^2}\mathcal H_2(Z)dZ.
		\end{split}
	\end{align}
  \underline{Estimates about $\mathcal M^{(1)}_1(Y)$.} According to Lemma \ref{lem:est-eta}, we have for any $Y\geq 0$,
  \begin{align*}
  	|(u(Y)-\hat c)|\lesssim|\varepsilon|^{\f13}|\kappa\eta(Y)|.
  \end{align*}
  Hence, by a similar argument as in the estimates \eqref{eq:airy-app-SF-L1} and \eqref{eq:airy-app-sf-main1-b},
     \begin{align}\label{eq:eq:airy-app-sf-maind1-e}
	    	\begin{split}
	    		|(u-\hat c)\mathcal M_1^{(1)}(Y)|	 	\leq&C|\varepsilon|^{-\f23} e^{-\theta Y}(\|(u-\hat c)G\|_{L^\infty}+\|(u-\hat c)^2\partial_YG\|_{L^\infty}).
	    	\end{split}
	    \end{align}
\no\underline{Estimates about $\mathcal M_2^{(1)}(Y)$.} By Lemma \ref{lem:A1A2}, we notice that for any $Y$ such that $|\kappa\eta(Y)|\geq M+1$,
	    \begin{align}\label{eq:airy-app-sf-main2d}
	    	\begin{split}
	    		|(u-\hat c)\mathcal M_2^{(1)}(Y)|
	    		\leq&C|\varepsilon|^{-1} \int_{\mathcal N_{near}(Y)}\frac{|u(Y)-\hat c||\eta(Z)|^\f34|\partial_Z\eta(Z)|}{|\eta(Y)|^\f34|\partial_Y\eta(Y)|}e^{-W_\varepsilon(Y,Z)}|G(Z)|dZ\\
	    		&+C|\varepsilon|^{-1} \int_{\mathcal N_{far}(Y)}\frac{|u(Y)-\hat c||\eta(Z)|^\f34|\partial_Z\eta(Z)|}{|\eta(Y)|^\f34|\partial_Y\eta(Y)|}e^{-W_\varepsilon(Y,Z)}|G(Z)|dZ\\
	    		&+C|\varepsilon|^{-1}\int_{Z\in \mathcal N}|u(Y)-\hat c||\kappa\eta(Y)|^{-\f34}|\partial_Y\eta(Y)|^{-1}e^{-W_\varepsilon(Y,Z)}|G(Z)|dZ.	    	
	    \end{split}
	    \end{align}
	    By Lemma \ref{lem:est-eta}, we can obtain that $Y\in\{Y:|\kappa\eta(Y)|\geq M+1\},$
	    \begin{align*}
	    	\frac{|\eta(Z)|^\f34|\partial_Z\eta(Z)|}{|\eta(Y)|^\f34|\partial_Y\eta(Y)|}\lesssim\left\{
	    	\begin{aligned}
	    		&1,\quad Z\in\mathcal N_{near}(Y),\\
	    		&|\kappa\eta(Z)|^{\f34},\quad Z\in\mathcal N_{far}(Y),
	    	\end{aligned}
	    	\right.
	    \end{align*}
	    and $|\kappa\eta(Y)|^{-\f34}|\partial_Y\eta(Y)|^{-1}\lesssim 1$. Hence, by a similar argument as in the estimates about $\mathcal M_2(Y)$, we can infer that  for any $|\kappa\eta(Y)|\geq M+1$,
	    \begin{align}\label{eq:airy-app-sf-main2d1}
	    	|(u(Y)-\hat c)\mathcal M_2^{(1)}(Y)|\leq C|\varepsilon|^{-\f23}|\log|\hat c_i|| e^{-\theta Y} \|(u-\hat c)G\|_{L^\infty}.
	    \end{align}
    For any $|\kappa\eta(Y)|\leq M+1$, notice that 
	    \begin{align*}
	    		&|(u(Y)-\hat c)\mathcal M_2^{(1)}(Y)|\\
	    		\leq&C|\varepsilon|^{-\f23} \int_{\mathcal N_{near}(Y)}|\kappa\eta(Z)|^{\f34}|\partial_Z\eta(Z)|e^{-W_\varepsilon(Y,Z)}|G(Z)|dZ\\
	    			 &+C|\varepsilon|^{-\f23} \int_{\mathcal N_{far}(Y)}|\kappa\eta(Z)|^{\f34}|\partial_Z\eta(Z)|e^{-W_\varepsilon(Y,Z)}|G(Z)|dZ+C|\varepsilon|^{-\f23}\int_{ \mathcal N}|G(Z)|dZ\\
	    		\leq&C|\varepsilon|^{-\f23}\|(u-\hat c)G\|_{L^\infty}\int_0^{\tilde Y}|u(Z)-\hat c|^{-1}dZ\leq C|\varepsilon|^{-\f23}|\log|\hat c_i||  \|(u-\hat c)G\|_{L^\infty},
	    \end{align*}
	    which along with \eqref{eq:airy-app-sf-main2d1} implies that for any $Y\geq 0$,
	    \begin{align}\label{eq:eq:airy-app-sf-maind2-e}
	    	|(u-\hat c)\mathcal M_2^{(1)}(Y)|\leq C|\varepsilon|^{-\f23}|\log|\hat c_i|| e^{-\theta Y} \|(u-\hat c)G\|_{L^\infty}.
	    \end{align}
	    
	    \no\underline{Estimates about $\mathcal M_3^{(1)}(Y)$.}
	    By the definition of $\mathcal M_3^{(1)}(Y)$, we have that for any $Y\geq 0$,
	    \begin{align*}
	    	\mathcal M_3^{(1)}(Y)=&-\int_Y^{+\infty}\partial_Z\tilde{\mathcal H}(Z)\partial_Z\eta(Z)^{-1}G(Z)dZ\\
	    	=&\tilde{\mathcal H}(Y)\partial_Y\eta(Y)^{-1}G(Y)+\int_Y^{+\infty}\tilde{\mathcal H}(Z)\partial_Z(\partial_Z\eta(Z)^{-1}G(Z))dZ,
	    \end{align*}
	    which implies that 
	    \begin{align}\label{eq:airy-app-sf-main3d}
	    	\begin{split}
	    		&|(u-\hat c)\mathcal M_3^{(1)}(Y)|\\
	    		&\leq |\tilde{\mathcal H}(Y)\partial_Y\eta(Y)^{-1}(u-\hat c)G(Y)|+\int_{Y}^{+\infty}|u(Y)-\hat c||\tilde{\mathcal H}(Z)||\partial_Z\eta(Z)^{-1}\partial_ZG(Z)|dZ\\
	    		&\quad+\int_{Y}^{+\infty}|u(Y)-\hat c|\tilde{\mathcal H}(Z)|\frac{\partial_Z^2\eta(Z)}{\partial_Z\eta(Z)^2} G(Z)|dZ.
	    	\end{split}
	    \end{align}
	    By Lemma \ref{lem:A1A2}, we can obtain that for any $Y\geq 0$,
	    \begin{align*}
	    	|\tilde{\mathcal H}(Y)\partial_Y\eta(Y)^{-1}(u-\hat c)G(Y)|\leq C|\varepsilon|^{-\f23}e^{-\theta Y}\|(u-\hat c)G\|_{L^\infty}.
	    \end{align*}
	    Moreover, by a similar arguments as in the estimates about $\mathcal M_3$, we can deduce that for $Y\in\mathcal N\cup\mathcal N^+$,
	     \begin{align}
	    	|(u-\hat c)\mathcal M_3^{(1)}(Y)|\leq C|\varepsilon|^{-\f23}|\log|\hat c_i| e^{-\theta Y}(\|(u-\hat c)G\|_{L^\infty}+ \|(u-\hat c)^2\partial_Y G\|_{L^\infty}).
	    \end{align}
	   
	    For any $Y\in\mathcal N^-$, we notice that 
	     \begin{align*}
	     	|(u-\hat c)\mathcal M_3^{(1)}(Y)|\leq&C|\varepsilon|^{-\f23}|\log|\hat c_i||e^{-\theta Y}(\|(u-\hat c)G\|_{L^\infty}+ \|(u-\hat c)^2\partial_Y G\|_{L^\infty})\\
	    	&+\int_Y^{Y_-}|u(Y)-\hat c||\mathcal H_1(Z)|G(Z)|dZ,
	     \end{align*}
	     which along with the fact that 
	     \begin{align*}
	     	\int_Y^{Y_-}|u(Y)-\hat c||\mathcal H_1(Z)|G(Z)|dZ\leq& C|\varepsilon|^{-1}|\hat c|\|(u-\hat c)G\|_{L^\infty}\int_Y^{Y^-}|u(Z)-\hat c|^{-1}dZ\\
	     	\leq&C|\varepsilon|^{-1}|\hat c|\log|\hat c_i|| \|(u-\hat c) G\|_{L^\infty},
	     \end{align*}
	     implies that for any $Y\in\mathcal N^-$,
	     \begin{align*}
	     	|(u-\hat c)\mathcal M_3^{(1)}(Y)|\leq&  C|\varepsilon|^{-1}|\hat c||\log|\hat c_i|| \|(u-\hat 
        	c)G\|_{L^\infty}+ C|\varepsilon|^{-\f23}|\log|\hat c_i||\|(u-\hat c)^2\partial_Y G\|_{L^\infty}.
	     \end{align*}
	     Therefore,  we obtain that for any $Y\geq 0$,
	     \begin{align}\label{eq:eq:airy-app-sf-maind3-e}
	     	  \begin{split}
	     	  		&|(u-\hat c)e^{-\theta Y}\mathcal M_3^{(1)}(Y)|\\
	     	  		&\lesssim   |\varepsilon|^{-1}|\hat c||\log|\hat c_i|| \|(u-\hat 
        	c)G\|_{L^\infty}+ |\varepsilon|^{-\f23}|\log|\hat c_i||\|(u-\hat c)^2\partial_Y G\|_{L^\infty}.
	     	  \end{split}
	     \end{align}
	     By \eqref{eq:langer-d2-d1} and similar calculations as above, we can obtain that for any $Y\geq 0$
	     \begin{align}\label{eq:airy-app-sf-erd-e}
	     	\begin{split}
	     		\sum_{j=1}^2|(u-\hat c)\mathcal E^{(1)}_j(Y)|\leq C|\varepsilon|^{-\f13}(|\log|\hat c_i||+|c||\varepsilon|^{-\f13})(\|(u-\hat 
        	c)G\|_{L^\infty}+ \|(u-\hat c)^2\partial_Y G\|_{L^\infty}).
	     	\end{split}
	     \end{align}
	     Therefore, by \eqref{eq:eq:airy-app-sf-maind1-e}, \eqref{eq:eq:airy-app-sf-maind2-e}, \eqref{eq:eq:airy-app-sf-maind3-e} and \eqref{eq:airy-app-sf-erd-e}, we conclude that for any $Y\geq 0$,
	      \begin{align*}
        	|(u-\hat c)e^{-\theta Y}\partial_Y\psi_{app,s}(Y)|
        	\leq& C|\varepsilon|^{-1}|\hat c||\log|\hat c_i|| \|(u-\hat 
        	c)G\|_{L^\infty}+ C|\varepsilon|^{-\f23}|\log|\hat c_i||\|(u-\hat c)^2\partial_Y G\|_{L^\infty}.
        \end{align*}
        
	     \no\textbf{Estimates about $\partial_Y^2\psi_{app,s}(Y)$.} By integration by parts, we have that for any $Y\in\{Y:|\kappa\eta(Y)|\geq M+1,Y\geq Y_c\}$,
	     	     \begin{align}\label{eq:airy-app-sf-dd-e}
	     	\partial_Y^2\psi_{app,s}(Y)=&\mathcal M_{1,1}^{(2)}(Y)+\mathcal M_{2,1}^{(2)}(Y)+\mathcal E_1^{(2)}(Y),
	     \end{align}
	     where 
	     \begin{align*}
	     	\mathcal M_{1,1}^{(2)}(Y)=&-A_1(Y)A_2(0)\partial_Y\eta(0)^{-1}\partial_YG(0)+A_1(Y)\partial_YA_2(Y_-)\partial_Y\eta(Y_-)^{-1}G(Y_-)\\
	     	&-A_1(Y)\partial_YA_2(Y_+)\partial_Y\eta(Y_+)^{-1}G(Y_+),\\
	     	\mathcal M_{2,1}^{(2)}(Y)=&A_1(Y)\int_{Z\in\mathcal N}\partial_Z^2A_2(Z)\partial_Z\eta(Z)^{-1}G(Z)dZ\\
	     	&-A_1(Y)\left(\int_0^{Y_-}+\int_{Y_+}^{Y}\right)\partial_ZA_2(Z)\partial_Z\eta(Z)^{-1}\partial_ZG(Z)dZ\\
	     	&-A_2(Y)\int_Y^{+\infty}\partial_ZA_1(Z)\partial_Z\eta(Z)^{-1}\partial_ZG(Z)dZ,\\
	     	\mathcal E_1^{(2)}(Y)=&A_1(Y)\int_0^YA_2(Z)\frac{\partial_Z^2\eta(Z)}{\partial_Z\eta(Z)^2} \partial_ZG(Z)dZ-A_1(Y)\int_{\mathcal N}\partial_ZA_2(Z)\frac{\partial_Z^2\eta(Z)}{\partial_Z\eta(Z)^2} G(Z)dZ\\
	  &+A_2(Y)\int_Y^{+\infty}A_1(Z)\frac{\partial_Z^2\eta(Z)}{\partial_Z\eta(Z)^2} \partial_ZG(Z)dZ.
	     \end{align*}
	      \no\underline{Estimates about $\mathcal M_{1,1}^{(2)}(Y)$.} By Lemma \ref{lem:A1A2} and  a similar argument as in the calculation of $\mathcal M_1(Y)$, we can obtain that for any $Y\in\{Y:|\kappa\eta(Y)|\geq M+1,Y\geq Y_c\}$,
	   	    \begin{align}\label{eq:airy-app-sf-maindd1-e}
	    	\begin{split}
	    	|(u-\hat c)^2\mathcal M_{1,1}^{(2)}(Y)|\leq& C|\varepsilon|^{-\f23} e^{-\theta Y}(\|(u-\hat c)^2\partial_Y G\|_{L^\infty}+\|(u-\hat c) G\|_{L^\infty})
	  	    	\end{split}
	    \end{align}
	    
	    \no\underline{Estimates about $\mathcal M_{2,1}^{(2)}(Y)$.}
	    For any $Y\in\{Y:|\kappa\eta(Y)|\geq M+1,Y\geq Y_c\}$, we have 
	    \begin{align}\label{eq:airy-app-sf-main2dd}
	    	\begin{split}
	    		&|(u-\hat c)^2\mathcal M_{2,1}^{(2)}(Y)|\\
	    		&\leq C|\varepsilon|^{-1} \int_{\mathcal N_{near}(Y)}|u(Y)-\hat c|^2 |\kappa\eta(Y)|^{-\f14}|\kappa\eta(Z)|^{\f14}e^{-W_\varepsilon(Y,Z)}|\partial_Z G(Z)|dZ\\
	    		&\quad+C|\varepsilon|^{-1} \int_{\mathcal N_{far}(Y)}|u(Y)-\hat c|^2|\kappa\eta(Y)|^{-\f14}||\kappa\eta(Z)|^{\f14}e^{-W_\varepsilon(Y,Z)}|\partial_Z G(Z)|dZ\\
	    		&\quad+C|\varepsilon|^{-\f43}\int_{Z\in \mathcal N}|u(Y)-\hat c|^2|\kappa\eta(Y)|^{-\f14}|
	    		 e^{-W_\varepsilon(Y,Z)}| G(Z)|dZ.	 
	    	\end{split}
	    \end{align}
	    By Lemma \ref{lem:est-eta}, we have that for $Y\in\{Y:|\kappa\eta(Y)|\geq M+1,Y\geq Y_c\}$,
	    \begin{align}\label{eq:airy-langer-N+}
	    	|\kappa\eta(Y)|^{-\f14}|\kappa\eta(Z)|^{\f14}\lesssim\left\{
	    	\begin{aligned}
	    		&1,\quad Z\in\mathcal N_{near}(Y),\\
	    		&|\kappa\eta(Z)|^\f14,\quad Z\in\mathcal N_{far}(Y).
	    	\end{aligned}
	    	\right.
	    \end{align}
	    Therefore, by a similar calculation of $\mathcal M_2(Y)$, we have that for $Y\in\{Y:|\kappa\eta(Y)|\geq M+1,Y\geq Y_c\}$,
	    	    \begin{align}\label{eq:airy-app-sf-main2dd-1}
	    	\begin{split}
	    		&\int_{\mathcal N_{near}(Y)}|u(Y)-\hat c|^2 |\kappa\eta(Y)|^{-\f14}|\kappa\eta(Z)|^{\f14}e^{-W_\varepsilon(Y,Z)}|\partial_Z G(Z)|dZ\\
	    	&\quad\leq C|\varepsilon|^\f13e^{-\theta Y}\|(u-\hat c)^2\partial_Y G\|_{L^\infty},
	    	\end{split}
	    \end{align}
	    and 
	   	        \begin{align}\label{eq:airy-app-sf-main2dd-2}
	    	\begin{split}
	    		&\int_{\mathcal N_{far}(Y)}|u(Y)-\hat c|^2|\kappa\eta(Y)|^{-\f14}||\kappa\eta(Z)|^{\f14}e^{-W_\varepsilon(Y,Z)}|\partial_Z G(Z)|dZ\\
	    	&\quad\leq C|\varepsilon|^\f13e^{-\theta Y}\|(u-\hat c)^2\partial_Y G\|_{L^\infty}.
	    	\end{split}
	    \end{align}
	    For the last term on the right hand side of \eqref{eq:airy-app-sf-main2dd}, we notice that if $Y-Y_c\geq L$,
	    \begin{align*}
	    	&\int_{Z\in \mathcal N}|u(Y)-\hat c|^2|\kappa\eta(Y)|^{-\f14}|
	    		 e^{-W_\varepsilon(Y,Z)}| G(Z)|dZ\\
	    	&	 \leq Ce^{-C|\kappa\eta(Y)|^\f32}\|(u-\hat c) G\|_{L^\infty}\int_{Z\in\mathcal N}|u(Z)-\hat c|^{-1}dZ\\
	    	&	 \leq Ce^{-C|\varepsilon|^{-\f12}}|\log|\hat c_i| |e^{-\theta Y}\|(u-\hat c) G\|_{L^\infty}\leq C|\varepsilon|^{\f23}e^{-\theta Y}\|(u-\hat c) G\|_{L^\infty},
	    \end{align*}
	    and  for any $Y\in\{Y:|\kappa\eta(Y)|\geq M+1, 0\leq Y-Y_c\leq L\}$,
	    \begin{align*}
	    	\int_{Z\in \mathcal N}|u(Y)-\hat c|^2|\kappa\eta(Y)|^{-\f14}|
	    		 e^{-W_\varepsilon(Y,Z)}| G(Z)|dZ
	    		 \leq& C|\varepsilon|^\f23\int_{Z\in\mathcal N}|\kappa\eta(Y)|^\f74e^{-C|\kappa\eta(Y)|^\f12}| G(Z)|dZ\\
	    		 \leq&C|\varepsilon|^{\f23}|\log|\hat c_i|| e^{-\theta Y}\|(u-\hat c)G\|_{L^\infty},
	    \end{align*}
	    which implies that for any $Y\in\{Y:|\kappa\eta(Y)|\geq M+1,Y\geq Y_c\}$,
	    \begin{align}\label{eq:airy-app-sf-main2dd-3}
	    	\begin{split}
	    		&\int_{Z\in \mathcal N}|u(Y)-\hat c|^2|\kappa\eta(Y)|^{-\f14}|
	    		 e^{-W_\varepsilon(Y,Z)}| G(Z)|dZ\leq C|\varepsilon|^{\f23}|\log|\hat c_i||e^{-\theta Y}\|(u-\hat c) G\|_{L^\infty}.
	    	\end{split}
	    \end{align}
	    Therefore by \eqref{eq:airy-app-sf-main2dd}, \eqref{eq:airy-app-sf-main2dd-1}, \eqref{eq:airy-app-sf-main2dd-2} and \eqref{eq:airy-app-sf-main2dd-3}, we infer that for any $Y\in\{Y:|\kappa\eta(Y)|\geq M+1,Y\geq Y_c\}$,
	    \begin{align}\label{eq:airy-app-sf-maindd2-e}
	    	|(u-\hat c)^2\mathcal M_{2,1}^{(2)}(Y)|\leq C|\varepsilon|^{-\f23}|\log|\hat c_i||e^{-\theta Y}(\|(u-\hat c)G\|_{L^\infty}+\|(u-\hat c)^2\partial_Y G\|_{L^\infty}).  
	    \end{align}
        By a similar argument as in the calculation of $\mathcal E^{(1)}_j$, we can obtain that for any $|\kappa\eta(Y)|\geq M+1$,
        \begin{align*}
        	|(u-\hat c)\mathcal E^{(2)}_1(Y)|\leq C|\varepsilon|^{-\f13}(1+|\varepsilon|^\f13|\hat c_i|^{-1})e^{-\theta Y}(\|(u-\hat c)G\|_{L^\infty}+\|(u-\hat c)^2\partial_Y G\|_{L^\infty}),
        \end{align*}
	    which along with \eqref{eq:airy-app-sf-dd-e}, \eqref{eq:airy-app-sf-maindd1-e} and \eqref{eq:airy-app-sf-maindd2-e} deduces that for any $Y\in\{Y:|\kappa\eta(Y)|\geq M+1,Y\geq Y_c\}$,
	    \begin{align*}
	    	|(u-\hat c)^2\partial_Y^2\psi_{app,s}(Y)|\leq C|\varepsilon|^{-\f23}|\log|\hat c_i||e^{-\theta  Y}(\|(u-\hat c)G\|_{L^\infty}+\|(u-\hat c)^2\partial_Y G\|_{L^\infty}). 
	    \end{align*}
	    By a similar argument, we  can also infer that for any $Y\in\{Y:|\kappa\eta(Y)|\geq M+1, Y\leq Y_c\}$,
	     \begin{align*}
	    	|(u-\hat c)^2\partial_Y^2\psi_{app,s}(Y)|\leq C|\varepsilon|^{-\f23}|\log|\hat c_i||e^{-\theta Y}(\|(u-\hat c)G\|_{L^\infty}+\|(u-\hat c)^2\partial_Y G\|_{L^\infty}). 
	    \end{align*}
	    To show the estimates about $(u-\hat c)^2\partial_Y^2\psi_{app,s}(Y)$ for the case of $Y\in\{Y:|\kappa\eta(Y)|\leq M+1\}$, we notice that 
	    \begin{align}\label{eq:airy-app-sf-maindd1-se}
		\begin{split}
			\partial_Y^2\psi_{app,s}(Y)=&\mathcal M_{1,2}^{(2)}(Y)+\mathcal M_{2,2}^{(2)}(Y) +\mathcal E^{(2)}(Y),
		\end{split}
	\end{align}
	with 
	\begin{align*}
		\mathcal M_{1,2}^{(2)}(Y)=&-A_1(Y)A_2(0)\partial_Y\eta(0)^{-1} \partial_Y G(0)+A_1(Y)\partial_Y A_2(0)\partial_Y\eta(0)^{-1} G(0),   \\
		\mathcal M_{2,2}^{(2)}(Y)=&A_1(Y)\int_0^Y\partial_Z^2A_2\partial_Z\eta(Z)^{-1}G(Z) dZ+A_2(Y)\int_Y^{+\infty}\partial_Z^2A_1\partial_Z\eta(Z)^{-1}G(Z)dZ, \\ 
			\mathcal E^{(2)}(Y)=&A_1(Y)\int_0^Y\frac{\partial_Z^2\eta(Z)}{\partial_Z\eta(Z)^2}(A_2(Z)\partial_ZG(Z)-\partial_ZA_2(Z)G(Z))dZ\\
			&+A_2(Y)\int_Y^{+\infty}\frac{\partial_Z^2\eta(Z)}{\partial_Z\eta(Z)^2}(A_1(Z)\partial_ZG(Z)-\partial_ZA_1(Z)G(Z))dZ.
	\end{align*}
Again by Lemma \ref{lem:A1A2} and a similar argument of the calculation of $\mathcal M_1(Y)$, we can obtain that for any $Y\in\{Y:|\kappa\eta(Y)|\leq M+1\}$,
       \begin{align}\label{eq:airy-app-sf-maindd-s}
       	\begin{split}
       		|(u-\hat c)^2\mathcal M_{1,2}^{(2)}(Y)|\leq& C(|\varepsilon|^{-\f13}|\hat c|^{-1}+|\hat c|^{-2})\|(u-\hat c)^2\partial_Y G\|_{L^\infty}\\
	    	&+C(|\varepsilon|^{-\f23}+|\varepsilon|^{-\f13}|\hat c|^{-1})\|(u-\hat c) G\|_{L^\infty}\\
	    	\leq&C|\varepsilon|^{-\f23}(\|(u-\hat c) G\|_{L^\infty}+\|(u-\hat c)^2\partial_Y G\|_{L^\infty}.
       	\end{split}
       \end{align}
       
      \no\underline{Estimates about $\mathcal M^{(2)}_{2,2}(Y)$ for $Y\in\{Y:|\kappa\eta(Y)|\leq M+1\}$. }
       By Lemma \ref{lem:A1A2}, we have that for any $Y\in\{Y:|\kappa\eta(Y)|\leq M+1\}$,
      \begin{align*}
      	|(u-\hat c)^2\mathcal M_{2,2}^{(2)}(Y)|\leq&C|\varepsilon|^{-\f43}\int_{Z\in\{|\kappa\eta(Z)|\geq M+2\}}|u(Y)-\hat c|^2|\kappa\eta(Z)|^\f34|\partial_Z\eta(Z)|e^{-W_\varepsilon(Y,Z)}|G(Z)|dZ\\
       		&+C|\varepsilon|^{-\f43}\int_{Z\in\{|\kappa\eta(Z)|\leq M+2\}}|u(Y)-\hat c|^2e^{-W_\varepsilon(Y,Z)}|G(Z)|dZ=\mathcal{II}_1+\mathcal{II}_2.
      \end{align*}
       About $\mathcal{II}_1$, we notice that for any $Y\in\{Y:|\kappa\eta(Y)|\leq M+1\}$,
       \begin{align*}
       	|\mathcal{II}_1(Y)|\leq& C|\varepsilon|^{-1}\|(u-\hat c)G\|_{L^\infty}\int_{Z\in\{|\kappa\eta(Z)|\geq M+2\}}|\kappa\eta(Z)|^\f34e^{-C|\kappa\eta(Z)|^\f12}dZ\\
       		\leq& C|\varepsilon|^{-\f23}\|(u-\hat c)G\|_{L^\infty}\leq C|\varepsilon|^{-\f23}\|(u-\hat c)G\|_{L^\infty}.
       \end{align*}
       About $\mathcal{II}_2$, we notice that for any $Y\in\{Y:|\kappa\eta(Y)|\leq M+1\}$,
       \begin{align*}
       	|\mathcal{II}_2(Y)|\leq& C|\varepsilon|^{-\f23} \|(u-\hat c)G\|_{L^\infty}\int_{Z\in\{|\kappa\eta(Z)|\leq M+2\}}|u(Z)-\hat c|^{-1}dZ\\
       		\leq& C|\varepsilon|^{-\f23}|\log|\hat c_i||\|(u-\hat c)G\|_{L^\infty}.
       \end{align*}
       Therefore, we obtain that for any $Y\in\{Y:|\kappa\eta(Y)|\leq M+1\}$,
       \begin{align}\label{eq:airy-app-sf-maindd2-se}
       	|(u-\hat c)^2\mathcal M_{2,2}^{(2)}(Y)|\leq C|\varepsilon|^{-\f23}|\log|\hat c_i||\|(u-\hat c)G\|_{L^\infty}.
       \end{align}
       Again by a similar argument as in the calculation of $\mathcal E^{(1)}_j$, we can obtain that for any $Y\in\{Y:|\kappa\eta(Y)|\leq M+1\}$,
       \begin{align*}
       	\mathcal E^{(2)}_2(Y) \leq C|\varepsilon|^{-\f13}(|\log|\hat c_i||+|c||\varepsilon|^{-\f13})(\|(u-\hat 
        	c)G\|_{L^\infty}+ \|(u-\hat c)^2\partial_Y G\|_{L^\infty}),
       \end{align*}
       which along with\eqref{eq:airy-app-sf-maindd1-se}, \eqref{eq:airy-app-sf-maindd-s} and \eqref{eq:airy-app-sf-maindd2-se} deduces that for any $Y\in\{Y:|\kappa\eta(Y)|\leq M+1\}$,
       \begin{align*}
       	|(u-\hat c)\partial_Y^2\psi_{app,s}(Y)|\leq& C|\varepsilon|^{-\f23}|\log|\hat c_i||(\|(u-\hat c)G\|_{L^\infty}+\|(u-\hat c)^2\partial_Y G\|_{L^\infty}).   
	    	\end{align*}
       Therefore,  we obtain that for any $Y\geq 0$,
          \begin{align*}
       	|(u-\hat c)^2\partial_Y^2\psi_{app,s}(Y)|\leq& C|\varepsilon|^{-\f23}|\log|\hat c_i||e^{-\theta Y}(\|(u-\hat c)G\|_{L^\infty}+\|(u-\hat c)^2\partial_Y G\|_{L^\infty}).    
	    	\end{align*}
		
	    \no\textbf{Estimates about $\partial_Yw_{app,s}(Y)$.} By the definition of $w_{app,s}$, we have that for any $Y\geq 0$,
	     \begin{align*}
	    	\partial_Y w_{app,s}(Y)=\mathcal M_1^{(3)}(Y)+\mathcal M_2^{(3)}(Y)+\mathcal E^{(3)}(Y),
	    \end{align*}
	    where 
	    \begin{align*}
	    	\mathcal M^{(3)}_1(Y)=&-\varepsilon^{-1}\partial_Y G(Y)-\partial_YA_1(Y)A_2(0)\partial_Y\eta(0)^{-1}\partial_YG(0),\\
	    	\mathcal M^{(3)}_2(Y)=&-\partial_Y A_1(Y)\int_0^Y\partial_ZA_2\partial_Z\eta^{-1}\partial_Z GdZ-\partial_Y A_2(Y)\int_0^Y\partial_ZA_1\partial_Z\eta^{-1}\partial_Z GdZ,\\
	    	\mathcal E^{(3)}(Y)=&\partial_Y A_1(Y)\int_0^YA_2(Z)\frac{\partial_Z^2\eta}{(\partial_Z\eta)^2}\partial_Z GdZ+\partial_Y A_2(Y)\int_0^Y A_1(Z)\frac{\partial_Z^2\eta}{(\partial_Z\eta)^2}\partial_Z GdZ.
	    \end{align*}

       Firstly, we notice that for any $Y\geq 0$,
       \begin{align*}
       	|\varepsilon|^{-1}|(u-\hat c)\partial_Y G(Y)|\leq Ce^{-\theta Y}|\varepsilon|^{-1}|\hat c_i|^{-1}\|(u-\hat c)^2\partial_YG\|_{L^\infty},
       \end{align*}
       which along with a similar argument as in $\mathcal M^{(1)}_1(Y)$ implies that for any $Y\geq 0$,
       \begin{align}\label{eq:airy-app-sf-mainddd-e}
       	|(u-\hat c)\mathcal M_1^{(3)}(Y)|\leq C(|\varepsilon|^{-1}|\hat c_i|^{-1}+|\varepsilon|^{-\f23}|\hat c|^{-2})e^{-\theta Y}\|(u-\hat c)^2\partial_YG\|_{L^\infty}.
       \end{align}
       Again by Lemma \ref{lem:A1A2} and a similar calculation of $\mathcal M_2(Y)$, we can obtain that for any $Y\geq 0$,
	    \begin{align}\label{eq:airy-app-sf-mainddd-e1}
	    	|(u-\hat c)\mathcal M_2^{(3)}(Y)	|\leq&C(|\varepsilon|^{-\f43}+|\varepsilon|^{-1}|\hat c_i|^{-1})e^{-\theta Y}\| (u-\hat c)^2\partial_Y G\|_{L^\infty}.
	    \end{align}
	    By a similar argument as above, we can obtain that for any $Y\geq 0$,
	    \begin{align*}
	    	|(u-\hat c)\mathcal E^{(3)}(Y)|\leq&C(|\varepsilon|^{-1}+|\varepsilon|^{-\f23}|\hat c_i|^{-1})e^{-\theta Y}\|(u-\hat c)^2\partial_Y G\|_{L^\infty},
	    \end{align*}
	    which along with \eqref{eq:airy-app-sf-mainddd-e} and \eqref{eq:airy-app-sf-mainddd-e1} deduces that for any $Y\geq 0$,
	    \begin{align*}
	    	|(u-\hat c)\partial_Yw_{app,s}(Y)|\leq C(|\varepsilon|^{-1}|\hat c_i|^{-1}+|\varepsilon|^{-\f43})e^{-\theta Y}\|(u-\hat c)^2\partial_YG\|_{L^\infty}.
	    \end{align*}
	    
	    The proof is completed.
	   \end{proof}
	   
	\begin{Corollary}\label{cor:airy-app-sf-err}
		Under the same assumptions in Proposition \ref{prop:airy-green-Y}, $\psi_{app,s}(Y)$ defined  by \eqref{eq:psi-app-01} satisfies that for any $Y\geq 0$,
		\begin{align*}
			&|e^{\theta Y}(u-\hat c)Err_1(Y)w_{app,s}(Y)|\leq C|\varepsilon|^{-\f13}|\log|\hat c_i||(\|(u-\hat c)^2\partial_Y G\|_{L^\infty}+\|(u-\hat c) G\|_{L^\infty}),  \\
			&|e^{\theta Y}(u-\hat c)Err_2(Y)\partial_Y w_{app,s}(Y)|\leq C(|\hat c_i|^{-1}+|\varepsilon|^{-\f13})\|(u-\hat c)^2\partial_YG\|_{L^\infty}.
		\end{align*}
	\end{Corollary}
\begin{proof}
By Lemma \ref{lem:err1-err2},  we have for any $Y\geq 0$,
\begin{align*}
|Err_1(Y)|\leq C|\varepsilon|^\f13 |u(Y)-c|\leq C|\varepsilon|^\f13 |u-\hat{c}|,
\end{align*}
which implies that for any $Y\geq 0$,
\begin{align*}
|e^{\theta_0 Y}(u-\hat c)Err_1(Y)w_{app,s}(Y)|\leq& C|\varepsilon|^\f13|(u-\hat c)^2w_{app,s}(Y)|\\
\leq& C|\varepsilon|^{-\f13}|\log|\hat c_i||\|(u-\hat c)^2\partial_Y G\|_{L^\infty}+|\log|\hat c_i||\|(u-\hat c) G\|_{L^\infty}.      
\end{align*}
We also notice that for any $Y\geq 0$,
\begin{align*}
|Err_2(Y)|\leq C|\varepsilon|,
\end{align*}
which implies that for any $Y\geq 0$,
\begin{align*}
|e^{\theta_0 Y}(u-\hat c)Err_2(Y)\partial_Y w_{app,s}(Y)|\leq& C|\varepsilon||(u-\hat c)\partial_Yw_{app,s}(Y)|\\
\leq& C(|\hat c_i|^{-1}+|\varepsilon|^{-\f13})\|(u-\hat c)^2\partial_YG\|_{L^\infty}.
\end{align*}

The proof is completed.
\end{proof}

By the construction of $\psi_{app,s}$ in Proposition \ref{prop:airy-green-Y}, we can obtain the following estimates for the boundary value of $\psi_{app,s}$ and $\partial_Y\psi_{app,s}$.
\begin{Corollary}\label{coro:airy-app-sf-boundary}
	Under the same assumptions in Proposition \ref{prop:airy-green-Y}, the boundary value of the  solution $\psi_{app,s}(Y)$ defined  by \eqref{eq:psi-app-01} satisfies 
	\begin{align*}
		&\left|\psi_{app,s}(0)-u'(0)^{-1}c_r\mathcal M_0\right|\leq C|\varepsilon|^{-\f23}|\log|\hat c_i|| (\|(u-\hat 
        	c)G\|_{L^\infty}+\|(u-\hat c)^2\partial_Y G\|_{L^\infty}),\\
        &\left|\partial_Y\psi_{app,s}(0)+\mathcal M_0\right|\leq C|\varepsilon|^{-\f23}|\hat c|^{-1}|\log|\hat c_i|| (\|(u-\hat 
        	c)G\|_{L^\infty}+\|(u-\hat c)^2\partial_Y G\|_{L^\infty}),
	\end{align*}
	where
	\begin{align*}
		\mathcal M_0[G]:=\int_0^{+\infty}\mathcal H_1(Z)\partial_Z\eta(Z)^{-1}G(Z)dZ.
	\end{align*}
	Moreover, we have
	\begin{align*}
		|\mathcal M_0[G]|\leq C|\varepsilon|^{-1}|\log|\hat c_i||(\|(u-\hat 
        	c)G\|_{L^\infty}+\|(u-\hat c)^2\partial_Y G\|_{L^\infty}).
	\end{align*}
\end{Corollary}
\begin{proof}
By the definition of $\psi_{app,s}(Y)$, we know that 
\begin{align*}
\psi_{app,s}(0)=\sum_{j=1}^4\mathcal M_j(0)+\sum_{j=1}^3\mathcal E_j^{(0)}(0).
\end{align*}	
According to the calculation in the proof of Proposition \ref{prop:airy-green-Y}, we have 
\begin{align*}
&\left|\psi_{app,s}(0)-\mathcal M_3(0)\right|\leq C|\varepsilon|^{-\f23}|\log|\hat c_i|| (\|(u-\hat 
        	c)G\|_{L^\infty}+\|(u-\hat c)^2\partial_Y G\|_{L^\infty}).
\end{align*}
We also  notice that 
\begin{align*}
	\left|\mathcal M_3(0)-Y_c\mathcal M_0[G]\right|=&\left|\int_0^{+\infty}(Z-Y_c)\mathcal H_1(Z)\partial_Z\eta(Z)^{-1}G(Z)dZ\right|\\
	\leq& C|\varepsilon|^{-\f23}|\log|\hat c_i|| (\|(u-\hat 
        	c)G\|_{L^\infty}+\|(u-\hat c)^2\partial_Y G\|_{L^\infty}),
\end{align*}
provided by a similar argument as in the proof of  Proposition \ref{prop:airy-green-Y}. 
Then we obtain 
\begin{align*}
	&\left|\psi_{app,s}(0)-Y_c\mathcal M_0[G]\right|\leq C|\varepsilon|^{-\f23}|\log|\hat c_i|| (\|(u-\hat c)G\|_{L^\infty}+\|(u-\hat c)^2\partial_Y G\|_{L^\infty}),
\end{align*}
which along with the fact $Y_c=u'(0)^{-1}c_r+\mathcal O(|\hat c|^2)$
implies that 
\begin{align*}
	\left|\psi_{app,s}(0)-u'(0)^{-1}c_r\mathcal M_0[G]\right|\leq C|\varepsilon|^{-\f23}|\log|\hat c_i||(\|(u-\hat c)G\|_{L^\infty}+\|(u-\hat c)^2\partial_Y G\|_{L^\infty}).
\end{align*}
By the definition of $\partial_Y\psi_{app,s}(Y)$, we know that 
\begin{align*}
			\partial_Y\psi_{app,s}(0)=\mathcal M_0[G]+\mathcal M_{1}^{(1)}(0)+\mathcal M_{2}^{(1)}(0)+\mathcal E^{(1)}_1(0)+\mathcal E^{(1)}_2(0),
\end{align*}
which along with the arguments in Proposition \ref{prop:airy-green-Y} deduces that 
\begin{align*}
	&\left|\partial_Y\psi_{app,s}(0)+\mathcal M_0[G]\right|\leq C|\hat c|^{-1}|\varepsilon|^{-\f23}|\log|\hat c_i|| (\|(u-\hat 
        	c)G\|_{L^\infty}+\|(u-\hat c)^2\partial_Y G\|_{L^\infty}).
\end{align*}

Moreover, by a similar calculation as in Proposition \ref{prop:airy-green-Y}, we obtain 
\begin{align*}
	\left|\mathcal M_0[G]\right|\leq  C|\varepsilon|^{-1}|\log|\hat c_i||(\|(u-\hat 
        	c)G\|_{L^\infty}+\|(u-\hat c)^2\partial_Y G\|_{L^\infty}).
\end{align*}

The proof is completed.
\end{proof}

\subsection{Non-homogeneous Airy equation}

In this part, we are aim to construct a solution to the Airy equation \eqref{eq:Airy-eq} with several types of source terms $F$. By a similar argument as in the proof of Proposition 4.7 in \cite{MWWZ}, we can obtain the following result.

\begin{theorem}\label{them:airy-gf-1}
	Let $1\lesssim\theta\leq 3\eta_0$ and $\hat c:=c+\mathrm i c_0$ with $c_0\in(|\varepsilon|^\f23,|\varepsilon|^\f13)$. Suppose $(\alpha,c)\in\mathbb H_2$ with $c_i>-c_0/2$. Then there exists a solution $\psi\in W^{4,\infty}$ to \eqref{eq:Airy-eq} satisfies the following properties:
	\begin{itemize}
		\item  for any $e^{\theta Y}F\in L^{\infty}$, we have 
		\begin{align*}
		&|\varepsilon|\|\partial^2_Y w\|_{L^\infty_\theta}+	|\varepsilon|^{\f23}\|\partial_Y w\|_{L^\infty_\theta}+\|(u-\hat c)w\|_{L^\infty_\theta}\leq C \|F\|_{L^\infty_\theta},\\
		&\|\psi\|_{L^\infty_\theta}+\|\partial_Y\psi\|_{L^\infty_\theta}\leq C|\log|\hat c_i||\|F\|_{L^\infty_\theta}.
		\end{align*}
		
		\item for any  $e^{\theta Y}(u-\hat c)F\in L^{\infty}$, we have
		\begin{align*}
		&\|(u-\hat{c})^2w\|_{L^\infty_\theta}+|\varepsilon|^\f23\|(u-\hat{c})\partial_Yw\|_{L^\infty_\theta}+|\varepsilon|\|(u-\hat{c})\partial_Y^2 w\|_{L^\infty_\theta}\leq C|\log|\hat{c}_i|\|(u-\hat c)F\|_{L^\infty_\theta},\\
		&\|\psi\|_{L^\infty_\theta}+\|(u-\hat{c})\partial_Y\psi\|_{L^\infty_\theta}\leq C|\hat{c}||\varepsilon|^{-\f13}|\log|\hat c_i|| \|(u-\hat c)F\|_{L^\infty_\theta}.
		\end{align*}
	\end{itemize}
\end{theorem}

Now, we present the results for the singular source term corresponding to the sublayer.

\begin{theorem}\label{them:airy-sf}
	Under the same assumptions in Proposition \ref{prop:airy-green-Y}, there is a solution to \eqref{eq:Airy-eq} with the source term $F(Y)=\partial_Y^2G(Y)$ satisfying the following properties
\begin{align*}
&\|\psi_{s}\|_{L^\infty_\theta}+\|(u-\hat{c})\partial_Y\psi_{s}\|_{L^\infty_\theta}\leq C|\varepsilon|^{-\f23}|\log|\hat{c}_i||(|\varepsilon|^{-\f13}|\hat c|\|(u-\hat c)G\|_{L^\infty}+\|(u-\hat c)^2\partial_Y G\|_{L^\infty}),  \\
&|\varepsilon|^{\f23}\|(u-\hat c)\partial_Yw_{s}\|_{L^\infty_\theta}\leq  C(|\varepsilon|^{-\f13}|\hat c_i|^{-1}+|\varepsilon|^{-\f23})|\log|\hat{c}_i||(\|(u-\hat c)^2\partial_Y G\|_{L^\infty}+\|(u-\hat c) G\|_{L^\infty}),\\
&\|(u-\hat c)^2w_{s}\|_{L^\infty_\theta}\leq C|\varepsilon|^{-\f23}|\log|\hat c_i||(\|(u-\hat c)^2\partial_Y G\|_{L^\infty}+C\|(u-\hat c) G\|_{L^\infty}),
\end{align*}
and 
\begin{align*}
\|\mathcal W(Y)\partial^2_Yw_{s}\|_{L^\infty_\theta}
\leq&C(|\varepsilon|^{-\f13}|\hat c_i|^{-1}+|\varepsilon|^{-\f23})\|(u-\hat c)^2\partial_Y G\|_{L^\infty}+C|\varepsilon|^{-\f23}|\log|\hat c_i||\|(u-\hat c) G\|_{L^\infty}\\
&+C|\varepsilon|^{-\f23}\|(u-\hat c)^3\partial_Y^2G\|_{L^\infty},
\end{align*}
where $\mathcal W(Y)>0$ is a smooth weighted function defined such that for any $Y\in\mathcal N$, $\mathcal W(Y)\sim |\varepsilon|^\f13 |u-\hat c|^3$ and for any $Y\in\mathcal N^-\cup\mathcal N^+$, $\mathcal W(Y)\sim |\varepsilon||u-\hat c|$.\\
Moreover, we have 
\begin{align*}
	&|\psi_s(0)-u'(0)^{-1}c_r\mathcal M_0[G]|\leq C|\varepsilon|^{-\f23}(\|(u-\hat c)^2\partial_Y G\|_{L^\infty}+C\|(u-\hat c) G\|_{L^\infty}),\\
	&|\partial_Y\psi_s(0)+\mathcal M_0[G]|\leq |\varepsilon|^{-\f23}|\hat c|^{-1}(\|(u-\hat c)^2\partial_Y G\|_{L^\infty}+C\|(u-\hat c) G\|_{L^\infty}),
\end{align*}
with 
\begin{align*}
	|\mathcal M_0[G]|\leq C|\varepsilon|^{-1}(\|(u-\hat c)^2\partial_Y G\|_{L^\infty}+C\|(u-\hat c) G\|_{L^\infty}).
\end{align*}
\end{theorem}
\begin{proof}
We first define $w_s^{(0)}(Y):=w_{app,s}(Y)$, where $w_{app,s}(Y)$ is the solution to \eqref{eq:Airy-app-SF} constructed in Proposition \ref{prop:airy-green-Y}. Then we know that 
\begin{align*}
Airy[w_{app,s}](Y)=\partial_Y^2G(Y)+Err_1(Y)w_{app,s}(Y)+Err_2\partial_Y w_{app,s}(Y).
\end{align*}
Moreover, by Proposition \ref{prop:airy-green-Y} and Corollary \ref{cor:airy-app-sf-err}, we have 
\begin{align}\label{eq:airy-sf-e3}
\begin{split}
\|\mathcal W(Y)\partial^2_Yw_{app,s}\|_{L^\infty_\theta}\leq& C\|(u-\hat c)^2w_{app,s}\|_{L^\infty_\theta}+C\|(u-\hat c)^3\partial_Y^2G\|_{L^\infty}\\
&+\|(u-\hat c)Err_1w_{app,s}\|_{L^\infty_\theta}
+\|(u-\hat c)Err_2\partial_Y w_{app,s}\|_{L^\infty_\theta}\\
\leq&C(|\varepsilon|^{-\f13}|\hat c_i|^{-1}+|\varepsilon|^{-\f23})|\log|\hat c_i||\|(u-\hat c)^2\partial_Y G\|_{L^\infty}\\
&+C|\varepsilon|^{-\f23}|\log|\hat c_i||\|(u-\hat c) G\|_{L^\infty}+C|\varepsilon|^{-\f23} \|(u-\hat c)^3\partial_Y^2G\|_{L^\infty},
\end{split}
\end{align}
where $\mathcal W(Y)>0$ is a smooth weighted function defined such that for any $Y\in\mathcal N$, $\mathcal W(Y)\sim |\varepsilon|^\f13 |u-\hat c|^3$ and for any $Y\in\mathcal N^-\cup\mathcal N^+$, $\mathcal W(Y)\sim |\varepsilon||u-\hat c|$.

Now we define $(w_{s,err}(Y),\psi_{s,err}(Y))$ to be the solution of
\begin{align*}
&\varepsilon(\partial_Y^2-\alpha^2)w_{s,err}-(u-c)w_{s,err}=-Err_1w_s^{(0)}-Err_2\partial_Yw_s^{(0)},\\
&(\partial_Y^2-\alpha^2)\psi_{s,err}(Y)=w_{s,err}(Y),\quad
		\lim_{Y\to\infty}\psi_{s,err}(Y)=\lim_{Y\to\infty}w_{s,err}(Y)=0,
\end{align*}
constructed in Theorem \ref{them:airy-gf-1}. Then we first have 
\begin{align*}
Airy[w^{(0)}_s+w_{s,err}](Y)=\partial_Y^2G.
\end{align*}
By Corollary \ref{cor:airy-app-sf-err}  and   Theorem \ref{them:airy-gf-1}, we obtain 
\begin{align*}
&\|(u-\hat{c})^2w_{s,err}\|_{L^\infty_\theta}+|\varepsilon|^\f23\|(u-\hat{c})\partial_Yw_{s,err}\|_{L^\infty_\theta}+|\varepsilon|\|(u-\hat{c})\partial_Y^2 w_{s,err}\|_{L^\infty_\theta}\\
&\leq C|\log|\hat{c}_i||(\|(u-\hat c)Err_1w_{app,s}\|_{L^\infty_\theta}+\|(u-\hat c)Err_2\partial_Y w_{app,s}\|_{L^\infty_\theta})\\
&\leq C(|\hat c_i|^{-1}+|\varepsilon|^{-\f13})|\log|\hat c_i||^2(\|(u-\hat c) G\|_{L^\infty}+\|(u-\hat c)^2\partial_YG\|_{L^\infty}).
\end{align*}
Similarly, we can obtain 
\begin{align}\label{eq:airy-sf-se1}
\begin{split}
&\|\psi_{s,err}\|_{L^\infty_\theta}+\|(u-\hat{c})\partial_Y\psi_{s,err}\|_{L^\infty_\theta}\\
&\quad\leq C|\log|\hat c_i||^3 (|\hat c_i|^{-1}+|\varepsilon|^{-\f13})(\|(u-\hat c) G\|_{L^\infty}+\|(u-\hat c)^2\partial_YG\|_{L^\infty}).
\end{split}
\end{align}
By Proposition \ref{prop:airy-green-Y}, \eqref{eq:airy-sf-e3} and the above estimates, we have 
\begin{align*}
&\|\bar \psi_{s}\|_{L^\infty_\theta}+\|(u-\hat{c})\partial_Y\bar\psi_{s}\|_{L^\infty_\theta}\\
        	&\quad\leq C|\varepsilon|^{-1}|\hat c||\log|\hat c_i||(\|(u-\hat c)G\|_{L^\infty}+C|\varepsilon|^{-\f23}|\log|\hat c_i| |\|(u-\hat c)^2\partial_Y G\|_{L^\infty},\\
&|\varepsilon|^{\f23}\|(u-\hat c)\partial_Yw_{s}\|_{L^\infty_\theta}\leq  C(|\varepsilon|^{-\f13}|\hat c_i|^{-1}+|\varepsilon|^{-\f23})(\|(u-\hat c)^2\partial_Y G\|_{L^\infty}+C\|(u-\hat c) G\|_{L^\infty}),\\
&\|(u-\hat c)^2w_{s}\|_{L^\infty_\theta}\leq C|\varepsilon|^{-\f23}|\log|\hat c_i||(\|(u-\hat c)^2\partial_Y G\|_{L^\infty}+C\|(u-\hat c) G\|_{L^\infty}),
\end{align*}
and 
\begin{align*}
\|\mathcal W(Y)\partial^2_Yw_{s}\|_{L^\infty_\theta}
\leq& C(|\varepsilon|^{-\f13}|\hat c_i|^{-1}+|\varepsilon|^{-\f23})\|(u-\hat c)^2\partial_Y G\|_{L^\infty}+C|\varepsilon|^{-\f23}|\log|\hat c_i||\|(u-\hat c) G\|_{L^\infty}\\
&+C|\varepsilon|^{-\f23}\|(u-\hat c)^3\partial_Y^2G\|_{L^\infty}.
\end{align*}
Here $\bar\psi_s:=\psi_{app,s}+\psi_{s,err}$.

Now we are left with the construction of $\psi_s(Y)$. By the definitions of $\psi_{app,s}$ and $\psi_{s,err}$, we have 
\begin{align*}
	(\partial_Y^2-\alpha^2)(\psi_{app,s}+\psi_{s,err})=w_s-\alpha^2\psi_{app,s}.
\end{align*}
Hence, we can construct $\psi_{s,1}$ by the iteration such that 
\begin{align*}
	(\partial_Y^2-\alpha^2)\psi_{s,1}=\alpha^2\psi_{app,s},\quad\psi_{app,1}(\infty)=0,
\end{align*}
and 
\begin{align}\label{eq:airy-sf-se2}
	\begin{split}
			&\|\psi_{s,1}\|_{L^\infty_{\theta}}+\|(u-\hat{c})\partial_Y\psi_{s,1}\|_{L^\infty_\theta}\leq C\alpha^2\|e^{\theta_0\eta_r(Y)^\f32}\psi_{s,app}\|_{L^\infty}\\
	&\quad\leq C\alpha^2(|\varepsilon|^{-\f23}+|\varepsilon|^{-1}|\hat c|)|\log|\hat c_i||(\|(u-\hat c)^2\partial_Y G\|_{L^\infty}+C\|(u-\hat c) G\|_{L^\infty}).
	\end{split}
\end{align}
Therefore, by defining $\psi_s:=\psi_{app,s}+\psi_{s,err}+\psi_{s,1}$, we have that $(\partial_Y^2-\alpha^2)\psi_s=w_s$ and 
\begin{align*}
		&\|\psi_{s}\|_{L^\infty_\theta}+\|(u-\hat{c})\partial_Y\psi_{s}\|_{L^\infty_\theta}\leq C(|\varepsilon|^{-\f23}+|\varepsilon|^{-1}|\hat c|)|\log|\hat c_i||(\|(u-\hat c)^2\partial_Y G\|_{L^\infty}+\|(u-\hat c) G\|_{L^\infty}).
\end{align*}
Moreover, by the definition of $\psi_s$ and \eqref{eq:airy-sf-se1} and \eqref{eq:airy-sf-se2}, we obtain
\begin{align*}
	&\|\psi_{s}-\psi_{app,s}\|_{L^\infty_{\theta}}+\|(u-\hat{c})\partial_Y(\psi_{s}-\psi_{app,s})\|_{L^\infty_\theta}\\
	&\quad\leq C|\varepsilon|^{-\f23}(\|(u-\hat 
        	c)G\|_{L^\infty}+\|(u-\hat c)^2\partial_Y G\|_{L^\infty}).
\end{align*}
	
The proof is completed.
\end{proof}

\subsection{Homogeneous Airy equation}

In this part, we construct a non-zero solution to the homogeneous Airy equation
\begin{align}\label{eq:airy-homo}
	\begin{split}
			&\varepsilon(\partial_Y^2-\alpha^2)w_a-(u-c)w_a=0\quad,(\partial_Y^2-\alpha^2)\psi_a(Y)=w_a(Y),\\
		&\lim_{Y\to\infty}\psi_a(Y)=\lim_{Y\to\infty}w_a(Y)=0.
	\end{split}
\end{align}
 We denote 
	\begin{align*}
		\tilde{\mathcal A}(1,Y):=-\int_Y^{+\infty}Ai(e^{\mathrm i\frac{\pi}{6}}\kappa\eta(Z))dZ \text{ and }\tilde{\mathcal A}(2,Y):=-\int_Y^{+\infty}\tilde{\mathcal A}(1,Z)dZ.
	\end{align*}
	We also define
	\begin{align*}
		w_a^{(0)}(Y):=\frac{Ai(e^{\mathrm i\frac{\pi}{6}}\kappa\eta(Y))}{Ai(e^{\mathrm i\frac{\pi}{6}}\kappa\eta(0))} \text{ and }\psi_a^{(0)}(Y)=\frac{\tilde{\mathcal A}(2,Y)}{Ai(e^{\mathrm i\frac{\pi}{6}}\kappa\eta(0))}.
	\end{align*}
	By \eqref{eq:Airy-Err-app}, we know that 
	\begin{align*}
		Airy[w_a^{(0)}(Y)]=Err_1(Y)w_a^{(0)}(Y)+Err_2(Y)\partial_Y w^{(0)}_a(Y).
	\end{align*}
	To eliminate the above errors, we define $w_{a,err}(Y)$ as the solution constructed in Theorem \ref{them:airy-gf-1} with source term $-Err_1(Y)w_a^{(0)}(Y)-Err_2(Y)\partial_Y w^{(0)}_a(Y)$. We also show that the boundary value of $\psi_{a,err}$ is a small perturbation compared with $\psi_a^{(0)}(0)$. Now we state our results.
	\begin{theorem}\label{them:airy-homo}
	Let $1\lesssim\theta\leq 3\eta_0$ and $\hat c:=c+\mathrm i c_0$ with $c_0\in(|\varepsilon|^\f23,|\varepsilon|^\f13)$. Suppose $(\alpha,c)\in\mathbb H_2$ with $c_i>-c_0/2$. Then there exists a solution $\psi_a(Y)\in W^{4,\infty}$ to \eqref{eq:airy-homo} satisfying 
	   	\begin{align*}
&\|(u-\hat{c})^2(w_a-w_a^{(0)})\|_{L^\infty_\theta}+|\varepsilon|^\f23\|(u-\hat{c})\partial_Y(w_a-w_a^{(0)})\|_{L^\infty_\theta}+|\varepsilon|\|(u-\hat{c})\partial_Y^2 (w_a-w_a^{(0)})\|_{L^\infty_\theta}\\
&\leq C|\log|\hat{c}_i||(|\varepsilon|^\f23+|c|^2)|c_i|,\\
	&\|(\psi_a-\psi_{a}^{(0)})\|_{L^\infty_\theta}+\|(u-\hat{c})\partial_Y(\psi_a-\psi_a^{(0)})\|_{L^\infty_\theta}\leq C|\varepsilon|^\f13(|c_ic|+|\varepsilon|^\f13\alpha^2).
\end{align*}	
\end{theorem}
\begin{proof}
	By the definition of $A_1(2,Y)$ and the properties of Airy function, we obtain 
	\begin{align*}
		\partial_Y^2\psi_a^{(0)}(Y)=w_a^{(0)}(Y),
	\end{align*}
	and 
	\begin{align*}
		\varepsilon(\partial_Y^2-\alpha^2)w_a^{(0)}-(u-c)w_a^{(0)}=Err_1w_a^{(0)}+Err_2\partial_Yw_a^{(0)}.
	\end{align*}
Hence, to construct a homogeneous solution to the Airy equation \eqref{eq:airy-homo}, we need to construct a solution $(w_{a,err},\psi_{a,err})$ to 
	\begin{align}\label{eq:airy-homo-err}
		\begin{split}
			&\varepsilon(\partial_Y^2-\alpha^2)w_{a,err}-(u-c)w_{a,err}=-Err_1w_a^{(0)}-Err_2\partial_Yw_a^{(0)},\\
			&(\partial_Y^2-\alpha^2)\psi_{a,err}(Y)=w_{a,err}(Y)-\alpha^2\psi_{a}^{(0)}(Y) ,\\
		&\lim_{Y\to\infty}\psi_{a,err}(Y)=\lim_{Y\to\infty}w_{a,err}(Y)=0.
		\end{split}
	\end{align}
	
	\no\textbf{Case 1. $0\in\mathcal N$.}
    We notice that if $0\in\mathcal N$, then $\mathcal N^-=\emptyset$ and $|Ai(e^{\mathrm i\frac{\pi}{6}}\kappa\eta(0))|\leq C$. Moreover, since $\mathrm{Im}(\kappa\eta(0))<\delta_0$, there exists $C_1$ independent of $\varepsilon$ such that $|Ai(e^{\mathrm i\frac{\pi}{6}}\kappa\eta(0))|\geq C_1>0$, which along with Lemma \ref{lem:A1A2} implies that 
    	\begin{align}\label{eq:airy-homo-n}
		\begin{split}
			&|w_a^{(0)}(Y)|\leq Ce^{-C|\kappa\eta(Y)|^\f32},\quad |\partial_Yw_a^{(0)}(Y)|\leq C|\varepsilon|^{-\f13}e^{-C|\kappa\eta(Y)|^\f32},\\
		&|\psi_a^{(0)}(Y)|\leq C|\varepsilon|^{\f23}e^{-C|\kappa\eta(Y)|^\f32}\text{ and }|\partial_Y\psi_a^{(0)}(Y)|\leq C|\varepsilon|^{\f13}e^{-C|\kappa\eta(Y)|^\f32}.
		\end{split}
	\end{align}
	Moreover, we have the following controls on the  boundary value of $\psi_a^{(0)}$ and $\partial_Y\psi_{a}^{(0)}$:
	\begin{align*}
		|\psi_a^{(0)}(0)|\sim|\varepsilon|^\f23 \text{ and }|\partial_Y\psi_a^{(0)}(0)|\sim|\varepsilon|^{\f13}.
	\end{align*}
	Therefore, we obtain 
	\begin{align}
		\|(u-\hat c)Err_1w_a^{(0)}\|_{L^\infty_\theta}+\|(u-\hat c)Err_2\partial_Y w_a^{(0)}\|_{L^\infty_\theta}\leq C|\varepsilon|.
	\end{align}
	Hence, by Theorem \ref{them:airy-gf-1},  there exists a solution $(w_{a,err},\psi_{a,err})$ to \eqref{eq:airy-homo-err} satisfying 
	\begin{align*}
&\|(u-\hat{c})^2w_{a,err}\|_{L^\infty_\theta}+|\varepsilon|^\f23\|(u-\hat{c})\partial_Yw_{a,err}\|_{L^\infty_\theta}+|\varepsilon|\|(u-\hat{c})\partial_Y^2 w_{a,err}\|_{L^\infty_\theta}\leq C|\log|\hat{c}_i||\varepsilon| ,
\end{align*}
and
\begin{align*}
&\|\psi_{a,err}^{(1)}\|_{L^\infty_\theta}+\|(u-\hat{c})\partial_Y\psi_{a,err}^{(1)}\|_{L^\infty_\theta}\leq C|\log|\hat c_i||^2 |\varepsilon|,
\end{align*}
where $(\partial_Y^2-\alpha^2)\psi_{a,err}^{(1)}=w_{a,err}$. By a similar argument as in the proof of Theorem \ref{them:airy-sf}, we can construct a solution $\psi_{a,err}^{(2)}$ to $(\partial_Y^2-\alpha^2)\psi_{a,err}^{(2)}=-\alpha^2\psi^{(0)}_a$ satisfying 
\begin{align*}
	&\|\psi_{a,err}^{(2)}\|_{L^\infty_\theta}+\|(u-\hat{c})\partial_Y\psi_{a,err}^{(2)}\|_{L^\infty_\theta}\leq C|\alpha|^2|\varepsilon|.
\end{align*}
Then by denoting $\psi_{a,err}(Y):=\psi_{a,err}^{(1)}(Y)+\psi_{a,err}^{(2)}(Y)$, we construct  a solution to \eqref{eq:airy-homo-err} and 
	\begin{align*}
&\|(u-\hat{c})^2w_{a,err}\|_{L^\infty_\theta}+|\varepsilon|^\f23\|(u-\hat{c})\partial_Yw_{a,err}\|_{L^\infty_\theta}+|\varepsilon|\|(u-\hat{c})\partial_Y^2 w_{a,err}\|_{L^\infty_\theta}\\
&\quad\leq C|\log|\hat{c}_i||(|c_i|+|\varepsilon|^\f13)|\varepsilon|^\f23 ,
\end{align*}
and
\begin{align*}
&\|\psi_{a,err}\|_{L^\infty_\theta}+\|(u-\hat{c})\partial_Y\psi_{a,err}\|_{L^\infty_\theta}\leq C|\log|\hat c_i||^2 (|c_i|+|\varepsilon|^\f13)|\varepsilon|^\f23.
\end{align*}
As a consequence, we have 
\begin{align*}
&\|\psi_a-\psi_a^{(0)}\|_{L^\infty_\theta}+\|(u-\hat{c})\partial_Y(\psi_a-\psi_a^{(0)})\|_{L^\infty_\theta}\leq C|\log|\hat c_i||^2 (|c_i|+|\varepsilon|^\f13)|\varepsilon|^\f23,
\end{align*}
which implies that 
\begin{align*}
	&\psi_a(0)=\psi_a^{(0)}(0)+\mathcal O(|\log|\hat c_i||^2 (|c_i|+|\varepsilon|^\f13)|\varepsilon|^\f23),\quad\partial_Y\psi_a(0)=\partial_Y\psi_{a}^{(0)}(0)+\mathcal O(|\log|\hat c_i||^2|\varepsilon|^\f23).
\end{align*}

\no\textbf{Case 2. $0\in\mathcal N^-$.}
In this case,  by Lemma \ref{lem:airy-langer-asy}, we have that for $k=0,1$,
\begin{align*}
	&\left|\partial_Y^kw_a^{(0)}(Y)\right|\sim |\varepsilon|^{-\frac{k}{3}}|\partial_Y^kAi(e^{\mathrm i\frac{\pi}{6}}\kappa\eta(Y))Ai(e^{\mathrm i\frac{5\pi}{6}}\kappa\eta(0))|,\\
	&\left|\partial_Y^k\psi_a^{(0)}(Y)\right|\sim |\varepsilon|^{\frac{2-k}{3}}|\tilde{\mathcal A}(2-k,Y)Ai(e^{\mathrm i\frac{5\pi}{6}}\kappa\eta(0))|,
\end{align*}
	which along with Lemma \ref{lem:A1A2} implies that for any $k=0,1$,
    \begin{align}\label{eq:airy-homo-k}
    \begin{split}
    	&\left|\partial_Y^kw_a^{(0)}(Y)\right|\leq C|\varepsilon|^{-\frac{k}{3}}|\kappa\eta(Y)|^{-\f14+\f k2}|\kappa\eta(0)|^\f14 e^{-C|\varepsilon|^{-\f13}|\eta(Y)-\eta(0)|(|\kappa\eta(Y)|^\f12+|\kappa\eta(0)|^\f12)},\\
    &\left|\partial_Y^k\psi_a^{(0)}(Y)\right|\leq C|\varepsilon|^{\frac{2-k}{3}}|\kappa\eta(Y)|^{-\f54+\f k2}|\kappa\eta(0)|^\f14 e^{-C|\varepsilon|^{-\f13}|\eta(Y)-\eta(0)|(|\kappa\eta(Y)|^\f12+|\kappa\eta(0)|^\f12)}.
    \end{split}
    \end{align}
	In particular, we have  
	\begin{align*}
		&|\psi_a^{(0)}(0)|\sim|\varepsilon|^\f23|\kappa\eta(0)|^{-1}\sim|\varepsilon||c|^{-1},\quad|\partial_Y\psi_a^{(0)}(0)|\sim |\varepsilon|^\f13|\kappa\eta(0)|^{-\f12}\sim|\varepsilon|^\f12|c|^{-\f12}.
	\end{align*}
    To show the estimates for the solution to \eqref{eq:airy-homo-err}, we need to show the control of $Err_1 w_a^{(0)}(Y)$ and $Err_2\partial_Yw_a^{(0)}(Y)$. By Lemma \ref{lem:err1-err2} and \eqref{eq:airy-homo-k}, we can obtain that for any $|Y|\geq |\varepsilon|^\f13$,
    \begin{align}\label{eq:airy-homo-errYL}
    \begin{split}
    &|(u-\hat c)Err_1w_a^{(0)}(Y)|\leq C|\varepsilon|^\f13|w_a^{(0)}(Y)|\leq C|\varepsilon| e^{-\theta_0\eta_r(Y)^\f32},\\
    &|(u-\hat c)Err_2\partial_Y w_a^{(0)}(Y)|\leq C|\varepsilon|e^{-\theta_0\eta_r(Y)^\f32},
    \end{split}
    \end{align}
   and for  $Y\in[0,|\varepsilon|^\f13]$, 
    \begin{align*}
    	&|(u-\hat c)Err_1w_a^{(0)}(Y)|\leq C|\varepsilon|^\f13|w_a^{(0)}(Y)|\leq C|\varepsilon|^\f13|c|^2 e^{-\theta_0\eta_r(Y)^\f32},\\
    	&|(u-\hat c)Err_2\partial_Y w_a^{(0)}(Y)|\leq C|\varepsilon|^\f23|c|e^{-\theta_0\eta_r(Y)^\f32}.
    \end{align*}
    As a consequence, we obtain 
    \begin{align*}
    	&\|(u-\hat c)Err_1w_a^{(0)}\|_{L^\infty_\theta}\leq C|\varepsilon|^\f13(|\varepsilon|^\f23+|c|^2),\quad\|(u-\hat c)Err_2\partial_Y w_a^{(0)}\|_{L^\infty_\theta}\leq C|\varepsilon|^\f23(|\varepsilon|^\f12+|c|).
    \end{align*}
   Therefore, by Theorem \ref{them:airy-gf-1}, we obtain 
   	\begin{align*}
&\|(u-\hat{c})^2w_{a,err}\|_{L^\infty_\theta}+|\varepsilon|^\f23\|(u-\hat{c})\partial_Yw_{a,err}\|_{L^\infty_\theta}+|\varepsilon|\|(u-\hat{c})\partial_Y^2 w_{a,err}\|_{L^\infty_\theta}\\
&\quad\leq C|\log|\varepsilon||(|c|^2|\varepsilon|^\f13+|c\varepsilon|^\f23).
\end{align*}

However, the bounds of the corresponding stream function $\psi_{a,err}$ deduced by Theorem \ref{them:airy-gf-1} are not small enough to make it as a small perturbation of $\psi_a^{(0)}$ when the frequency is near the upper branch of the neutral curve. Hence, our next task to provide  sharper estimates for $\psi_{a,err}(Y)$. For this purpose, we introduce the following notations
\begin{align*}
	&F_1:=\mathbf 1_{[|\varepsilon|^\f13,+\infty)}(Y)(-Err_1w_a^{(0)}-Err_2\partial_Yw_a^{(0)}),\\
	&F_2:=\mathbf 1_{[0,|\varepsilon|^\f13]}(Y)(-Err_1w_a^{(0)}-Err_2\partial_Yw_a^{(0)}).
\end{align*}
Let $(w_{a,err}^{(1)},\psi_{a,err}^{(1)})$ be the solution to \eqref{eq:Airy-eq} with source term $F_1$ constructed as in Theorem \ref{them:airy-gf-1} and $(w_{a,err}^{(2)},\psi_{a,err}^{(2)})$ be defined by  \eqref{eq:airy-app-green} with source term $F_2$ . 
Moreover, we define $(w_{a,err}^{(3)},\psi_{a,err}^{(3)})$ to be the solution to the following system
\begin{align*}
		&\varepsilon(\partial_Y^2-\alpha^2)w_{a,err}^{(3)}-(u-c)w_{a,err}^{(3)}=-Err_1w_{a,err}^{(2)}-Err_2\partial_Yw_{a,err}^{(2)},\\
			&(\partial_Y^2-\alpha^2)\psi_{a,err}^{(3)}(Y)=w_{a,err}^{(3)}(Y)-\alpha^2\psi_a^{(0)}-\alpha^2\psi_{a,err}^{(2)} ,\\
		&\lim_{Y\to\infty}\psi_{a,err}^{(3)}(Y)=\lim_{Y\to\infty}w_{a,err}^{(3)}(Y)=0.
\end{align*}
Firstly, by Theorem \ref{them:airy-gf-1} and \eqref{eq:airy-homo-errYL}, we obtain 
\begin{align*}
&\|e^{\theta_0\eta_r(Y)^\f32}\psi_{a,err}^{(1)}\|_{L^\infty}+\|e^{\theta_0\eta_r(Y)^\f32}(u-\hat{c})\partial_Y\psi_{a,err}^{(1)}\|_{L^\infty}\leq C|\log|\hat c_i||^2 |\varepsilon|.
\end{align*}
For the estimates of $\psi_{a,err}^{(2)}$, by the definition, we can write
\begin{align*}
	\partial_Y\psi_{a,err}^{(2)}(Y)=&A_1(1,Y)\int_0^YA_2(Z)\partial_Z\eta(Z)^{-1}F_2(Z)dZ+A_2(1,Y)\int_Y^{+\infty}A_1(Z)\partial_Z\eta^{-1}(Z)F_2(Z)dZ\\
	&+\int_Y^{+\infty}[A_1(1,Z)A_2(Z)-A_2(1,Z)A_1(Z)]\partial_Z\eta(Z)^{-1}F_2(Z)dZ.
\end{align*}
For any $Y\in[|\varepsilon|^\f13,+\infty)$, notice that
\begin{align*}
	\partial_Y\psi_{a,err}^{(2)}(Y)=&A_1(1,Y)\int_0^YA_2(Z)\partial_Z\eta(Z)^{-1}F_2(Z)dZ+A_2(1,Y)\int_Y^{+\infty}A_1(Z)\partial_Z\eta^{-1}(Z)F_2(Z)dZ,
\end{align*}
which implies that for any $Y\in[|\varepsilon|^\f13,+\infty)$,
\begin{align*}
	|(u-\hat c)\partial_Y\psi_{a,err}^{(2)}(Y)|\leq& C\int_0^{|\varepsilon|^\f13}|\kappa\eta(Y)|^{\f14}|\partial_Y\eta(Y)^{-1}||\kappa\eta(Z)|^{-\f14}e^{-W_\varepsilon(Y,Z)}|F_2(Z)|dZ\\
	\leq&C|c||\varepsilon|^\f13e^{-\theta Y}\int_0^{|\varepsilon|^\f13}1dZ
	\leq C|c||\varepsilon|^\f23e^{-\theta Y}.
\end{align*}
On the other hand, for any $Y\in[0,|\varepsilon|^\f13]$, we have 
\begin{align*}
	&|(u-\hat c)\partial_Y\psi_{a,err}^{(2)}(Y)|\\
	&\leq C\int_0^{|\varepsilon|^\f13}|\kappa\eta(Y)|^{\f14}||\kappa\eta(Z)|^{-\f14} e^{-W_\varepsilon(Y,Z)}|F_2(Z)|dZ+C|u-\hat c||\varepsilon|^{-\f13}\int_0^{|\varepsilon|^\f13}|\kappa\eta(Z)|^{-1}|F_2(Z)|dZ\\
	&\leq C|c||\varepsilon|^\f13\int_0^{|\varepsilon|^\f13}|\kappa\eta(Y)|^\f12e^{-\gamma_0|\varepsilon|^{-\f13}|\eta(Y)-\eta(Z)||\kappa\eta(Y)|^\f12}dZ+C|\varepsilon|^\f13|c|\int_0^{|\varepsilon|^\f13}1dZ\leq C|c||\varepsilon|^\f23.
\end{align*}
Therefore, we obtain 
\begin{align}
	\|e^{\theta Y}(u-\hat c)\partial_Y\psi_{a,err}^{(2)}\|_{L^\infty}\leq C|c||\varepsilon|^\f23,
\end{align}
which implies that 
\begin{align}
	\|e^{\theta Y}\psi_{a,err}^{(2)}\|_{L^\infty}\leq C|c||\varepsilon|^\f23|\log|\hat c_i||.
\end{align}

Now we are left with the estimates for $\psi_{a,err}^{(3)}$. Let $\psi_{a,err}^{(3,1)}$ be the solution to $(\partial_Y^2-\alpha^2)\psi_{a,err}^{(3,1)}=w_{a,err}^{(3)}$ and $\psi_{a,err}^{(3,2)}$ be the solution to $(\partial_Y^2-\alpha^2)\psi_{a,err}^{(3,2)}=-\alpha^2\psi_a^{(0)}-\alpha^2\psi_{a,err}^{(2)}$. Then by Proposition \ref{prop:airy-gf2}  Theorem \ref{them:airy-gf-1}, we obtain 
\begin{align*}
	&\|\psi_{a,err}^{(3,1)}\|_{L^\infty_\theta}+\|(u-\hat{c})\partial_Y\psi_{a,err}^{(3,1)}\|_{L^\infty_\theta}\leq C|\log|\hat c_i||^3(|c_i|+|\varepsilon|^\f13)|c_i|(|\varepsilon|^\f23+|c|^2)\leq C|c||\varepsilon|^\f23,
\end{align*}
and 
\begin{align*}
	&\|\psi_{a,err}^{(3,2)}\|_{L^\infty_\theta}+\|e^{\theta_0\eta_r(Y)^\f32}\partial_Y\psi_{a,err}^{(3,2)}\|_{L^\infty_\theta}\leq C\alpha^2(\|\psi_{a,err}^{(2)}\|_{L^\infty_\theta}+\|\psi_{a}^{(0)}\|_{L^\infty_\theta})\leq C|\varepsilon|^\f23\alpha^2.
\end{align*}
Then by defining $w_{a,err}:=w_{a,err}^{(1)}+w_{a,err}^{(2)}+w_{a,err}^{(3)}$ and $\psi_{a,err}=\psi_{a,err}^{(1)}+\psi_{a,err}^{(2)}+\psi_{a,err}^{(3,1)}+\psi_{a,err}^{(3,2)}$, we construct a solution to \eqref{eq:airy-homo-err}. Hence, $w_a:=w_a^{(0)}+w_{a,err}$ and $\psi_a:=\psi_a^{(0)}+\psi_{a,err}$ are a solution to  \eqref{eq:airy-homo} and 
\begin{align*}
	&\|(\psi_a-\psi_{a}^{(0)})\|_{L^\infty_\theta}\leq C|\varepsilon|^\f13(|c_ic|+|\varepsilon|^\f13\alpha^2),\quad \|(u-\hat{c})\partial_Y(\psi_a-\psi_a^{(0)})\|_{L^\infty}\leq C|\varepsilon|^\f13(|c_ic|+|\varepsilon|^\f13\alpha^2).
\end{align*}

The proof is completed.
\end{proof}

\section{The Orr-Sommerfeld equation}
This section is devoted to constructing the slow mode $\phi_s$ and the fast mode $\phi_f$ for the following  homogeneous Orr-Semmerfeld equation:
\begin{align}\label{eq:OS-homo}
\left\{\begin{aligned}
&\varepsilon(\partial_Y^2-\alpha^2)^2\phi-(u-c)(\partial_Y^2-\alpha^2)\phi+u''\phi=0,\\
&\lim_{Y\to\infty}\phi(Y)=0.
\end{aligned}
\right.
\end{align}
The leading order terms of the slow mode and the fast mode are the solution $\varphi_{Ray}$ to the modified homogeneous Rayleigh equation  and  the solution $\psi_a$ to the homogeneous Airy equation, respectively. We develop a modified Rayleigh-Airy iteration to construct the exact solutions from $\varphi_{Ray}$ and $\psi_a$ around the neutral curve. 

Let us introduce that following function spaces. We define $\hat c=c+\mathrm i|\varepsilon|\alpha^{-\f32}$ and 
\begin{align*}
\mathcal{X}:=\{f\in W^{4,\infty}|\|f\|_{\mathcal X}<+\infty\},
\end{align*}
where 
\begin{align*}
\|f\|_{\mathcal X}:=&\|\mathcal{W} \partial_Y^4f\|_{L^\infty_{\eta_0}}+|\varepsilon|^\f23\|(u-\hat c)\partial_Y^3f\|_{L^\infty_{\eta_0}}+\|(u-\hat c)^2\partial_Y^2f\|_{L^\infty_{\eta_0}}+\|(u-\hat c)\partial_Yf\|_{L^\infty_{\eta_0}}+\|f\|_{L^\infty_{\eta_0}}.
\end{align*}
Here $\mathcal W(Y)>0$ is a smooth weighted function defined such that for any $Y\in\mathcal N$, $\mathcal W(Y)\sim |\varepsilon|^\f13 |u-\hat c|^3$ and for any $Y\in\mathcal N^-\cup\mathcal N^+$, $\mathcal W(Y)\sim |\varepsilon||u-\hat c|$. We also define
\begin{align*}
\mathcal{Y}:=\{f\in W^{4,\infty}|\|f\|_{\mathcal Y}<+\infty\},
\end{align*}
where 
\begin{align*}
\|f\|_{\mathcal{Y}}=&\|(u-\hat  c)^3\pa_Y^4 f\|_{L^\infty_{\eta_0}}+\|(u-\hat  c)^2\pa_Y^3 f\|_{L^\infty_{\eta_0}}+\|(u-\hat c)\pa_Y^2 f\|_{L^\infty_{\eta_0}}+\|\pa_Y f\|_{L^\infty_{\eta_0}}+\| f\|_{L^\infty_{\eta_0}}.
\end{align*}

In the rest of this section, we always assume $(\al, c)\in\mathbb{H}_2\cap\{(\al,c):\al\sim c_r\}:=\mathbb{ H}_3$.

\subsection{Non-homogeneous OS equation}
In this part, we consider the following non-homogeneous OS equation for $(\al,c)\in\mathbb{ H}_3$:
\begin{align}\label{eq:OS-non-homo}
	\left\{\begin{aligned}
		&\varepsilon(\partial_Y^2-\alpha^2)^2\phi_{non}-(u-c)(\partial_Y^2-\alpha^2)\phi_{non}+u''\phi_{non}=f,\\
		&\lim_{Y\to\infty}\phi_{non}(Y)=0.
	\end{aligned}
	\right.
\end{align}
We develop a modified Rayleigh-Airy iteration to construct a solution of \eqref{eq:OS-non-homo}.
\begin{theorem}\label{thm:OS-non-homo}
Let  $\hat c=c+\mathrm i c_0$ with $|\varepsilon|^\f23\ll c_0\ll|\varepsilon|^\f13$ and $(\al,c)\in\mathbb{ H}_3$ with $c_i>-c_0/2$.  Then for any $e^{\eta_0 Y}f\in W^{2,\infty}$, there exists a solution $\phi_{non}\in\mathcal{X}$ to \eqref{eq:OS-non-homo} satisfying 
\begin{align*}
	\|\phi_{non}\|_{\mathcal{X}}\leq C|\log|\varepsilon||(\|(u-\hat c)^2\pa_Y^2f\|_{L^\infty_{\eta_0}}+\|(u-\hat c)\pa_Yf\|_{L^\infty_{\eta_0}}+\|f\|_{L^\infty_{\eta_0}}).
\end{align*}
\end{theorem}
\begin{proof}
	\textbf{The initial step of the iteration.}\smallskip
	
		We introduce that $\varphi^{(1)}$, as constructed in Proposition \ref{pro: phi_non}, is the solution to the following non-homogeneous Rayleigh equation
\begin{align*}
	Ray_{\hat c}[\varphi^{(1)}]=-f,\quad\varphi^{(1)}(\infty)=0.
\end{align*}
Then by Proposition \ref{pro: phi_non}, we have  
\begin{align}\label{eq:OS-non-homo-phi1}
	\|\varphi^{(1)}\|_{\mathcal{X}}\leq \|\varphi^{(1)}\|_{\mathcal{Y}}\leq C|\log|\varepsilon||(\|(u-\hat c)^2\pa_Y^2f\|_{L^\infty_{\eta_0}}+\|(u-\hat c)\pa_Yf\|_{L^\infty_{\eta_0}}+\|f\|_{L^\infty_{\eta_0}}),
\end{align}
and 
\begin{align}\label{eq:OS-non-homo-phi2}
	\|(u-\hat c)\partial_Y^2\varphi^{(1)}\|_{L^\infty_{\eta_0}}+\|\partial_Y\varphi^{(1)}\|_{L^\infty_{\eta_0}}+\|\varphi^{(1)}\|_{L^\infty_{\eta_0}}\leq C|\log|\varepsilon||\|f\|_{L^\infty_{\eta_0}}.
\end{align}
Moreover, we notice that 
\begin{align*}
	OS[\varphi^{(1)}]=&f+\varepsilon(\partial_Y^2-\al^2)^2\varphi^{(1)}-\mathrm{i}c_0(\pa_Y^2-\al^2)\varphi^{(1)}\\
	=&f+\varepsilon\pa_Y^2\Big(\frac{u''\varphi^{(1)}}{u-\hat c}-\frac{f}{u-\hat c}\Big)-(\mathrm i c_0+\varepsilon\al^2)\Big(\frac{u''\varphi^{(1)}}{u-\hat c}-\frac{f}{u-\hat c}\Big).
\end{align*}
Next we introduce $\psi^{(1)}=\psi_1^{(1)}+\psi_2^{(1)}+\psi_3^{(1)}$, where $\psi_j^{(1)}$ for $j=1,2,3$ satisfy 
\begin{align*}
	&Airy[(\partial_Y^2-\al^2)\psi_1^{(1)}]=-\varepsilon\pa_Y^2\Big[(1-\chi)\Big(\frac{u''\varphi^{(1)}}{u-\hat c}-\frac{f}{u-\hat c}\Big)\Big],\\
	&Airy[(\partial_Y^2-\al^2)\psi_2^{(1)}]=(\mathrm i c_0+\varepsilon\al^2)\Big(\frac{u''\varphi^{(1)}}{u-\hat c}-\frac{f}{u-\hat c}\Big),\\
	&Airy[(\partial_Y^2-\al^2)\psi_3^{(1)}]=-\varepsilon\pa_Y^2\Big[\chi \Big(\frac{u''\varphi^{(1)}}{u-\hat c}-\frac{f}{u-\hat c}\Big)\Big].
\end{align*}
Here $\chi(Y)$ is a smooth cut-off function with $\chi(Y)\equiv1$ on $[0,\f12]$ and $\chi(Y)\equiv0$ on $[1,+\infty)$. According to Theorem \ref{them:airy-gf-1} and \eqref{eq:OS-non-homo-phi2}, we obtain 
\begin{align*}
	&\|\psi_1^{(1)}\|_{\mathcal X}\leq C|\varepsilon||\log|\varepsilon||(\|\varphi^{(1)}\|_{L^\infty_{\eta_0}}+\|f\|_{L^\infty_{\eta_0}})\leq C|\varepsilon||\log|\varepsilon||^2\|f\|_{L^\infty_{\eta_0}},\\
	&\|\psi_2^{(1)}\|_{\mathcal X}\leq C|\varepsilon|^{-\f13}\al c_0|\log|\varepsilon||(\|\varphi^{(1)}\|_{L^\infty_{\eta_0}}+\|f\|_{L^\infty_{\eta_0}})\leq C|\varepsilon|^{-\f13}\al c_0|\log|\varepsilon||^2\|f\|_{L^\infty_{\eta_0}}.
\end{align*}
By Theorem \ref{them:airy-sf} and \eqref{eq:OS-non-homo-phi1}, we have  
\begin{align*}
	\|\psi_3^{(1)}\|_{\mathcal{X}}\leq& C(\al+|\varepsilon|^\f23|\hat c_i|^{-1})|\log|\varepsilon||\Big\|(u-\hat c)^3\pa_Y^2\Big[\chi \Big(\frac{u''\varphi^{(1)}}{u-\hat c}-\frac{f}{u-\hat c}\Big)\Big]\Big\|_{L^\infty}\\
	&+C(\al+|\varepsilon|^\f23|\hat c_i|^{-1})|\log|\varepsilon||\Big\|(u-\hat c)^2\pa_Y\Big[\chi \Big(\frac{u''\varphi^{(1)}}{u-\hat c}-\frac{f}{u-\hat c}\Big)\Big]\Big\|_{L^\infty}\\
	&+C(\al+|\varepsilon|^\f23|\hat c_i|^{-1})|\log|\varepsilon||\Big\|(u-\hat c)\chi \Big(\frac{u''\varphi^{(1)}}{u-\hat c}-\frac{f}{u-\hat c}\Big)\Big\|_{L^\infty}\\
	\leq& C(\al+|\varepsilon|^\f23|\hat c_i|^{-1})|\log|\varepsilon||^2(\|(u-\hat c)^2\pa_Y^2f\|_{L^\infty_{\eta_0}}+\|(u-\hat c)\pa_Yf\|_{L^\infty_{\eta_0}}+\|f\|_{L^\infty_{\eta_0}}),
\end{align*}
and 
\begin{align*}
&\|(u-\hat c)^2\pa_Y^2\psi_3^{(1)}\|_{L^\infty_{\eta_0}}+\|(u-\hat c)\pa_Y\psi_3^{(1)}\|_{L^\infty_{\eta_0}}+\|\psi_3^{(1)}\|_{L^\infty_{\eta_0}}	\\
&\leq C\al|\log|\varepsilon||^2(\|(u-\hat c)^2\pa_Y^2f\|_{L^\infty_{\eta_0}}+\|(u-\hat c)\pa_Yf\|_{L^\infty_{\eta_0}}+\|f\|_{L^\infty_{\eta_0}}).
\end{align*}
Therefore, we obtain 
\begin{align}
	\|\psi^{(1)}\|_{\mathcal{X}}\leq C(\al+|\varepsilon|^\f23|\hat c_i|^{-1})|\log|\varepsilon||^2(\|(u-\hat c)^2\pa_Y^2f\|_{L^\infty_{\eta_0}}+\|(u-\hat c)\pa_Yf\|_{L^\infty_{\eta_0}}+\|f\|_{L^\infty_{\eta_0}}),
\end{align}
and 
\begin{align}
&\|(u-\hat c)^2\pa_Y^2\psi^{(1)}\|_{L^\infty_{\eta_0}}+\|(u-\hat c)\pa_Y\psi^{(1)}\|_{L^\infty_{\eta_0}}+\|\psi^{(1)}\|_{L^\infty_{\eta_0}}\nonumber	\\
&\leq C\al|\log|\varepsilon||^2(\|(u-\hat c)^2\pa_Y^2f\|_{L^\infty_{\eta_0}}+\|(u-\hat c)\pa_Yf\|_{L^\infty_{\eta_0}}+\|f\|_{L^\infty_{\eta_0}}).
\end{align}
We define $\phi^{(1)}=\varphi^{(1)}+\psi^{(1)}$, which satisfies
\begin{align}
	OS[\phi^{(1)}]=f+u''\psi^{(1)}.
\end{align}
\textbf{The iteration scheme.} For any $j\geq 2$, we define $\phi^{(j)}=\varphi^{(j)}+\psi^{(j)}$, where $\varphi^{(j)}$ is the solution to the following non-homogeneous Rayleigh equation,
\begin{align*}
	Ray_{\hat c}[\varphi^{(j)}]=u''\psi^{(j-1)},\quad\varphi^{(j)}(\infty)=0,
\end{align*}
and $\psi^{(j)}=\psi_1^{(j)}+\psi_2^{(j)}+\psi_3^{(j)}$ with $\psi_1^{(j)}, \psi_2^{(j)},\psi_3^{(j)}$ satisfying the following non-homogeneous Airy equations, respectively
\begin{align}\label{eq:r-a-airy}
	\begin{split}
		&Airy[(\partial_Y^2-\al^2)\psi_1^{(j)}]=-\varepsilon\pa_Y^2\Big[(1-\chi)\Big(\frac{u''\varphi^{(j)}}{u-\hat c}-\frac{u''\psi^{(j-1)}}{u-\hat c}\Big)\Big],\\
		&Airy[(\partial_Y^2-\al^2)\psi_2^{(j)}]=(\mathrm i c_0+\varepsilon\al^2)\Big(\frac{u''\varphi^{(j)}}{u-\hat c}-\frac{u''\psi^{(j-1)}}{u-\hat c}\Big),\\
		&Airy[(\partial_Y^2-\al^2)\psi_3^{(j)}]=-\varepsilon\pa_Y^2\Big[\chi \Big(\frac{u''\varphi^{(j)}}{u-\hat c}-\frac{u''\psi^{(j-1)}}{u-\hat c}\Big)\Big].
	\end{split}
\end{align}
By Proposition \ref{pro: phi_non}, Theorem \ref{them:airy-sf}, Theorem \ref{them:airy-gf-1} and a similar argument as above, we can obtain  
\begin{align}\label{eq:OS-non-homo-j1}
	\begin{split}
		&\|\varphi^{(j)}\|_{\mathcal X}\leq C|\log|\varepsilon||(\|(u-\hat c)^2\pa_Y^2\psi^{(j-1)}\|_{L^\infty_{\eta_0}}+\|(u-\hat c)\pa_Y\psi^{(j-1)}\|_{L^\infty_{\eta_0}}+\|\psi^{(j-1)}\|_{L^\infty_{\eta_0}}),\\
	&\|\psi^{(j)}\|_{\mathcal X}\leq C(\al+|\varepsilon|^\f23|\hat c_i|^{-1})|\log|\varepsilon||^2(\|(u-\hat c)^2\pa_Y^2\psi^{(j-1)}\|_{L^\infty_{\eta_0}}+\|(u-\hat c)\pa_Y\psi^{(j-1)}\|_{L^\infty_{\eta_0}}\\
	&\qquad\qquad\quad+\|\psi^{(j-1)}\|_{L^\infty_{\eta_0}}),
	\end{split}
\end{align}
and 
\begin{align}\label{eq:OS-non-homo-j}
	\begin{split}
		&\|(u-\hat c)^2\pa_Y^2\psi^{(j)}\|_{L^\infty_{\eta_0}}+\|(u-\hat c)\pa_Y\psi^{(j)}\|_{L^\infty_{\eta_0}}+\|\psi^{(j)}\|_{L^\infty_{\eta_0}}	\\
&\leq C\al|\log|\varepsilon||^2(\|(u-\hat c)^2\pa_Y^2\psi^{(j-1)}\|_{L^\infty_{\eta_0}}+\|(u-\hat c)\pa_Y\psi^{(j-1)}\|_{L^\infty_{\eta_0}}+\|\psi^{(j-1)}\|_{L^\infty_{\eta_0}}).
	\end{split}
\end{align}
Moreover, we have 
\begin{align*}
	OS[\phi^{(j)}]=-u''\psi^{(j-1)}+u''\psi^{(j)}.
\end{align*}
\textbf{The construction of $\phi_{non}$.} We define $\phi_{non}=\sum_{j=1}^{+\infty}\phi^{(j)}$, which satisfies $OS[\phi_{non}]=f$ formally. Hence, we are left with the proof of the convergence of the series $\sum_{j=1}^{+\infty}\phi^{(j)}$. By \eqref{eq:OS-non-homo-j}, we obtain  
\begin{align*}
&\|(u-\hat c)^2\pa_Y^2\psi^{(j)}\|_{L^\infty_{\eta_0}}+\|(u-\hat c)\pa_Y\psi^{(j)}\|_{L^\infty_{\eta_0}}+\|\psi^{(j)}\|_{L^\infty_{\eta_0}}\\
&\leq C(\al|\log|\varepsilon||^2)^j(\|(u-\hat c)^2\pa_Y^2f\|_{L^\infty_{\eta_0}}+\|(u-\hat c)\pa_Yf\|_{L^\infty_{\eta_0}}+\|f\|_{L^\infty_{\eta_0}}),
\end{align*}
which along with \eqref{eq:OS-non-homo-j1} deduces 
\begin{align*}
	\|\phi^{(j)}\|_{\mathcal X}\leq C(\al|\log|\varepsilon||^2)^{j-1}|\log|\varepsilon||(\|(u-\hat c)^2\pa_Y^2f\|_{L^\infty_{\eta_0}}+\|(u-\hat c)\pa_Yf\|_{L^\infty_{\eta_0}}+\|f\|_{L^\infty_{\eta_0}}).
\end{align*}
Therefore, we have 
\begin{align*}
	\|\phi_{non}\|_{\mathcal X}\leq& C|\log|\varepsilon||(\sum_{j=1}^{+\infty}(\al|\log|\varepsilon||^2)^{j})(\|(u-\hat c)^2\pa_Y^2f\|_{L^\infty_{\eta_0}}+\|(u-\hat c)\pa_Yf\|_{L^\infty_{\eta_0}}+\|f\|_{L^\infty_{\eta_0}})\\
	\leq&C|\log|\varepsilon||(\|(u-\hat c)^2\pa_Y^2f\|_{L^\infty_{\eta_0}}+\|(u-\hat c)\pa_Yf\|_{L^\infty_{\eta_0}}+\|f\|_{L^\infty_{\eta_0}}).
\end{align*}

The proof is completed.	
\end{proof}

\subsection{The slow mode}
In this part, we construct the slow mode of \eqref{eq:OS-homo}. Let $\varphi_{Ray}(Y)$ be the solution to the homogeneous Rayleigh equation as constructed in Proposition \ref{prop:ray-homo}, with $c$ replaced by $\hat c:=c+\mathrm i |\varepsilon|\alpha^{-\f32}$. That is,
\begin{align}\label{eq:ray-hat-c}
(u-\hat  c)(\partial_Y^2-\alpha^2)\varphi_{Ray}-u''\varphi_{Ray}=0,\quad\varphi_{Ray}(\infty)=0.
\end{align} 
Hence, we have 
\begin{align*}
	OS[\varphi_{Ray}](Y)
	=&\varepsilon(\partial_Y^2-\alpha^2)^2\varphi_{Ray} -\mathrm i|\varepsilon|\alpha^{-\f32}(\partial_Y^2-\alpha^2)\varphi_{Ray}\\
	=&\varepsilon\partial_Y^2(u''(u-\hat c)^{-1}\varphi_{Ray})-(\varepsilon\alpha^2+\mathrm i|\varepsilon|\alpha^{-\f32})u''(u-\hat  c)^{-1}\varphi_{Ray}.
\end{align*}
By Proposition \ref{prop:ray-homo}, we know that 
\begin{align*}
\varphi_{Ray}=(u-\hat  c)e^{-\alpha Y}+\tilde{\varphi}_{R}(Y),\quad\|\tilde{\varphi}_{R}\|_{\mathcal Y}\leq C\alpha|\log\hat  c_i|.
\end{align*}
As a consequence,  there exist $H_1(Y)\in L^\infty, H_2(Y)\in W^{2,\infty}$ such that 
\begin{align}
\frac{u''(Y)\varphi_{Ray}(Y)}{u(Y)-\hat  c}=H_1(Y)+H_2(Y),
\end{align}
and 
\begin{align}\label{eq:os-ray-err}
\begin{split}
&\|H_1(Y)\|_{L^\infty_{\eta_0}}\leq C,\quad\|H_1-u''e^{-\alpha Y}\|_{L^\infty_{\eta_0}}+\|H_2-\frac{\al u''}{u'(Y_c)(u-\hat  c)}\|_{L^\infty_{\eta_0}}\leq C\alpha|\log\hat  c_i|.\\
&\|(u-\hat  c)H_2\|_{L^\infty_{\eta_0}}+\|(u-\hat  c)^2\partial_Y H_2\|_{L^\infty_{\eta_0}}+\|(u-\hat  c)^3\partial_Y^2 H_2\|_{L^\infty_{\eta_0}}\leq C\alpha|\log\hat  c_i|.
\end{split}
\end{align}
Then we can write 
\begin{align*}
	OS[\varphi_{Ray}](Y)=\varepsilon\partial_Y^2H_1(Y)+\varepsilon\partial_Y^2H_2(Y)-(\mathrm i|\varepsilon|\alpha^{-\f32}+\varepsilon\alpha^2)(H_1+H_2).
\end{align*}

\begin{theorem}\label{them:slow-mode}
Let $(\al,c)\in\mathbb{ H}_3$ with $c_i>-|\varepsilon|\alpha^{-\f32}/ 2$.  Then there exists a solution $\phi_s$  to \eqref{eq:OS-homo} such that $e^{\alpha Y}\phi_s(Y)\in W^{4,\infty}(\mathbb R_+)$ and 
\begin{align*}
	&\phi_s(0)=-c+u'(0)^{-1}\alpha+\mathcal O(|\varepsilon|\alpha^{-\f32}+\alpha^2)|\log|\varepsilon||^3,\\
	&\partial_Y\phi_s(0)= u'(0)+\mathcal{O}(|\al|| \log |\varepsilon| |).
\end{align*}
Moreover, if $\alpha\gg|\varepsilon|^\f14$, we have
\begin{align*}
	&\mathrm{Im}(\phi_s(0))=- c_i+\frac{\al c_r u''(Y_c)\pi}{u'(Y_c)^{3}}+o(\alpha^2),\quad\mathrm{Im}(\partial_Y\phi_s(0))=-\frac{\al u''(Y_c)\pi}{u'(Y_c)^{2}}+o(\alpha).
\end{align*}

\end{theorem}

\begin{proof}
Let $\varphi_{Ray}$ be the solution to \eqref{eq:ray-hat-c}. We define $\psi^{(1)}_s=\psi_{1,s}^{(1)}+\psi_{2,s}^{(1)}+\psi_{3,s}^{(1)}$, where $\psi_{1,s}^{(1)},\psi_{2,s}^{(1)}$ and $\psi_{3,s}^{(1)}$ satisfy the following equations, respectively
\begin{align}
	\begin{split}
		&Airy[(\partial_Y^2-\alpha^2)\psi_{1,s}^{(1)}]=-\varepsilon\partial_Y^2H_1-\varepsilon(\alpha^{-\f32}-\alpha^2)H_1-\varepsilon\partial_Y^2[(1-\chi)H_2]\\
&Airy[(\partial_Y^2-\alpha^2)\psi_{2,s}^{(1)}]=-\varepsilon(\alpha^{-\f32}-\alpha^2)H_2,\quad Airy[(\partial_Y^2-\alpha^2)\psi_{3,s}^{(1)}]=-\varepsilon\partial_Y^2(\chi H_2).
	\end{split}
\end{align}
Here $\chi(Y)$ is a smooth cut-off function with $\chi(Y)\equiv1$ on $[0,\f12]$ and $\chi(Y)\equiv0$ on $[1,+\infty)$. 
Hence, by \eqref{eq:os-ray-err} 
\begin{align*}
\|\varepsilon\partial_Y^2H_1+\varepsilon(\alpha^{-\f32}-\alpha^2)H_1+\varepsilon\partial_Y^2[(1-\chi)H_2]\|_{L^\infty_{\eta_0}}\leq C|\varepsilon|\alpha^{-\f32},
\end{align*} 
which along with Theorem \ref{them:airy-gf-1} implies that 
\begin{align}\label{eq:OS-slow-airy1}
\|\psi_{1,s}^{(1)}\|_{\mathcal X}\leq C|\varepsilon|\alpha^{-\f32}|\log|\varepsilon||.
\end{align}
In particular,
\begin{align}\label{eq:slow-airy1-bv}
	|(u-\hat c)\partial_Y\psi^{(1)}_{1,s}(0)|+|\psi_{1,s}^{(1)}(0)|\leq C|\varepsilon|\alpha^{-\f32}|\log|\varepsilon||.
\end{align}
Again by \eqref{eq:os-ray-err}, we have 
\begin{align*}
\|(u-\hat c)\varepsilon(\alpha^{-\f32}-\alpha^2)H_2\|_{L^\infty_{\eta_0}}\leq C|\varepsilon|^\f56|\log|\hat c_i||,
\end{align*}
which along with Theorem \ref{them:airy-gf-1} implies that 
\begin{align}\label{eq:OS-slow-airy2}
\|\psi_{2,s}^{(1)}\|_{\mathcal X}\leq C|\varepsilon|^\f12\al|\log|\varepsilon||^3.
\end{align}
By the definition of $\mathbb H_3$, \eqref{eq:os-ray-err} and Theorem \ref{them:airy-sf}, we obtain 
\begin{align}\label{eq:slow-mode-psi1-e1}
\begin{split}
&|\varepsilon|^{\f23}\|(u-\hat c)\partial^3_Y\psi_{3,s}^{(1)}\|_{L^\infty_{\eta_0}}+\|\mathcal{W}(Y)\partial^4_Y\psi_{3,s}^{(1)}\|_{L^\infty_{\eta_0}}\\
&\leq  C\alpha|\log|\varepsilon||^2(\alpha^\f23|\varepsilon|^{-\f13}+|\varepsilon|^\f13+\alpha^2|\varepsilon|^{-\f13})\leq C\alpha^{\f52}|\varepsilon|^{-\f13}|\log|\varepsilon||,\\
&\|(u-\hat c)^2\partial_Y^2\psi_{3,s}^{(1)}\|_{L^\infty_{\eta_0}}\leq C|\varepsilon|^\f13\alpha|\log|\varepsilon||,  \\
&\|\psi_{3,s}^{(1)}\|_{L^\infty_{\eta_0}}+\|(u-\hat{c})\partial_Y\psi_{3,s}^{(1)}\|_{L^\infty}\leq C\alpha^2|\log|\varepsilon||^2.
\end{split}
\end{align}
Hence, we have 
\begin{align*}
\|\psi_{3,s}^{(1)}\|_{\mathcal X}\leq \alpha^\f52|\varepsilon|^{-\f13}|\log|\varepsilon||.
\end{align*}
Thus, we obtain 
\begin{align}\label{eq:slow-mode-airy-x}
\begin{split}
	&\|\psi^{(1)}_s\|_{\mathcal X}\leq C(|\varepsilon|\alpha^{-\f32}+\alpha^\f52|\varepsilon|^{-\f13})|\log|\varepsilon||,\\
	&\|(u-\hat{c})^2\partial^2_Y\psi_s^{(1)}\|_{L^\infty_{\eta_0}}\leq C(|\varepsilon|\alpha^{-\f32}+|\varepsilon|^\f13\alpha)|\log|\varepsilon||, \\
	&\|\psi^{(1)}_s\|_{L^\infty_{\eta_0}}+\|(u-\hat{c})\partial_Y\psi_s^{(1)}\|_{L^\infty_{\eta_0}}\leq C(|\varepsilon|\alpha^{-\f32}+\alpha^2)|\log|\varepsilon||.
\end{split}
\end{align}
From \eqref{eq:slow-mode-airy-x},  we infer that 
\begin{align*}
&|\psi^{(1)}_s(0)|\leq C(|\varepsilon|\alpha^{-\f32}+\alpha^2)|\log|\varepsilon||,\quad|\partial_Y\psi^{(1)}_s(0)|\leq C\alpha|\log|\varepsilon||.
\end{align*}
Moreover, we have 
\begin{align*}
	OS[\varphi_{Ray}+\psi_s^{(1)}]=u''\psi_s^{(1)}.
\end{align*}
Next we define $\tilde\phi_{s}\in \mathcal X$ to be the solution to $OS[\tilde\phi_s]=-u''\psi^{(1)}_s$. Then by Theorem \ref{thm:OS-non-homo} and \eqref{eq:slow-mode-airy-x}, 
we can obtain  
\begin{align*}
	\|\tilde\phi_s\|_{\mathcal X}\leq& C|\log|\varepsilon||(\|\psi^{(1)}_s\|_{L^\infty_{\eta_0}}+\|(u-\hat{c})\partial_Y\psi^{(1)}_s\|_{L^\infty_{\eta_0}}+\|(u-\hat{c})^2\partial^2_Y\psi^{(1)}_s\|_{L^\infty_{\eta_0}})\\
	\leq &C(|\varepsilon|\alpha^{-\f32}+\alpha^2)|\log|\varepsilon||^2.
\end{align*}
We finish our construction by defining $\phi_s=\varphi_{Ray}+\psi_s^{(1)}+\tilde\phi_s$ and applying Proposition \ref{prop:ray-homo} with $c$ replaced by $\hat c$.

To construct the neutral stable mode, we need to obtain the asymptotic expansion of $\phi_s(0)$ up to the order $\alpha^2$. Hence, the above bounds of $\psi_s^{(1)}(0)$ are not sharp enough. In the rest of the proof, we shall provide the exact expansion of $\psi^{(1)}_s(0)$ up to the order of $\alpha^2$ for the case $\alpha\gg|\varepsilon|^\f14$.

By \eqref{eq:slow-airy1-bv}, \eqref{eq:OS-slow-airy2}, \eqref{eq:slow-mode-psi1-e1} and Theorem \ref{them:airy-sf},  we have  
\begin{align}\label{eq:slow-up-psi}
	\begin{split}
			&|\psi_s^{(1)}(0)-u'(0)^{-1}c_r\mathcal M_{0,s}|\leq C(|\varepsilon|\alpha^{-\f32}+\alpha|\varepsilon|^\f13)|\log|\varepsilon||^2\ll \alpha^2,\\
	&|\psi_s^{(1)}(0)+\mathcal M_{0,s}|\leq C(|\varepsilon|\alpha^{-\f52}+|\varepsilon|^\f13)|\log|\varepsilon||^2\leq C(|\varepsilon|^\f38+|\varepsilon|^\f13)|\log|\varepsilon||^2\ll\alpha,
	\end{split}
\end{align}
where $\mathcal M_{0,s}$ is defined as \eqref{eq:def-m-0s}.

By Proposition \ref{prop:ray-homo}, Lemma \ref{lem:m0s} and \eqref{eq:slow-up-psi}  we can obtain that  
\begin{align*}
	&\mathrm{Im}(\varphi_{Ray}(0)+\psi_s^{(1)}(0))=- c_i+\frac{\al c_r u''(Y_c)\pi}{u'(Y_c)^{3}}+o(\alpha^2),\\
	&\mathrm{Im}(\partial_Y(\varphi_{Ray}+\psi_s^{(1)})(0))=-\frac{\al u''(Y_c)\pi}{u'(Y_c)^{2}}+o(\alpha).
\end{align*}
Now we turn to show that $|\tilde\phi_s(0)|+\alpha|\partial_Y\tilde\phi_s(0)|\ll\alpha^2$. We notice that by \eqref{eq:OS-slow-airy1}, \eqref{eq:OS-slow-airy2}, \eqref{eq:slow-mode-psi1-e1} and the iteration scheme of $\tilde\phi_s$ introduced in the proof of Theorem \ref{thm:OS-non-homo},
\begin{align*}
	\|\tilde\phi_s-\tilde\varphi_{s}\|_{\mathcal X}\ll\alpha^2, \text{ where }Ray_{\hat c}[\tilde\varphi_{s}]=u''\psi_{s}^{(1)},
\end{align*}
which implies that 
\begin{align}\label{eq:slow-up-t-dis}
|(\tilde\phi_s-\tilde\varphi_s)(0)|+\al|\partial_Y(\tilde\phi_s-\tilde\varphi_s)(0)|\ll\alpha^2	
\end{align}
By Proposition \ref{pro: phi_non}, we can obtain
\begin{align}\label{eq:slow-up-p-phi}
	|\partial_Y\tilde\varphi_s(0)|\leq C\alpha^2|\log|\varepsilon||^3\ll\alpha.
\end{align}
To obtain the estimates for $\tilde\varphi_s(0)$, we decompose $\tilde\varphi_s=\tilde\varphi_1+\tilde\varphi_2$, where
\begin{align*}
	\tilde\varphi_1(Y)=(u-\hat c)\int_Y^{+\infty}(u(Y')-\hat c)^{-2}\int_{Y'}^{+\infty}u''(Z)\psi_{s}^{(1)}(Z)dZdY',\quad Ray_{\hat c}[\tilde\varphi_2]=\alpha^2(u-\hat c)\tilde\varphi_1.
\end{align*}
We get by integration by parts that 
\begin{align*}
\tilde\varphi_1(0)=u'(Y_c)^{-1}\int_0^{+\infty}u''(Z)\psi_{s}^{(1)}(Z)dZ+\tilde {\mathcal R},
\end{align*}
where
\begin{align*}
	\tilde {\mathcal R}=&\frac{\hat c}{u'(Y_c)(u(Y_0)-\hat c)}\int_{Y_0}^{+\infty}u''\psi_{s}^{(1)}dZ+\frac{\hat c}{u'(Y_c)}\int_0^{Y_0}\frac{u''\psi_{s}^{(1)}}{u-\hat c}dZ\\
	&-\hat c\int_{Y_0}^{+\infty}(u(Y)-\hat c)^{-2}\int_Y^{+\infty}u''\psi_{s}^{(1)}dZdY+\frac{\hat c}{u'(Y_c)}\int_0^{Y_0}\frac{u'(Y)-u'(Y_c)}{(u-\hat c)^2}\int_Y^{+\infty}u''\psi_{s}^{(1)}dZdY.
\end{align*}
We notice that by integration by parts,
\begin{align*}
	\int_0^{+\infty}u''(Z)\psi_{s}^{(1)}(Z)dZ=-u'(0)\psi_s^{(1)}(0)-c\partial_Y\psi_s^{(1)}(0)+\int_0^{+\infty}(u-c)\partial_Y^2\psi_s^{(1)}dY,
\end{align*}
which along with \eqref{eq:slow-mode-airy-x} and \eqref{eq:slow-up-psi} implies that 
\begin{align*}
	&\Big|u'(Y_c)^{-1}\int_0^{+\infty}u''(Z)\psi_{s}^{(1)}(Z)dZ\Big|\\
	&\leq  C(|\varepsilon|\alpha^{-\f32}+\alpha|\varepsilon|^\f13)|\log|\varepsilon||^2+C\|(u-\hat c)^2\partial_Y^2\psi_s^{(1)}\|_{L^\infty_{\eta_0}}\int_0^{+\infty}e^{-\eta_0 Z}|u-\hat c|^{-1}dZ\\
	&\leq C(|\varepsilon|\alpha^{-\f32}+\alpha|\varepsilon|^\f13)|\log|\varepsilon||^3\ll\alpha^2.
\end{align*}
By \eqref{eq:slow-mode-airy-x}, we have
\begin{align*}
	|\tilde{\mathcal R}|\leq C |\hat c|\|\psi_s^{(1)}\|_{L^\infty}+C|\hat c|\|\psi_s^{(1)}\|_{L^\infty}\int_0^{Y_0}|u-\hat c|^{-1}dZ\leq C\alpha^3|\log|\varepsilon||^3\ll\alpha^2.
\end{align*}
Therefore, we obtain 
\begin{align}\label{eq:slow-up-t-phi1}
	|\tilde \varphi_1(0)|\ll\alpha^2.
\end{align}
Using the bounds of $\varphi_{non}^{(0)}$ from the proof of Proposition 3.2 in \cite{MWWZ} for the case $m=0$, we can obtain 
\begin{align*}
	\|\tilde\varphi_1\|_{L^\infty_{\eta_0}}\leq C|\log|\varepsilon||\|\psi_{3,s}^{(1)}\|_{L^\infty}\leq C\alpha^2|\log|\varepsilon||^3.
\end{align*}
From the above bounds and Proposition \ref{pro: phi_non}, we infer that 
\begin{align}\label{eq:slow-up-t-phi2}
	|\tilde\varphi_2(0)|\leq \|\tilde\varphi_2\|_{L^\infty_{\eta_0}}\leq C\alpha^4|\log|\varepsilon||^4\ll\alpha^2.
\end{align}
Gathering \eqref{eq:slow-up-t-phi1} and \eqref{eq:slow-up-t-phi2}, we conclude  
\begin{align*}
	|\tilde\varphi_s(0)|\ll\alpha^2,
\end{align*}
which along with \eqref{eq:slow-up-t-dis} and \eqref{eq:slow-up-p-phi} deduces that 
\begin{align*}
	|\tilde\phi_s(0)|+\alpha|\partial_Y\tilde\phi_s(0)|\ll\alpha^2.
\end{align*}

The proof is completed.
      \end{proof}

We define 
\begin{align}\label{eq:def-m-0s}
	\mathcal M_{0,s}:=\mathcal M_0[-\varepsilon\chi H_2]=\int_0^{+\infty}\mathcal H_1(Z)\partial_Z\eta(Z)^{-1}(-\varepsilon\chi H_2)(Z)dZ,
\end{align}
where $\mathcal H_1$ is defined in \eqref{eq:airy-app-green-H}
\begin{lemma}\label{lem:m0s}
	Let  $\hat c:=c+\mathrm i |\varepsilon|\alpha^{-\f32}$.  Suppose $(\al,c)\in\mathbb{ H}_3$ with $c_i>-|\varepsilon|\alpha^{-\f32}/ 2$ and $\alpha\gg|\varepsilon|^\f14$. Then we have 
\begin{align*}
&|\mathrm{Re}\left(\mathcal{M}_{0,s}\right)|\leq C\alpha|\log|\varepsilon|| \text{ and }\mathrm{Im}(\mathcal{M}_{0,s})= \f{\alpha u''(Y_c)}{u'(Y_c)^2}\arg(-\hat c)+\frac{\al u''(Y_c)\pi}{u'(Y_c)^{2}}+o(\alpha).
\end{align*}
\end{lemma}
\begin{proof}
We can write 
\begin{align}\label{eq:slow-mode-m0s}
\begin{split}
\mathcal{M}_{0,s}=&\underbrace{-\frac{\alpha\varepsilon }{u'(Y_c)}\int_0^{+\infty}\mathcal H_1(Z)u''(Z)\partial_Z\eta(Z)^{-1}(u(Z)-\hat c)^{-1}\chi(Z)dZ}_{\mathcal M_{1,s}}\\
&\underbrace{-\varepsilon\int_0^{+\infty}\mathcal H_1(Z)\partial_Z\eta(Z)^{-1}\chi(Z) H_{2,err}(Z)dZ}_{\mathcal M_{2,s}} ,
\end{split}
\end{align}
where 
\begin{align*}
H_{2,err}(Z):=H_2(Z)-\frac{\alpha u''(Y)}{2u'(Y_c)(u-\hat c)}.
\end{align*}
By Lemma \ref{lem:m0}, we first have 
\begin{align}\label{eq:m0s-main}
\begin{split}
&|\mathcal M_{1,s}|\leq C\alpha|\log|\varepsilon||,\quad\mathrm{Im}(\mathcal M_{1,s})=\f{\alpha u''(Y_c)}{u'(Y_c)^2}\arg(-\hat c)+\frac{\al u''(Y_c)\pi}{u'(Y_c)^{\f43}}+o(\alpha) .
\end{split}
\end{align}
We denote that $\tilde Y_0\geq Y_c$ such that $\tilde\eta_r(Y_0)=c_r$, and we write 
\begin{align*}
&\mathcal M_{2,s}=\Big(\int_0^{Y_0}+\int_{Y_0}^{+\infty}\Big)\mathcal H_1(Z)\partial_Z\eta^{-1}(\varepsilon\chi(Z) H_{2,err}(Z))dZ
\end{align*}
 By Lemma \ref{lem:A1A2} and \eqref{eq:os-ray-err}, we have
\begin{align*}
&\left|\int_0^{Y_0}\mathcal H_1(Z)\partial_Z\eta(Z)^{-1}(\varepsilon\chi(Z) H_{2,err}(Z))dZ\right|\leq C\alpha|\log|\varepsilon||\int_0^{Y_0}dZ\leq C\alpha^2|\log|\varepsilon||\ll\alpha.
\end{align*}
On the other hand, by a similar argument for $\mathcal{M}_3^{(1)}$ as in the proof of Proposition \ref{prop:airy-green-Y}, we have
\begin{align*}
&\left|\int_{Y_0}^{+\infty}\mathcal H_1(Z)\partial_Z\eta(Z)^{-1}(\varepsilon\chi(Z) H_{2,err}(Z))dZ\right|\leq C|\varepsilon|\alpha^{-\f12}|\log|\varepsilon||\ll\al.
\end{align*}
Hence, we obtain
\begin{align*}
&\left|\mathcal M_{2,s}\right|\leq C|\varepsilon|^\f13,
\end{align*}
which along with \eqref{eq:slow-mode-m0s} and \eqref{eq:m0s-main} implies that 
\begin{align*}
&|\mathrm{Re}\left(\mathcal{M}_{0,s}\right)|\leq C\alpha|\log|\varepsilon|| \text{ and }\mathrm{Im}(\mathcal{M}_{0,s})= \f{\alpha u''(Y_c)}{u'(Y_c)^2}\arg(-\hat c)+\frac{\al u''(Y_c)\pi}{u'(Y_c)}+o(\alpha).
\end{align*}

The proof is completed.
\end{proof}

\begin{lemma}\label{lem:m0}
Let  $\hat c:=c+\mathrm i |\varepsilon|\alpha^{-\f32}$.  Suppose $(\al,c)\in\mathbb{ H}_3$ with $c_i>-|\varepsilon|\alpha^{-\f32}/ 2$ and $\alpha\gg|\varepsilon|^\f14$. Then we have
\begin{align*}
	&\left|\mathrm{Re}\left(\varepsilon\int_0^{+\infty}\mathcal H_1(Z)\f{(\partial_Z^2u)\chi(Z)}{(u-\hat c)\partial_Z\eta}dZ\right)\right|\leq C|\log|\varepsilon||,\\
	&\mathrm{Im}\left(\varepsilon\int_0^{+\infty}\mathcal H_1(Z)\f{(\partial_Z^2u)\chi(Z)}{(u-\hat c)\partial_Z\eta}dZ\right)=-\f{u''(Y_c)}{u'(Y_c)}\arg(-\hat c)-\frac{u''(Y_c)\pi}{u'(Y_c)}+\mathcal O(|\varepsilon|^{\frac{1}{24}}+c_r^2|\varepsilon|^{-\f13}),
\end{align*}
where $\chi(Y)\equiv1$ on $[0,\f12]$ and $\chi(Y)\equiv0$ on $[1,+\infty)$.
\end{lemma}

\begin{proof}
We denote $\mathcal H_1(Z)=\mathcal H_1^{(1)}(Z)+\mathcal H_1^{(2)}(Z)$ with 
\begin{align*}
	&\mathcal H_1^{(1)}(Z):=\partial_Y A_1A_2-\partial_Z A_2A_1,\quad\mathcal H_1^{(2)}(Z):=A_2(1,Z)\partial_Z^2A_1-A_1(1,Z)\partial_Z^2A_2.
\end{align*}
Then we can write 
\begin{align*}
	&\varepsilon\int_0^{+\infty}\mathcal H_1(Z)\f{(\partial_Z^2u)\chi(Z)}{(u-\hat c)\partial_Z\eta}dZ\\
	&\quad=\varepsilon\int_0^{+\infty}\mathcal H_1^{(1)}(Z)\f{(\partial_Z^2u)\chi(Z)}{(u-\hat c)\partial_Z\eta}d+\varepsilon\int_0^{+\infty}\mathcal H_1^{(2)}(Z)\f{(\partial_Z^2u)\chi(Z)}{(u-\hat c)\partial_Z\eta}dZ:=\mathcal I_1+\mathcal I_2.
\end{align*}
By \eqref{eq:Airy-wron}, we obtain that for any $Z\geq 0$, $\mathcal H_1^{(1)}(Z)\partial_Y\eta(Z)^{-1}\equiv\varepsilon^{-1}$, 
which implies that 
\begin{align}\label{eq:H11-1}
	\mathcal I_1=\int_0^{+\infty}u''(Z)(u(Z)-\hat c)^{-1}\chi(Z)dZ.
\end{align}
On the other hand, we know that 
\begin{align}\label{eq:H11-2}
	\Big|\int_L^{+\infty}u''(Z)(u(Z)-\hat c)^{-1}\chi(Z)dZ\Big|+\hat c^{-1}_i\Big|\mathrm{Im}\Big(\int_L^{+\infty}u''(Z)(u(Z)-\hat c)^{-1}\chi(Z)dZ\Big)\Big|\leq C.
\end{align}
Hence, we only need to show the estimates for $\int_0^{L}u''(Z)(u(Z)-\hat c)^{-1}\chi(Z)dZ$. Notice that 
\begin{align*}
\int_0^{L}u''(Z)(u(Z)-\hat c)^{-1}\chi(Z)dZ=-\f{u''(Y_c)}{u'(Y_c)}\log(-\hat c)+\mathcal{I}_{Err,1},
\end{align*}
where
\begin{align*}
\mathcal{I}_{Err,1}=&\f{u''(L)}{u'(Y_c)}\log(u(L)-\hat c)+\f{u''(Y_c)-u''(0)}{u'(Y_c)}\log(-\hat c)+u'(Y_c)^{-1}\int_0^L\frac{(u'(Y_c)-u'(Z))u''\chi}{u(Z)-\hat c}dZ\\
&-u'(Y_c)^{-1}\int_0^L\log(u(Z)-\hat c)\partial_Z[u'(Z)\chi]dZ.
\end{align*}
We can easily obtain 
\begin{align}\label{eq:H11-3}
	\begin{split}
		&\left|\mathrm{Re}\left(\mathcal{I}_{Err,1}\right)\right|\leq C|\log|\varepsilon||,\quad\left|\mathrm{Im}\left(\mathcal{I}_{Err,1}\right)\right|\leq C|c|.
	\end{split}
\end{align}
Therefore, by \eqref{eq:H11-1}, \eqref{eq:H11-2} and \eqref{eq:H11-3}, we obtain 
\begin{align*}
	&\left|\mathrm{Re}(\mathcal I_1)\right|\leq C|\log|\varepsilon||,\quad\mathrm{Im}(\mathcal I_1)=-\f{u''(Y_c)}{u'(Y_c)}\arg(-\hat c)+\mathcal{O}(|c|).
\end{align*}

Now we turn to show the estimates for $\mathcal I_2$. By the definition of $A_1$ and $A_2$,  we notice that 
\begin{align}\label{eq:H12-decom}
	\begin{split}
		&\mathcal{I}_2=\underbrace{\int_0^{+\infty}\mathcal H_{1,1}^{(2)}u''(Z)\partial_Y\eta(Z)^{-1}\chi(Z)dZ}_{\mathcal I_{2,1}} +\underbrace{\int_0^{+\infty}\mathcal H_{1,err1}(Z)u''(Z)\partial_Y\eta(Z)^{-1}\chi(Z)dZ}_{\mathcal I_{2,2}}
	\end{split}
\end{align}
with 
\begin{align*}
&\mathcal H_{1,1}^{(2)}=A_2(1,Z)A_1(Z)-A_1(1,Z)A_2(Z),\\
	&\mathcal H_{1,err1}(Z)=A_2(1,Z)(u(Z)-\hat c)^{-1}(Err_1(Z)A_1(Z)+Err_2(Z)\partial_YA_1(Z))\\
	&\qquad\qquad\qquad-A_1(1,Z)(u(Z)-\hat c)^{-1}(Err_1(Z)A_2(Z)+Err_2(Z)\partial_YA_2(Z)).
\end{align*}
Notice that 
\begin{align*}
	\mathcal{I}_{2,1}=\int_{0}^{2Y_c}\mathcal H_{1,1}^{(2)}u''(Z)\partial_Z\eta(Z)^{-1}dZ+\underbrace{\int_{2Y_c}^{+\infty}\mathcal H_{1,1}^{(2)}u''(Z)\partial_Y\eta(Z)^{-1}\chi(Z)dZ}_{\mathcal R_{2,1}}.
\end{align*}
Moreover, we have
\begin{align*}
&\int_{0}^{2Y_c}\mathcal H_{1,1}^{(2)}u''(Z)\partial_Z\eta(Z)^{-1}dZ\\
&=u''(Y_c)\int_{0}^{2Y_c}\mathcal H_{1,1}^{(2)}\partial_Z\eta(Z)dZ+\int_{0}^{2Y_c}\mathcal H_{1,1}^{(2)}[u''(Z)\partial_Z\eta(Z)^{-1}-u''(Y_c)\partial_Z\eta
(Z)]dZ.
\end{align*}
On the other hand, by lemma \ref{lem:est-eta} and \ref{lem:A1A2}, we can obtain 
\begin{align*}
	&\Big|\int_{0}^{2Y_c}\mathcal H_{1,1}^{(2)}[u''(Z)\partial_Z\eta(Z)^{-1}-u''(Y_c)\partial_Z\eta
(Z)]dZ\Big|\\
&\leq Cc_r\int_0^{2Y_c}|A_2(1,Z)A_1(Z)-A_1(1,Z)A_2(Z)|\leq Cc_r.
\end{align*}
Hence, we obtain 
\begin{align}\label{eq:H-11-2-1}
	\int_{0}^{2Y_c}\mathcal H_{1,1}^{(2)}u''(Z)\partial_Z\eta(Z)^{-1}dZ=u''(Y_c)\int_{0}^{2Y_c}\mathcal H_{1,1}^{(2)}\partial_Z\eta(Z)dZ+\mathcal O(c_r).
\end{align}
By Lemma \ref{lem: est-Scorer}, we have 
\begin{align*}
	\int_{0}^{2Y_c}\mathcal H_{1,1}^{(2)}\partial_Z\eta(Z)dZ=&\pi|\varepsilon|^{-\f13}e^{-\mathrm i\frac{\pi}{2}}u'(Y_c)^{-\f23}\int_0^{2Y_c}Hi(e^{-\mathrm i\frac{\pi}{2}}\kappa\eta(Z))\partial_Z\eta(Z)dZ+\mathcal O(c_r^2|\varepsilon|^{-\f13})\\
	=&-\pi u'(Y_c)^{-1}\int_{e^{\mathrm i\frac{\pi}{2}}\kappa\eta(0)}^{e^{\mathrm i\frac{\pi}{2}}\kappa\eta(2Y_c)}Hi(-z)dz+\mathcal O(c_r^2|\varepsilon|^{-\f13}).
\end{align*}
Hence, by Lemma \ref{lem:asy-Hi} and applying the optimal truncation of the asymptotic expansion \eqref{eq:asy-int-Hi}, we can obtain 
\begin{align}\label{eq:Hi-int}
\begin{split}
	&\Big|\int_{0}^{2Y_c}\mathcal H_{1,1}^{(2)}\partial_Z\eta(Z)dZ\Big|\leq C,\\
&\mathrm{Im}\Big(\int_{0}^{2Y_c}\mathcal H_{1,1}^{(2)}\partial_Z\eta(Z)dZ\Big)=\frac{-u''(Y_c)\pi}{u'(Y_c)}+\mathcal O(|\varepsilon|c_r^{-3}+c^2|\varepsilon|^{-\f13}).
\end{split}
\end{align}
Therefore, by \eqref{eq:H-11-2-1} and \eqref{eq:Hi-int}, we have
\begin{align}\label{eq:H-11-2-2}
	\begin{split}
			&\Big|\int_{0}^{2Y_c}\mathcal H_{1,1}^{(2)}u''(Z)\partial_Z\eta(Z)^{-1}dZ\Big|\leq C,\\
	&\mathrm{Im}\Big(\int_{0}^{2Y_c}\mathcal H_{1,1}^{(2)}u''(Z)\partial_Z\eta(Z)^{-1}dZ\Big)=\frac{-u''(Y_c)\pi}{u'(Y_c)}+\mathcal O(c_r+c_r^2|\varepsilon|^{-\f13}).
	\end{split}
\end{align}
By Lemma \ref{lem:airy-langer-asy} and \eqref{eq:airy-decay}, we know that for any $Y\geq 2Y_c$,
\begin{align*}
	 \mathcal H_{1,1}^{(2)}(Y)=\tilde C|\varepsilon|^{-\f13}(\kappa\eta(Y))^{-1}(1+\mathcal O(|\kappa\eta(Y)|^{-\f32}),\quad\tilde C\in\mathbb R,
\end{align*}
which along with the fact $|\mathrm{Im}(\eta^{-1})(Y)|\leq c_ic_r^{-1}$ for $Y\geq 2Y_c$, implies that 
\begin{align*}
	|\mathcal R_{2,1}|\leq C|\log|\varepsilon||,\quad|\mathrm{Im}(\mathcal R_{2,1})|\leq C(c_ic_r^{-1}+|\varepsilon|^{\f16}c_r^{-\f12})\ll C|\varepsilon|^{\frac{1}{24}} .
\end{align*}
The above bounds along with \eqref{eq:H-11-2-2} deduce that 
\begin{align*}
	|\mathcal I_{2,1}|\leq C\text{ and }\mathrm{Im}(\mathcal I_{2,1})=\frac{-u''(Y_c)\pi}{u'(Y_c)^{\f13}}+\mathcal O(|\varepsilon|^{\frac{1}{24}}+c_r^2|\varepsilon|^{-\f13}) .
\end{align*}
By Lemma \ref{lem:err1-err2} and Lemma \ref{lem:A1A2} and a similar argument as in the proof of  Proposition \ref{prop:airy-green-Y}, we can infer that 

\begin{align*}
	\left|\int_0^{+\infty}\mathcal H_{1,err1}(Z)u''\partial_Y\eta(Z)^{-1}\chi(Z)dZ\right|\leq C|\varepsilon|^\f13 .
\end{align*}
Gathering the above results, we obtain
\begin{align*}
	&\left|\mathrm{Re}\left(\mathcal I_2\right)\right|\leq C,\quad\mathrm{Im}(\mathcal I_{2})=\frac{-u''(Y_c)\pi}{u'(Y_c)^{\f13}}+\mathcal O(|\varepsilon|^{\frac{1}{24}}+c_r|\varepsilon|^{-\f13}).
\end{align*}

The proof is completed.	
\end{proof}


\subsection{The fast mode}

This part is devoted to constructing the fast mode for the Orr-Sommerfeld equation. We construct the fast mode solution $\phi_f$ to Orr-Sommerfeld equation around the solution $\psi_a$ to the homogeneous Airy equation,  as constructed in Theorem \ref{them:airy-homo}. By our construction,  $\psi_a(Y)$ is a solution to the following Airy equation
\begin{align*}
		&\varepsilon(\partial_Y^2-\alpha^2)w_a-(u-c)w_a=0,\quad(\partial_Y^2-\alpha^2)\psi_a(Y)=w_a(Y),\\
		&\lim_{Y\to\infty}\psi_a(Y)=\lim_{Y\to\infty}w_a(Y)=0.
\end{align*}
Moreover, we notice that 
\begin{align*}
	\varepsilon(\partial_Y^2-\alpha^2)^2\psi_a-(u-c)(\partial_Y^2-\alpha^2)\psi_a-u''\psi_a =-u''\psi_a.
\end{align*}
Hence, to construct $\phi_f$, we need to construct a solution $\tilde \phi_f(Y)$ to the equation:
\begin{align}\label{eq:fast-mode-err}
	\varepsilon(\partial_Y^2-\alpha^2)^2\tilde\phi_f-(u-c)(\partial_Y^2-\alpha^2)\tilde\phi_f-u''\tilde\phi_f=u''\psi_a
\end{align}
with the following properties
\begin{align*}
	\lim_{Y\to\infty}\tilde\phi_f(Y)=0,\quad|\tilde\phi_f(0)|\ll|\psi_a(0)| \text{ and }|\partial_Y\tilde\phi_f(0)|\ll|\partial_Y\psi_a(0)|.
\end{align*}

\begin{theorem}\label{them:fast-mode}
Let $(\al,c)\in\mathbb{ H}_3$ with $c_i>-|\varepsilon|\alpha^{-\f32}/ 2$. Then there exists a solution $e^{\eta_0 Y}\phi_f(Y)\in W^{4,\infty}(\mathbb R_+)$ to \eqref{eq:OS-homo} such that 
\begin{align*}
	&|\phi_f(0)-\psi_a(0)|\leq C|\varepsilon|^\f23(|\varepsilon|^\f13+\alpha)|\log|\varepsilon||\ll|\psi_a(0)|,\\
	&|\partial_Y(\phi_f-\psi_a)(0)|\leq C|\varepsilon|^\f23|\log|\varepsilon||\ll|\partial_Y\psi_a(0)|.
\end{align*}

\end{theorem}
\begin{proof}
Firstly, according to Theorem \ref{them:airy-homo}, we know that 
\begin{align*}
	\|(\psi_a-\psi_{a}^{(0)})\|_{L^\infty_{\eta_0}}
	\ll&C|\varepsilon|\alpha^{-1}\sim|\psi_a^{(0)}(0)|\sim|\psi_a(0)|.
\end{align*}
We apply Theorem \ref{thm:OS-non-homo}  to construct the solution $\tilde\phi_f$ to \eqref{eq:fast-mode-err}. To show that $\psi_a(0)$ and  $\partial_Y\psi_a(0)$ is the leading order of $\phi_f(0)$ and $\partial_Y\phi_f(0)$, respectively, We need to provide more precise estimates for $\tilde\phi_f$. 

We define $\varphi_f^{(1)}(Y)=\varphi_1^{(1)}(Y)+\varphi_2^{(1)}(Y)$ satisfying $Ray_{\hat c}[\varphi_f^{(1)}]=u''\psi_a$,  where 
\begin{align}
	\varphi_{1}^{(1)}(Y):=(u-\hat  c)\int_Y^{+\infty}\frac{\int_{Y'}^\infty u''(Z)\psi_a^{(0)}(Z)dZ}{(u(Y')-\hat c)^2}dY',
\end{align}
and $\varphi_2^{(1)}(Y)$ is the solution to the following system constructed in Proposition \ref{pro: phi_non}
\begin{align}
	Ray_{\hat c}[\varphi_2^{(0)}]=u''(\psi_a-\psi_a^{(0)})-\alpha^2(u-\hat c)\varphi_1^{(1)}.
\end{align}

Now we show the estimates for $\varphi_{1}^{(1)}(Y)$. By \eqref{eq:airy-homo-n} and \eqref{eq:airy-homo-k}, we have that for any $Y\geq Y_c$, 
\begin{align*}
|\varphi_{1}^{(1)}(Y)|\leq& C\int_Y^{+\infty}|u(Y')-\hat c|^{-1}\int_{Y'}^{+\infty}|u''(Z)\psi_a^{(0)}(Z)|dZdY'\\
\leq&C|\varepsilon|^\f23\int_Y^{+\infty}|u(Y')-\hat c|^{-1}\int_{Y'}^{+\infty}e^{-\eta_0 Z}e^{-C|\varepsilon|^{-\f13}|\eta(Z)-\eta(0)|}dZdY' \\
\leq&C|\varepsilon||\log|\varepsilon||e^{-\eta_0 Y}.
\end{align*}
For any $Y\in[0,Y_c]$, we write 
\begin{align}\label{eq:fast-mode-ray01}
	\begin{split}
			\varphi_{1}^{(1)}(Y)
	=&u'(Y_c)^{-1}\int_Y^{+\infty}u''(Z)\psi_a^{(0)}(Z)dZ-\frac{u(Y)-\hat c}{u'(Y_c)}\int_Y^{+\infty}\frac{u''(Z)\psi_a^{(0)}(Z)}{u(Z)-\hat c}dZ\\
	&+\frac{u(Y)-\hat c}{u'(Y_c)}\int_Y^{+\infty}\frac{(u'(Y_c)-u'(Y'))\int_{Y'}^\infty u''(Z)\psi_a^{(0)}(Z)dZ}{(u(Y')-\hat c)^2}dY'.
	\end{split}
\end{align}
We find that for any $Y\in[0, Y_c]$,
\begin{align}\label{eq:fast-mode-ray02}
	\begin{split}
		\left|\int_Y^{+\infty}u''(Z)\psi_a^{(0)}(Z)dZ\right|\leq& C|\varepsilon|^\f23 |\kappa\eta(0)|^\f14\int_Y^{+\infty}e^{-C|\varepsilon|^{-\f13}|\eta(Z)-\eta(0)||\kappa\eta(0)|^\f12}dZ\\
	\leq&C|\varepsilon|.
	\end{split}
\end{align}
We notice that for any $Y\in[\delta|\varepsilon|^\f13,Y_c]$ with some $\delta>0$ such that $|\eta(Y)-\eta(0)|\geq C|\varepsilon|^\f13$, 
\begin{align*}
	\left|\frac{u(Y)-\hat c}{u'(Y_c)}\int_Y^{+\infty}\frac{u''(Z)\psi_a^{(0)}(Z)}{u(Z)-\hat c}dZ\right|\leq& C|\hat c||\varepsilon|^\f23e^{-C|\kappa\eta(0)|^\f12}\int_Y^{+\infty}e^{-\eta_0 Z}|u(Z)-\hat c|^{-1}dZ\\
	\leq& C|\varepsilon|^\f23|\alpha||\log|\varepsilon||,
\end{align*}
and for any $Y\in[0,\delta|\varepsilon|^\f13]$,
\begin{align*}
	&\left|\frac{u(Y)-\hat c}{u'(Y_c)}\int_Y^{+\infty}\frac{u''(Z)\psi_a^{(0)}(Z)}{u(Z)-\hat c}dZ\right|\\
	&\quad\leq C|\hat c|\int_{Y}^{\delta|\varepsilon|^\f13}\left|\frac{u''(Z)\psi_a^{(0)}(Z)}{u(Z)-\hat c}\right|dZ+C|\hat c|\int_{\delta|\varepsilon|^\f13}^{+\infty}\left|\frac{u''(Z)\psi_a^{(0)}(Z)}{u(Z)-\hat c}\right|dZ\\
	&\quad\leq C|\varepsilon|^\f23|\kappa\eta(0)|^\f14\int_Y^{\delta|\varepsilon|^\f13}e^{-C|\varepsilon|^{-\f13}|\eta(Z)-\eta(0)||\kappa\eta(0)|^\f12}dZ+C|\varepsilon|^\f23|\alpha||\log|\varepsilon||\\
	&\quad\leq C|\varepsilon|^\f23(|\varepsilon|^\f13+\alpha)|\log|\varepsilon||.
\end{align*}
By gathering the above two estimates, we obtain that for any $Y\in[0,Y_c]$,
\begin{align}\label{eq:fast-mode-ray03}
	\left|\frac{u(Y)-\hat c}{u'(Y_c)}\int_Y^{+\infty}\frac{u''(Z)\psi_a^{(0)}(Z)}{u(Z)-\hat c}dZ\right|\leq C|\varepsilon|^\f23(|\varepsilon|^\f13+\alpha)|\log|\varepsilon||.
\end{align}
For the last term on the right hand side of \eqref{eq:fast-mode-ray01}, we have that for any $Y\in[0,Y_c]$
\begin{align*}
	&\left|\frac{u(Y)-\hat c}{u'(Y_c)}\int_Y^{+\infty}\frac{(u'(Y_c)-u'(Y'))\int_{Y'}^\infty u''(Z)\psi_a^{(0)}(Z)dZ}{(u(Y')-\hat c)^2}dY'\right|\\
	&\qquad\leq C|\hat c|\left|\int_{Y}^\infty u''(Z)\psi_a^{(0)}(Z)dZ\right|\int_Y^{+\infty}|u(Z)-\hat c|^{-1}dZ\\
	&\qquad\leq C|\varepsilon|\alpha|\log|\varepsilon||,
\end{align*}
which along with \eqref{eq:fast-mode-ray01}, \eqref{eq:fast-mode-ray02} and \eqref{eq:fast-mode-ray03} gives that for any $Y\in[0,Y_c]$,
\begin{align*}
	|\varphi_{1}^{(1)}(Y)|\leq C|\varepsilon|^\f23(|\varepsilon|^\f13+\alpha)|\log|\varepsilon||.
\end{align*}
Hence, we obtain 
\begin{align}
	\|\varphi_{1}^{(1)}\|_{L^\infty_{\eta_0}}\leq C|\varepsilon|^\f23(|\varepsilon|^\f13+\alpha)|\log|\varepsilon||.
\end{align}
By the definition of $\varphi_{1}^{(0)}$, we have 
\begin{align*}
	\partial_Y\varphi_{1}^{(1)}(Y)=&u'(Y)\int_Y^{+\infty}\frac{\int_{Y'}^\infty u''(Z)\psi_a^{(0)}(Z)dZ}{(u(Y')-\hat c)^2}dY'-\frac{\int_{Y}^\infty u''(Z)\psi_a^{(0)}(Z)dZ}{u(Y)-\hat c}.
\end{align*}
Hence, by a similar argument as above, we can obtain  
\begin{align}
\|(u-\hat{c})\partial_Y\varphi_{1}^{(1)}\|_{L^\infty_{\eta_0}}+\alpha\|\partial_Y\varphi_1^{(1)}\|_{L^\infty_{\eta_0}}\leq C|\varepsilon|^\f23(|\varepsilon|^\f13+\alpha)|\log|\varepsilon||.
\end{align}
Moreover, by the definition of $\varphi_{1,1}^{(0)}$, we have 
\begin{align*}
(u-\hat c)\partial_Y^2\varphi_{1}^{(1)}(Y)=u''(\varphi_{1}^{(1)}(Y)+\psi_a^{(0)}(Y)).
\end{align*}
Then we get 
\begin{align*}
\|(u-\hat{c})\partial^2_Y\varphi_{1}^{(1)}\|_{L^\infty_{\eta_0}}\leq& C\|\varphi_{1}^{(1)}\|_{L^\infty_{\eta_0}}+C\|\psi_a^{(0)}\|_{L^\infty_{\eta_0}}\\
\leq&C|\varepsilon|^\f23(|\varepsilon|^\f13+\alpha)|\log|\varepsilon||+C|\varepsilon|^\f23|\kappa\eta(0)|^{-1}\leq C|\varepsilon|\alpha^{-1}.
\end{align*}
Similarly, we have 
\begin{align*}
\|(u-\hat{c})^2\partial^3_Y\varphi_{1}^{(1)}\|_{L^\infty_{\eta_0}}\leq C|\varepsilon|^\f12\alpha^\f12,\quad
\|(u-\hat{c})^3\partial^4_Y\varphi_{1}^{(1)}\|_{L^\infty_{\eta_0}}&\leq C\alpha^2.
\end{align*}
Hence, we obtain
\begin{align}\label{eq:fast-mode-ray0-e}
\begin{split}
&\|(u-\hat{c})\partial_Y\varphi_{1}^{(1)}\|_{L^\infty_{\eta_0}}+\|\varphi_{1}^{(0)}\|_{L^\infty_{\eta_0}}\leq C|\varepsilon|^\f23(|\varepsilon|^\f13+\alpha)|\log|\varepsilon||,\\
&\|(u-\hat{c})\partial^2_Y\varphi_{1}^{(1)}\|_{L^\infty_{\eta_0}}\leq C|\varepsilon|\alpha^{-1}, \quad\|(u-\hat{c})^2\partial^3_Y\varphi_{1}^{(1)}\|_{L^\infty_{\eta_0}}\leq C|\varepsilon|^\f12\alpha^\f12,\\
&\|(u-\hat{c})^3\partial^4_Y\varphi_{1}^{(1)}\|_{L^\infty_{\eta_0}}\leq C\alpha^2.
\end{split}
\end{align}
In particular, we have
\begin{align*}
|\varphi_1^{(1)}(0)|\leq C|\varepsilon|^\f23(|\varepsilon|^\f13+\alpha)|\log|\varepsilon||\ll|\psi_a(0)|\text{ and } |\partial_Y\varphi_{1}^{(1)}(0)|\leq C|\varepsilon|^\f23|\log|\varepsilon||\ll|\partial_Y\psi_a(0)|.
\end{align*}

Now we turn to show the estimates for $\varphi_2^{(1)}(Y)$. By  Proposition \ref{pro: phi_non}, Theorem \ref{them:airy-homo} and \eqref{eq:fast-mode-ray0-e}, we  have 
\begin{align*}
&\|(u-\hat{c})\partial^2_Y\varphi_{2}^{(1)}\|_{L^\infty_{\eta_0}}+\|\partial_Y\varphi_{2}^{(1)}\|_{L^\infty_{\eta_0}}+\|\varphi_{2}^{(1)}\|_{L^\infty_{\eta_0}}\\
&\leq C|\log|\hat c_i||\|u''(\psi_a-\psi_a^{(0)})\|_{L^\infty_{\eta_0}}+C\alpha^2\|\varphi^{(1)}_1\|_{L^\infty_{\eta_0}}\leq C|\varepsilon|^{\f23}\alpha|\log|\varepsilon||\ll |\varepsilon|\alpha^{-1},
\end{align*}
and 
\begin{align*}
	\|\varphi_2^{(1)}\|_{\mathcal X}\leq \|\varphi_2^{(1)}\|_{\mathcal Y}\leq& C|\log|\varepsilon||\sum_{j=0}^2\Big(\|(u-\hat c)^j\partial_Y^j(u''(\psi_a-\psi_a^{(0)})-\alpha^2(u-\hat c)\varphi_1^{(1)})\|_{L^\infty_{\eta_0}}\Big)\\
	\leq&|\varepsilon|^\f23\alpha|\log|\varepsilon||.
\end{align*}
The above estimates along with \eqref{eq:fast-mode-ray0-e} deduce
\begin{align}\label{eq:fast-mode-rayf-e}
\begin{split}
&\al^2\|(u-\hat c)\partial_Y^2\varphi_f^{(1)}\|_{L^\infty_{\eta_0}}+|\varepsilon|^\f13\|(u-\hat{c})\partial_Y\varphi^{(1)}_f\|_{L^\infty_{\eta_0}}+|\varepsilon|^\f13\|\varphi^{(1)}_f\|_{L^\infty_{\eta_0}}\leq C|\varepsilon|\alpha|\log|\varepsilon||,\\
&\|\varphi^{(1)}_f\|_{\mathcal X}\leq C\alpha^2.
\end{split}
\end{align}
In particular, we obtain 
\begin{align}\label{eq:fast-phi-bv}
|\varphi^{(1)}_f(0)|\leq C|\varepsilon|^\f23\alpha|\log|\varepsilon||\ll|\psi_a(0)|\text{ and } |\partial_Y\varphi^{(1)}_f(0)|\leq C|\varepsilon|^\f23|\log|\varepsilon||\ll|\partial_Y\psi_a(0)|.
\end{align}
Moreover, we have
\begin{align*}
OS[\psi_a+\varphi_f^{(1)}]=\varepsilon(\partial_Y^2-\alpha^2)^2\varphi_f^{(1)}.
\end{align*}
Next we define $\psi^{(1)}_f=\psi^{(1)}_{f,1}+\psi^{(1)}_{f,2}+\psi^{(1)}_{f,3}$,
where $\psi^{(1)}_{f,1}, \psi^{(1)}_{f,2}$ and $\psi^{(1)}_{f,3}$ are the solutions to \eqref{eq:r-a-airy} by taking  $\varphi^{(j)}=\varphi_f^{(1)}$ and $\psi^{(j-1)}=\psi_a$. Then by  \eqref{eq:OS-non-homo-j1},  we have
\begin{align*}
	\|\psi^{(1)}_f\|_{\mathcal X}\leq C(\al+|\varepsilon|^\f23|\hat c_i|^{-1})|\log|\varepsilon||^2(\|(u-\hat c)^2\pa_Y^2\psi_a\|_{L^\infty_{\eta_0}}+\|(u-\hat c)\pa_Y\psi_a\|_{L^\infty_{\eta_0}}+\|\psi_a\|_{L^\infty_{\eta_0}}).
\end{align*}
Moreover, by Theorem \ref{them:airy-sf} and Theorem \ref{them:airy-gf-1}, we can obtain  
\begin{align}\label{eq:fast-psi-1}
	\begin{split}
			&\|(u-\hat c)^2\pa_Y^2\psi^{(1)}_{f,1}\|_{L^\infty_{\eta_0}}+\|(u-\hat c)\pa_Y\psi^{(1)}_{f,1}\|_{L^\infty_{\eta_0}}+\|\psi^{(f)}_{f,1}\|_{L^\infty_{\eta_0}}\\
	&\leq C|\varepsilon||\log|\varepsilon||\sum_{j=0}^2\Big(\|\partial_Y^j\varphi_f^{(1)}\|_{L^\infty([1/2,+\infty))}+\|\partial_Y^j\psi_a\|_{L^\infty([1/2,+\infty))}\Big),\\
	&\leq C|\varepsilon||\log|\varepsilon||\sum_{j=0}^2\Big(\|(u-\hat c)^j\partial_Y^j\varphi_f^{(1)}\|_{L^\infty}+\|(u-\hat c)^j\partial_Y^j\psi_a\|_{L^\infty}\Big),		
	\end{split}
\end{align}
\begin{align}\label{eq:fast-psi-2}
	\begin{split}
		&\|(u-\hat c)^2\pa_Y^2\psi^{(1)}_{f,2}\|_{L^\infty_{\eta_0}}+\|(u-\hat c)\pa_Y\psi^{(1)}_{f,2}\|_{L^\infty_{\eta_0}}+\|\psi^{(f)}_{f,2}\|_{L^\infty_{\eta_0}}\\
	&\leq C|\varepsilon|^\f23\al^{-\f12}|\log|\varepsilon||(\|\varphi_f^{(1)}\|_{L^\infty}+\|\psi_a\|_{L^\infty}),
	\end{split}
\end{align}
and 
\begin{align}\label{eq:fast-psi-3}
	\begin{split}
		&\|(u-\hat c)^2\pa_Y^2\psi^{(1)}_{f,3}\|_{L^\infty_{\eta_0}}+\|(u-\hat c)\pa_Y\psi^{(1)}_{f,3}\|_{L^\infty_{\eta_0}}+\|\psi^{(f)}_{f,3}\|_{L^\infty_{\eta_0}}\\
	&\leq C|\varepsilon|^\f13|\log|\varepsilon||(\|(u-\hat c)\partial_Y\varphi_f^{(1)}\|_{L^\infty}+\|(u-\hat c)\partial_Y\psi_a\|_{L^\infty})\\
	&\quad+C\alpha|\log|\varepsilon||(\|\varphi_f^{(1)}\|_{L^\infty}+\|\psi_a\|_{L^\infty}).
	\end{split}
\end{align}
We also notice that by Theorem \ref{them:airy-homo}, \eqref{eq:airy-homo-n} and \eqref{eq:airy-homo-k},
\begin{align}\label{eq:fast-psi-a}
	\alpha^{-2}|\varepsilon|^{\f23}\|(u-\hat c)^2\pa_Y^2\psi_a\|_{L^\infty_{\eta_0}}+|\varepsilon|^\f16\alpha^{-\f12}\|(u-\hat c)\pa_Y\psi_a\|_{L^\infty_{\eta_0}}+\|\psi_a\|_{L^\infty_{\eta_0}}\leq C|\varepsilon|^\f23.
\end{align}
Therefore, by \eqref{eq:fast-mode-rayf-e}, \eqref{eq:fast-psi-1}-\eqref{eq:fast-psi-a},  we obtain 
\begin{align}
\begin{split}
	&\|\psi_f^{(1)}\|_{\mathcal{X}}\leq C\al^3|\log|\varepsilon||^2,\\
	&\|(u-\hat c)^2\pa_Y^2\psi^{(1)}_f\|_{L^\infty_{\eta_0}}+\|(u-\hat c)\pa_Y\psi^{(1)}_f\|_{L^\infty_{\eta_0}}+\|\psi^{(f)}_f\|_{L^\infty_{\eta_0}}\leq C|\varepsilon|^\f23\al|\log|\varepsilon||^2\ll|\varepsilon|\al^{-1}.
\end{split}
\end{align}
In particular, we have
\begin{align}\label{eq:fast-psi-bv}
	|\psi_f^{(1)}(0)|\leq C|\varepsilon|^\f23\alpha |\log|\varepsilon||\ll\psi_a(0)\text{ and }|\partial_Y\psi_f^{(1)}(0)|\leq C|\varepsilon|^\f23|\log|\varepsilon||\ll|\partial_Y\psi_a(0)|. 
\end{align}
Moreover, $OS[\psi_a+\varphi_f^{(1)}+\psi_f^{(1)}]=u''\psi_f^{(1)}$. We define that $\tilde{\phi}_f^{(1)}$ is the solution to $OS[\tilde{\phi}_f^{(1)}]=-u''\psi^{(1)}_f$ constructed in Theorem \ref{thm:OS-non-homo}. Then by Theorem \ref{thm:OS-non-homo}, we find that 
\begin{align*}
	\|\tilde{\phi}_f^{(1)}\|_{\mathcal{X}}\leq& C|\log|\varepsilon||(\|(u-\hat c)^2\pa_Y^2\psi^{(1)}_f\|_{L^\infty_{\eta_0}}+\|(u-\hat c)\pa_Y\psi^{(1)}_f\|_{L^\infty_{\eta_0}}+\|\psi^{(f)}_f\|_{L^\infty_{\eta_0}})\\
	\leq&C|\varepsilon|^\f23\al|\log|\varepsilon||^3\ll|\varepsilon|\al^{-1}.
\end{align*} 
In particular, we have
\begin{align}\label{eq:fast-tphi-bv}
|\tilde{\phi}^{(1)}_f(0)|\leq C|\varepsilon|^\f23\alpha|\log|\varepsilon||\ll|\psi_a(0)|\text{ and } |\partial_Y\tilde{\phi}^{(1)}_f(0)|\leq C|\varepsilon|^\f23|\log|\varepsilon||\ll|\partial_Y\psi_a(0)|.
\end{align}
Moreover, $OS[\psi_a+\varphi_f^{(1)}+\psi_f^{(1)}+\tilde{\phi}_f^{(1)}]=0$.
Then we take $\tilde{\phi}_f=\varphi_f^{(1)}+\psi_f^{(1)}+\tilde{\phi}_f^{(1)}$, and by \eqref{eq:fast-phi-bv},\eqref{eq:fast-psi-bv} and \eqref{eq:fast-tphi-bv}, we have
\begin{align*}
|\tilde{\phi}_f(0)|\leq C|\varepsilon|^\f23\alpha|\log|\varepsilon||\ll|\psi_a(0)|\text{ and } |\partial_Y\tilde{\phi}_f(0)|\leq C|\varepsilon|^\f23|\log|\varepsilon||\ll|\partial_Y\psi_a(0)|.	
\end{align*}

The proof is completed.
\end{proof}

\section{Dispersion relation and T-S waves}

In this section, we construct a non-trivial solution $\phi(Y)$ to the homogeneous Orr-Sommerfeld equation with non-slip boundary condition
\begin{align}\label{eq:OS-nonslip}
	\begin{split}
		&\varepsilon(\partial_Y^2-\alpha^2)^2\phi-(u-c)(\partial_Y^2-\alpha^2)\phi+u''\phi=0,\\
		&\phi(0)=\partial_Y\phi(0)=0,\quad\lim_{Y\to\infty}\phi(Y)=0.
	\end{split}
\end{align}
We shall construct such a solution $\phi$ by a linear combination of the slow mode $\phi_s$ and the fast mode $\phi_f$, which are constructed in Theorem \ref{them:slow-mode} and Theorem \ref{them:fast-mode}, respectively. That is, for certain $\alpha\ll 1$, there exists a $c\in\mathbb C$ such that 
\begin{align}
	\phi=C_s(\alpha)\phi_s+C_f(\alpha)\phi_f
\end{align}
with $C_s$ and $C_f$ are not zero. Since $\phi_s$ and $\phi_f$ are both solutions to the homogeneous Orr-Sommerfeld equation, we only need to match the non-slip boundary condition by choosing appropriate $C_s$ and $C_f$. In details, 
\begin{align*}
	C_s\phi_s(0)+C_f\phi_f(0)=0 \text{ and }C_s\partial_Y\phi_s(0)+C_f\partial_Y\phi_f(0)=0.
\end{align*}
Hence, the existence of the non-zero constants $C_s$ and $C_f$ is guaranteed by the following dispersion relation:
\begin{align}\label{eq:dis-rel}
	\frac{\phi_s(0)}{\partial_Y\phi_s(0)}=\frac{\phi_f(0)}{\partial_Y\phi_f(0)}.
\end{align} 

The task of this section is to find $c\in\mathbb C$ such that \eqref{eq:dis-rel} holds for certain  small $\alpha$. To ensure the existences of the slow mode $\phi_s$ and the fast mode $\phi_f$, we assume $(\alpha,c)\in\mathbb H_3$.
By Theorem \ref{them:slow-mode} and Theorem \ref{them:fast-mode}, we know that \eqref{eq:dis-rel} is equivalent to 
\begin{align}\label{eq:solu-c}
	c-u'(0)^{-1}\alpha+\frac{u'(0)\tilde{\mathcal A}(2,0)}{\tilde{\mathcal A}(1,0)}+\mathcal R_d(\alpha,c)=0,
\end{align}
where $\mathcal R_d(\alpha,c)$ is smooth on $c$ and analytic on $\alpha$. We point out that the solutions to the Rayleigh equation are analytic on $c$ and $\alpha$, and the solutions to the Airy equation are smooth on $c_r, c_i$ and analytic on $\alpha$ for $\alpha\neq0$ (here we recall the relation $\varepsilon=-\mathrm i\nu^\f12\alpha^{-1}$). Moreover, by the definitions of $A_j(Y)$ with $j=1,2$, we know that $|\partial_{c_r}A_j(k,Y)+|\partial_{c_i}A_j(k,Y)|\leq C|\varepsilon|^{-\f13}|A_j(k-1,Y)|$ and $|\partial_{c_r}\partial_Y^kA_j(Y)|+|\partial_{c_r}\partial_Y^kA_j(Y)|\leq C|\varepsilon|^{-\f13}|\partial_Y^{k+1}A_j(Y)|$. 
Hence, we have that for any $(\alpha,c)\in\mathbb H_3$,
\begin{align}
	\begin{split}
		&|\mathcal R_d(\alpha,c)|\leq C(\alpha^2+|\varepsilon|\alpha^{-\f32})|\log|\varepsilon||,\quad|\partial_\alpha\mathcal R_d(\alpha,c)|\leq C(\nu),\\
	&|\partial_{c_r}\mathcal R_d(\alpha,c)|+|\partial_{c_i}\mathcal R_d(\alpha,c)|\leq C|\varepsilon|^\f12 |\log|\varepsilon||.
	\end{split}
\end{align}
Moreover, if $\alpha\gg|\varepsilon|^\f14$ and $(\alpha,c)\in\mathbb H_3$, then we have the following formula for the imaginary part of \eqref{eq:solu-c}:
\begin{align}\label{eq:dr-i}
	c_i-\frac{\alpha^2u''(0)\pi}{u'(0)^2}+\mathrm{Im}\Big(\frac{u'(0)\tilde{\mathcal A}(2,0)}{\tilde{\mathcal A}(1,0)}\Big)+o(\alpha^2)=0.
\end{align}
Here we used the fact $u'(Y_c)=u'(0)+\mathcal O(\alpha)$ adn $u''(Y_c)=u''(0)+\mathcal O(\alpha)$.
\smallskip

We first provide a priori estimate for $c(\alpha)$ solving \eqref{eq:solu-c}.

\begin{proposition}\label{prop:diss-priori}
		Let $\nu\ll1$. Then there exist $A\leq 1\ll B$ independent of $\nu$ such that for any $\alpha\in(A\nu^{\f18},B\nu^{\frac{1}{12}})$, the solution  $c(\alpha)\in \mathbb H_3$ to \eqref{eq:dis-rel} 
		continuously depends on $\alpha$ satisfying the following properties:
		
	\begin{itemize}
		\item There exists $A_1\geq A_c\geq A_0$ such that 
		      \begin{align*}
		      	&\mathrm{Im}(c(A_0\nu^\f18))<0,\quad \mathrm{Im}(c(A_c\nu^\f18))=0,\quad\mathrm{Im}(c(A_1\nu^\f18))>0.
		      \end{align*}
		\item There exits $B_c\leq B_0\leq B_1$ such that 
		       \begin{align*}
		      	&\mathrm{Im}(c(B_0\nu^\frac{1}{12}))>0,\quad \mathrm{Im}(c(B_c\nu^\frac{1}{12}))=0,\quad\mathrm{Im}(c(B_1\nu^\frac{1}{12}))<0.
		      \end{align*}
		\item If $\nu^\f18\ll\alpha\ll\nu^\frac{1}{12}$, then 
		      \begin{align*}
		      	\alpha\mathrm{Im}(c(\alpha))\sim\nu^\f14.
		      \end{align*}
	\end{itemize}

\end{proposition}
\begin{proof}
From \eqref{eq:solu-c}, we know that 
\begin{align}\label{eq:dis-rel-1}
c=u'(0)^{-1}\alpha-\frac{u'(0)\tilde{\mathcal A}(2,0)}{\tilde{\mathcal A}(1,0)}-\mathcal R_d(\alpha,c).
\end{align}

\no\textbf{Case 1. $\alpha\sim\nu^\f18$.} Let $\alpha=A\nu^\f18$. Then we first have that $|\varepsilon|=A^{-1}\nu^{\f38}$. For the case $A\gg1$ such that $|\kappa\eta(0)|>M$, we get by Lemma \ref{lem:airy-0} that
\begin{align*}
	\frac{\tilde{\mathcal A}(2,0)}{\tilde{\mathcal A}(1,0)}
	=&-e^{\mathrm i\f{\pi}{4}}|\varepsilon|^\f12c_r^{-\f12}+o(|\varepsilon|^\f12c_r^{-\f12}+|\varepsilon|^{\f13}),
	\end{align*}
	which implies that 
	\begin{align*}
		c_r=u'(0)^{-1}A\nu^\f18+C_0A^{-\f12}\nu^{\frac{3}{16}}c_r^{-\f12}+o(\nu^\f18).
	\end{align*}
	Then we can obtain that $c_r\sim A\nu^\f18$. Moreover, we have 
\begin{align}\label{eq:lower-bra-A1}
	\mathrm{Im}(c)=&-\mathrm{Im}\left(\frac{u'(0)\tilde{\mathcal A}(2,0)}{\tilde{\mathcal A}(1,0)}\right)+o(\nu^\f18)\geq CA^{-1}\nu^\f18=C'A^{-\f23}|\varepsilon|^\f13.
\end{align}
On the other hand, for the case of $A\sim1$,  we rewrite \eqref{eq:dis-rel-1} as
\begin{align*}
	1=\frac{\alpha}{u'(0)c}+\frac{u'(0)\tilde {\mathcal A}(2,0)}{-c\tilde{\mathcal {A}}(1,0)}+\mathcal O(\nu^{\frac{1}{16}}),
\end{align*}
which along with Lemma \ref{lem:airy-0} implies that 
\begin{align*}
	1=\frac{\alpha}{u'(0)c}+\frac{\mathcal A(2,0)}{\kappa\eta(0)\mathcal A(1,0)}+\mathcal O(\nu^{\frac{1}{16}}).
\end{align*}
Therefore, we obtain 
\begin{align}\label{eq:low-branch-dr}
	\mathcal F(-\kappa\eta(0))+\mathcal O(\nu^{\frac{1}{16}})=\frac{u'(0)c}{\alpha},
\end{align}
where $\mathcal F(z)$ is the function defined in \cite{L} related to the Hankel function or the Tietjens function. By the property of $\mathcal F(z)$ (see Fig 3.2 in \cite{L}), we know that there exists $z_0\in[2,2.5]$ such that $\mathrm{Im}(\mathcal F(z_0))=0$ and $\mathrm {Im}(\mathcal F(z))<0$ for any $z\in[2,z_0)$. Then there exits  $z_1=a_1+b_1i$ with $a_1<z_0$ and $|b_1|\leq C\nu^{\frac{1}{16}}$ such that $\mathrm{Re}(\mathcal F(z_0+z_1))\in[2,3]$, $\mathrm{Im}(\mathcal F(z_0+z_1)+\mathcal O(\nu^{\frac{1}{16}}))<0$ and $|\mathrm{Im}(\mathcal F(z_0+z_1)+O(\nu^{\frac{1}{16}}))|>\frac{1}{8}\nu^{\frac{1}{16}}$. Moreover, we have $|z_0+z_1|\in[2,2.5]$. Then there exists a small $r_0$ such that for any $-\kappa\eta(0)\in B(z_0+z_1,r_0)$, we always have $c_i<0$. 

In particular, we take $\mathrm{Re}(-\kappa\eta(0))=\mathrm{Re}(z_0+z_1)$, which implies 
\begin{align*}
	c_r=u'(0)^\f23|\varepsilon|^\f13\mathrm{Re}(z_0+z_1).
\end{align*}
By \eqref{eq:low-branch-dr}, we obtain 
\begin{align*}
	A_0=u'(0)^\f53\frac{\mathrm{Re}(z_0+z_1)}{\mathrm{Re}(\mathcal F(z_0+z_1))},
\end{align*}
where $\alpha=A_0\nu^{\f18}$.
As a consequence, we infer that there exists $A_0\sim1$ such that $c(A_0\nu^\f18)\in\mathbb H$ with $\mathrm{Im}(c(A_0\nu^\f18))<0$.

Since $c$ depends on $\alpha$ continuously,  we obtain that there exist $A_0, A_c, A_1$ such that $\mathrm{Im}(c(A_0\nu^\f18))<0$, $\mathrm{Im}(c(A_c\nu^\f18))=0$ and $\mathrm{Im}(c(A_1\nu^\f18))>0$.\smallskip

\no\textbf{Case 2. $\nu^\f18\ll\alpha\ll\nu^\frac{1}{12}$.} We denote $\alpha=A\nu^{\gamma}$ with $\gamma\in(\f18,\frac{1}{12})$. In this case, we first notice that by Lemma \ref{lem:airy-0},
\begin{align*}
	\mathrm{Im}(c(\alpha))=&CA^{-1}\nu^{\f14-\gamma}-A^{-3}\nu^{\f12-3\gamma}+\mathcal O(\alpha^2|\log|\varepsilon||)\\
	=&CA^{-1}\nu^{\f14-\gamma}(1+o(1))>0,
\end{align*}
provide that $\nu^{\f12-3\gamma}\ll\nu^{\f14-\beta}$ for any $\gamma>\f18$ and $\alpha^2|\log|\varepsilon||=A^2\nu^{2\gamma}|\log\nu|\ll\nu^{\f14-\gamma}$ for any $\gamma<\frac{1}{12}$.

\no\textbf{Case 3. $\alpha\sim\nu^{\frac{1}{12}}$.} 
Let $\alpha=B\nu^{\frac{1}{12}}$. In this case, by \eqref{eq:dr-i}, we know that 
\begin{align*}
	\mathrm{Im}(c(\alpha))=\frac{u''(0)}{u'(0)^2}\alpha^2\pi-\mathrm{Im}\left(\frac{u'(0)\tilde{\mathcal A}(2,0)}{\tilde{\mathcal A}(1,0)}\right)+o(\alpha^2).
\end{align*}
 We first notice that for $B\ll1$,
\begin{align}\label{eq:diss-rel-upp1}
	\begin{split}
			\mathrm{Im}(c(\alpha))=&\frac{u''(0)}{u'(0)^2}B^2\nu^\f16+C_{Ai}B^{-\f32}\nu^{\f16}+o(\nu^\f16)\\
	\geq&C(B^{-\f32}-B^2)\nu^\f16>0.
	\end{split}
\end{align}
Here $C_{Ai}$ is a positive constant just relying on Airy functions.
On the other hand, for $B\gg1$
\begin{align*}
	\mathrm{Im}(c(\alpha))
	&=\frac{u''(0) }{u'(0)^2}B^2\nu^\f16+C_{Ai}B^{-\f32}\nu^{\f16}+o(\nu^\f16)\\
	&\leq -C(B^2-B^{-\f32})\nu^\f16<0.
\end{align*}
Therefore, there exists $B_c>1$ such that for $\alpha=B_c\nu^{\frac{1}{12}}$,
\begin{align*}
\mathrm{Im}(c(\alpha))=0.
\end{align*}

The proof is completed.
\end{proof}

\begin{proposition}\label{prop:diss-existence}
	Let $\nu\ll1$.  Then  for any $\alpha\in(A\nu^{\f18},B\nu^{\frac{1}{12}})$ with $A\leq 1\ll B$, there exists a unique $c(\alpha)\in\mathbb H_3$ solving \eqref{eq:solu-c}. Moreover, $c(\alpha)$ depends on $\alpha$ continuously .
\end{proposition}

\begin{proof}
 We introduce 
 $$
 \mathbb F(\alpha_r,\alpha_i;c_r,c_i)=\big(F_r(\alpha_r,\al_i;c_r,c_i),F_i(\alpha_r,\alpha_i ;c_r,c_i)\big),
$$
  where
	\begin{align*}
		&F_r(\alpha_r, \al_i ;c_r,c_i)=\mathrm{Re}\left(c-u'(0)^{-1}\alpha+\frac{u'(0)\tilde{\mathcal A}(2,0)}{\tilde{\mathcal A}(1,0)}+\mathcal R_d(\alpha,c)\right),\\
		&F_i(\alpha_r,\al_i;c_r,c_i)=\mathrm{Im}\left(c-u'(0)^{-1}\alpha+\frac{u'(0)\tilde{\mathcal A}(2,0)}{\tilde{\mathcal A}(1,0)}+\mathcal R_d(\alpha,c)\right).
	\end{align*}
	
We need to show the Jacobian determinant of $F_r$ and $F_i$ is non-zero for $(\alpha,c)\in \mathbb{ H}_3$. For this purpose, by Lemma \ref{lem:airy-0}, we have that for $j=1,2$,
	\begin{align*}
		\frac{u'(0)\tilde{\mathcal A}(2,0)}{\tilde{\mathcal A}(1,0)}=\kappa^{-1}u'(0)\frac{\mathcal A(2,\kappa\eta(0))}{\mathcal A(1,\kappa\eta(0))}+\mathcal O(|c|^2).
	\end{align*}
	By our definition of $\eta(Y)$, we know that 
	\begin{align*}
		\eta(0;c_r,c_i)=-u'(0)^{-1}(c_r+\mathrm ic_i)+\mathcal O(|c|^2),
	\end{align*}
	and smoothly depends on $c_r$ and $c_i$. Therefore, we obtain 
	\begin{align}\label{eq:d-cr-ci}
		\begin{split}
				&\partial_{c_r}\left(\frac{u'(0)\tilde{\mathcal A}(2,0)}{\tilde{\mathcal A}(1,0)}\right)=-1+\frac{Ai(e^{\mathrm i\frac{\pi}{6}}\kappa\eta(0))\mathcal A(2,\kappa\eta(0))}{\mathcal A(1,\kappa\eta(0))^2}+\mathcal O(|c|),\\
		&\partial_{c_i }\left(\frac{u(0)\tilde{\mathcal A}(2,0)}{\tilde{\mathcal A}(1,0)}\right)=-\mathrm i +\mathrm i\frac{Ai(e^{\mathrm i(\frac{\pi}{6}}\kappa\eta(0))\mathcal A(2,\kappa\eta(0))}{\mathcal A(1,\kappa\eta(0))^2}+\mathcal O(|c|).
		\end{split}
	\end{align}
	Thus, we have 
	\begin{align*}
		&\partial_{c_r}F_r(\alpha_r,\al_i;c_r,c_i)=\mathrm{Re}\left(\frac{Ai(e^{\mathrm i\frac{\pi}{6}}\kappa\eta(0))\mathcal A(2,\kappa\eta(0))}{\mathcal A(1,\kappa\eta(0))^2}\right)+\mathcal O(|c|),\\
		&\partial_{c_r}F_i(\alpha_r,\al_i;c_r,c_i)=\mathrm{Im}\left(\frac{Ai(e^{\mathrm i\frac{\pi}{6}}\kappa\eta(0))\mathcal A(2,\kappa\eta(0))}{\mathcal A(1,\kappa\eta(0))^2}\right)+\mathcal O(|c|),\\
		&\partial_{c_r}F_i(\alpha_r,\al_i;c_r,c_i)=-\mathrm{Im}\left(\frac{Ai(e^{\mathrm i\frac{\pi}{6}}\kappa\eta(0))\mathcal A(2,\kappa\eta(0))}{\mathcal A(1,\kappa\eta(0))^2}\right)+\mathcal O(|c|),\\
		&\partial_{c_i}F(\alpha_r,\al_i;c_r,c_i)=\mathrm{Re}\left(\frac{Ai(e^{\mathrm i\frac{\pi}{6}}\kappa\eta(0))\mathcal A(2,\kappa\eta(0))}{\mathcal A(1,\kappa\eta(0))^2}\right)+\mathcal O(|c|).
	\end{align*}
Hence, the Jacobian determinant $J(F_r,F_i)(\alpha;c_r,c_i)$ of $F_r$ and $F_i$ satisfies 
	\begin{align*}
		J(F_r,F_i)(\alpha_r,\al_i;c_r,c_i)=&\left|\frac{Ai(e^{\mathrm i \frac{\pi}{6}}\kappa\eta(0))\mathcal A(2,\kappa\eta(0))}{\mathcal A(1,\kappa\eta(0))^2}\right|^2+\mathcal O(|c|).
	\end{align*}
And by Lemma \ref{lem:airy-0}, we know that for any $(\alpha,c)\in \mathbb H_3$,
	\begin{align*}
		\left|\frac{Ai(e^{\mathrm i\frac{\pi}{6}}\kappa\eta(0))\mathcal A(2,\kappa\eta(0))}{\mathcal A(1,\kappa\eta(0))^2}\right|\sim1.
	\end{align*}
	Therefore, for $\nu\ll1$, we obtain that for any $(\alpha,c)\in\mathbb  H_3$,
    \begin{align}\label{eq:F-c-Jaco}
    	\begin{split}
    			J(F_r,F_i)(\alpha_r,\al_i;c_r,c_i)=&\left|\frac{Ai(e^{\mathrm i \frac{\pi}{6}}\kappa\eta(0))\mathcal A(2,\kappa\eta(0))}{\mathcal A(1,\kappa\eta(0))^2}\right|^2+\mathcal O(|c|)\\
		\geq&C(1-|c|)\geq \frac{1}{2}C>0.
    	\end{split}
    \end{align}
    
	On the other hand, since $F_r,F_i$ are both analytic on $\alpha$, for any fixed $\nu$, we have 
	\begin{align}\label{eq:F-alpha}
		|\partial_\alpha F_r(\alpha;c_r,c_i)|+|\partial_\alpha F_i(\alpha;c_r,c_i)|\leq C. .
	\end{align}
We denote 
	\begin{align*}
		\mathbb A(c^{(j)})=\frac{u'(0)\tilde{\mathcal A}(2,0)}{\tilde{\mathcal A}(1,0)}\text{ with taking } c=c^{(j)}.
	\end{align*}
We now show that there exists a $(\al^0,c^0)\in\mathbb H_3$ such that $\mathbb F(\alpha^0;c_r^0,c_i^0)=0$. Let $\alpha=A\nu^\f18$ with $A\gg1$.  We take $c^{(0)}=u'(0)^{-1}\alpha$, and for $j\geq 1$,
	\begin{align*}
		c^{(j)}=u'(0)^{-1}\alpha-\mathbb A(c^{(j-1)})-\mathcal R_d(\alpha,c^{(j-1)}).
	\end{align*}
	By Lemma \ref{lem:airy-0}, we know that
	\begin{align*}
		\mathbb A(c^{(0)})=-A^{-1}e^{\mathrm i\frac{\pi}{4}}\nu^\f18+\mathcal O(A^{-\f43})\nu^\f18,
	\end{align*}
   which implies 
   \begin{align*}
   	&c^{(1)}=\nu^\f18\left(u'(0)A+A^{-1}e^{\mathrm i \frac{\pi}{4}}\right)+\mathcal O(A^{-\f43})\nu^\f18,\\
   	&\mathrm{Im}(c^{(1)})=\frac{\sqrt{2}}{2}A^{-1}\nu^{\f18}+\mathcal O(A^{-\f43})\nu^{\f18}.
   \end{align*}
   Hence, $(\al,c^{(1)})\in\mathbb H_3$.
   By the inductive argument and Lemma \ref{lem:airy-0}, we can obtain that for any $j\geq 1$, $(\alpha,c^{(j)})\in\mathbb H_2$. Moreover, again by Lemma \ref{lem:airy-0} and \eqref{eq:d-cr-ci}, we have
   \begin{align*}
   	\left|\partial_{c_r}\left(\frac{u'(0)\tilde{\mathcal A}(2,0)}{\tilde{\mathcal A}(1,0)}\right)\right|+\left|\partial_{c_i}\left(\frac{u'(0)\tilde{\mathcal A}(2,0)}{\tilde{\mathcal A}(1,0)}\right)\right|\leq CA^{-\f43}.
   \end{align*}
	Therefore, we obtain that for any $j\geq 1$,
	\begin{align*}
		|c^{(j+1)}-c^{(j)}|\leq& CA^{-\f43}|c^{(j)}-c^{(j-1)}|+C|\varepsilon\log|\varepsilon|| |c^{(j)}-c^{(j-1)}|\\
		\leq&CA^{-\f43}|c^{(j)}-c^{(j-1)}|.
	\end{align*}
	Since $A\gg1$, there exists $c^0$ such that $\lim_{j\to\infty}c^{(j)}=c^0$ and $c^0$ satisfies 
		\begin{align*}
		c^{0}=u'(0)^{-1}\alpha-\mathbb A(c^{0})-\mathcal R_d(\alpha,c^{(0)}).
	\end{align*}
	Hence,  $(\alpha,c_r^{(0)},c_i^{(0)})$ is a zero point of $\mathbb F(\alpha;c_r,c_i)$, which along with \eqref{eq:F-c-Jaco} and \eqref{eq:F-alpha} implies that for  $\alpha\in(A\nu^{\f18},B\nu^{\frac{1}{12}})$, there exists a unique $c(\alpha)$ solving \eqref{eq:solu-c} with $(c(\alpha),\alpha)\in\mathbb H_3$. Moreover, $c(\alpha)$ depends on $\alpha$ continuously.
	\end{proof}
	
By Proposition \ref{prop:diss-priori} and \ref{prop:diss-existence}, we conclude the following result.

\begin{theorem}
Let $\nu\ll1$. Then there exist $A\leq1\ll B$ independent of $\nu$ such that for any $\alpha\in(A\nu^{\f18},B\nu^{\frac{1}{12}})$, there exists a  pair $(c,\phi)$ solving \eqref{eq:OS-nonslip} with $\phi\in W^{4,\infty}$. Moreover, $c$ satisfies the following properties

	\begin{itemize}
		\item There exists $A_1\geq A_c\geq A_0>A$ such that 
		      \begin{align*}
		      	&\mathrm{Im}(c(A_0\nu^\f18))<0,\quad \mathrm{Im}(c(A_c\nu^\f18))=0,\quad\mathrm{Im}(c(A_1\nu^\f18))>0.
		      \end{align*}
		\item There exists $B_c\leq B_0\leq B_1<B$ such that 
		       \begin{align*}
		      	&\mathrm{Im}(c(B_0\nu^\frac{1}{12}))>0,\quad \mathrm{Im}(c(B_c\nu^\frac{1}{12}))=0,\quad\mathrm{Im}(c(B_1\nu^\frac{1}{12}))<0.
		      \end{align*}
		\item If $\nu^\f18\ll\alpha\ll\nu^\frac{1}{12}$, then 
		      \begin{align*}
		      	\alpha\mathrm{Im}(c(\alpha))\sim\nu^\f14.
		      \end{align*}
	\end{itemize}
\end{theorem}

\appendix

\section{The Airy function}\label{Sec:Airy}

 Let $Ai(y)$ be the Airy function, which is a nontrivial solution of $f''-yf=0$. By \cite{OLBC}, we have the following asymptotic formula for $|\arg z|\leq\pi-\delta$ with $\delta>0$ and $|z|\geq M$ for some large $M$,
 \begin{align}\label{eq:airy-decay}
 	\partial^k_z Ai(z)=\frac{1}{2\sqrt{\pi}}z^{-\f14+\f k2}e^{-\f23z^{\f32}}(1+R(z)),\quad R(z)=\mathcal O(z^{-\f32}),\quad\forall k=0,1,2.
 \end{align}  
 We denote
 \begin{align*}
 	&\mathcal A(1,z):=-\int_z^{+\infty}Ai(e^{\mathrm i\frac{\pi}{6}}t)dt,\quad\mathcal A(2,z):=-\int_z^{+\infty}\mathcal A(1,t)dt,\\
 	&\mathcal C(1,z):=\int_{-\infty}^zAi(e^{\mathrm i\frac{5\pi}{6}}t)dt,\quad\mathcal C(2,z):=\int_{-\infty}^z\mathcal B(1,t)dt.
 \end{align*}

\begin{lemma}[Lemma A.2 in \cite{MWWZ}]\label{lem:pri-Airy-decay}
	Suppose that $\delta_0$ is a small positive constant. Assume that $\mathrm{Im}z\leq\delta_0$ and $|z|\geq M$. Then we have
\begin{align*}
	&\mathcal A(1,z)=-e^{-\mathrm i\frac{\pi}{6}}(e^{\mathrm i\frac{\pi}{6}}z)^{-\f34}e^{-\f23(e^{\mathrm i\frac{\pi}{6}}z )^\f32}(1+\mathcal R_1(z)),\\
	&\mathcal A(2,z)=e^{-\mathrm i\frac{\pi}{3}}(e^{\mathrm i\frac{\pi}{6}}z)^{-\f54}e^{-\f23(e^{\mathrm i\frac{\pi}{6}}z )^\f32}(1+\mathcal R_2(z)),\\
	&\mathcal C(1,z)=-e^{-\mathrm i\frac{5\pi}{6}}(e^{\mathrm i\frac{5\pi}{6}}z)^{-\f34}e^{-\f23(e^{\mathrm i\frac{5\pi}{6}}z )^\f32}(1+\mathcal R_3(z)),\\
	&\mathcal C(2,z)=e^{\mathrm i\frac{\pi}{3}}(e^{\mathrm i\frac{5\pi}{6}}z)^{-\f54}e^{-\f23(e^{\mathrm i\frac{5\pi}{6}}z )^\f32}(1+\mathcal R_4(z)),
\end{align*}
where $\mathcal R_i(z)=\mathcal O(z^{-\f32})$ for $i=1,2,3,4.$
\end{lemma}
\begin{lemma}[Lemma A.3 in \cite{MWWZ}]\label{lem:Airy-green-decay}
	Let $\delta_0>0$ be a small constant.  Suppose $-\delta_0<\mathrm{Im}Z_0<\delta_0$, $\mathrm{Re}Z_0<0$.  There exists  $M_1>0$ such that if $|\mathrm{Re}(Y+Z_0)|>M_1\delta_0$, then for any $0\leq Z\leq Y$,
	\begin{align*}
		\left|e^{-\f23(e^{\mathrm i\frac{\pi}{6}}(Y+Z_0))^\f32 } e^{-(e^{\mathrm i\frac{5\pi}{6}}(Z+Z_0))^\f32 }\right|\leq e^{-C_1|Y-Z|(|Y+Z_0|+|Z+Z_0)^\f12}.
	\end{align*}
\end{lemma}

Next, we present the asymptotic formula for the classical Scorer function, which can be found in \cite{OLBC}.

\begin{lemma}\label{lem:asy-Hi}
	Assume $|\arg z|\leq \frac{2\pi}{3}-\delta$, where $\delta$ is an arbitrarily small positive constant. Then the following asymptotic expansion holds:
	\begin{align}\label{eq:asy-int-Hi}
		\int_0^zHi(-t)dt\sim\pi^{-1}\log z+\frac{2\gamma+\log 3}{3\pi}+\pi^{-1}\sum_{k=1}^{\infty}(-1)^{k-1}\frac{(3k-1)!}{k!(3z^3)^k},\quad z\to\infty.
	\end{align}
\end{lemma}

\begin{lemma}\label{lem: est-Scorer}
Let $\delta_0>0$ be a small constant as in Lemma \ref{lem:pri-Airy-decay} and $\hat c:=c+\mathrm i |\varepsilon|\alpha ^{-\f32}$.  Suppose $-\delta_0<\mathrm{Im}(\kappa\eta(0))<\delta_0$ and $(\al,c)\in\mathbb{ H}_2$.  Then for any $Y\in[0,2Y_c]$,
\begin{align*}
	A_2(1,Y)A_1(Y)-A_1(1,Y)A_2(Y)=\pi|\varepsilon|^{-\f13}e^{-\mathrm i\frac{\pi}{2}}u'(Y_c)^{-\f23}Hi(e^{-\mathrm i\frac{\pi}{2}}\kappa\eta(Y))+\mathcal O(|\varepsilon|^{-\f13}c_r).
\end{align*}	
\end{lemma}
\begin{proof}
	By the definition of $A_j(Y)$ and $A_j(1,Y)$, we know that 
	\begin{align*}
		A_2(1,Y)A_1(Y)-A_1(1,Y)A_2(Y)=&-2\pi|\varepsilon|^{-\f23}u'(Y_c)^{-\f13}\mathrm i\int_{-\infty}^YAi(e^{\mathrm i\frac{5\pi}{6}}\kappa\eta(Z))dZAi(e^{\mathrm i\frac{\pi}{6}}\kappa\eta(Y))\\
		&-2\pi|\varepsilon|^{-\f23}u'(Y_c)^{-\f13}\mathrm i\int^{+\infty}_YAi(e^{\mathrm i\frac{\pi}{6}}\kappa\eta(Z))dZAi(e^{\mathrm i\frac{5\pi}{6}}\kappa\eta(Y)).
	\end{align*}
We also notice that 
\begin{align*}
\int^{+\infty}_YAi(e^{\mathrm i\frac{\pi}{6}}\kappa\eta(Z))dZAi(e^{\mathrm i\frac{5\pi}{6}}\kappa\eta(Y))=&\kappa^{-1}\int_Y^{+\infty}	\partial_Z(\mathcal A(1,\kappa\eta(Z))\partial_Z\eta^{-1}dZAi(e^{\mathrm i\frac{5\pi}{6}}\kappa\eta(Y))\\
=&-\kappa^{-1}\mathcal A(1,\kappa\eta(Y))Ai(e^{\mathrm i\frac{5\pi}{6}}\kappa\eta(Y))+\mathcal R_1(Y),
\end{align*}
where 
\begin{align*}
	\mathcal R_1(Y)=&\kappa^{-1}\mathcal A(1,\kappa\eta(Y))Ai(e^{\mathrm i\frac{5\pi}{6}}\kappa\eta(Y))(1-\partial_Y\eta(Y)^{-1})\\
	&+\kappa^{-1}\int_Y^{+\infty}\mathcal A(1,\kappa\eta(Z))\frac{\partial_Z^2\eta}{(\partial_Z\eta)^2}dZAi(e^{\mathrm i\frac{5\pi}{6}}\kappa\eta(Y)).
\end{align*}
On the other hand, by Lemma \ref{lem:est-eta}, we know that for any $Y\in[0,2Y_c]$,
\begin{align*}
	|1-\partial_Y\eta(Y)^{-1}|\leq C c_r\text{ and }|\partial_Y^2\eta(Y)/(\partial_Y\eta(Y))^2 |\leq C,
\end{align*}
which along with Lemma \ref{lem:pri-Airy-decay} and \ref{lem:Airy-green-decay} implies that for $Y\in[0,2Y_c]$,
\begin{align*}
	|\mathcal R_1(Y)|\lesssim \kappa^{-1}(c_r+\kappa^{-1})\lesssim|\varepsilon|^{\f13}c_r.
\end{align*}
Hence, we obtain that for any $Y\in[0,2Y_c]$,
\begin{align*}
		\int^{+\infty}_YAi(e^{\mathrm i\frac{\pi}{6}}\kappa\eta(Z))dZAi(e^{\mathrm i\frac{5\pi}{6}}\kappa\eta(Y))=-\kappa^{-1}\mathcal A(1,\kappa\eta(Y))Ai(e^{\mathrm i\frac{5\pi}{6}}\kappa\eta(Y))+\mathcal O(|\varepsilon|^{\f13}c_r).
\end{align*}
Similarly, we can obtain 
\begin{align*}
\int_{-\infty}^YAi(e^{\mathrm i\frac{5\pi}{6}}\kappa\eta(Z))dZAi(e^{\mathrm i\frac{\pi}{6}}\kappa\eta(Y))=\kappa^{-1}\mathcal C(1,\kappa\eta(Y))Ai(e^{\mathrm i\frac{\pi}{6}}\kappa\eta(Y))+\mathcal O(|\varepsilon|^{\f13}c_r).	
\end{align*}
Thus, we can write that for any $Y\in[0,2Y_c]$,
\begin{align*}
			&A_2(1,Y)A_1(Y)-A_1(1,Y)A_2(Y)\\
	&=-2\pi|\varepsilon|^{-\f13}u'(Y_c)^{-\f23}\mathrm i[\mathcal C(1,\kappa\eta(Y))Ai(e^{\mathrm i\frac{\pi}{6}}\kappa\eta(Y))-\mathcal A(1,\kappa\eta(Y))Ai(e^{\mathrm i\frac{5\pi}{6}}\kappa\eta(Y))]+\mathcal O(|\varepsilon|^{-\f13}c_r)\\
	&=2\pi|\varepsilon|^{-\f13}e^{-\mathrm i\frac{2\pi}{3}}u'(Y_c)^{-\f23}\Big(\int_{-\infty}^{e^{\mathrm i\frac{\pi}{6}}\kappa\eta(Y)}Ai(e^{\mathrm i\frac{2\pi}{3}}z)dzAi(e^{\mathrm i\frac{\pi}{6}}\kappa\eta(Y))+\int_{e^{\mathrm i\frac{\pi}{6}}\kappa\eta(Z)}^{+\infty}Ai(z)dzAi(e^{\mathrm i\frac{5\pi}{6}}\kappa\eta(Y))\Big)\\
	&\quad+\mathcal O(|\varepsilon|^{-\f13}c_r),
\end{align*}
which along with  the facts $Ai(e^{\mathrm i\frac{2\pi}{3}}z)=\f12e^{\mathrm i\frac{\pi}{3}}(Ai(z)-\mathrm iBi(z))$, $\int_{-\infty}^0Bi(z)dz=0$ and $Gi(z)+\mathrm iAi(z)=e^{\mathrm i\frac{\pi}{3}}Hi(e^{-\mathrm i\frac{2\pi}{3}}z)$, implies that 
\begin{align*}
&A_2(1,Y)A_1(Y)-A_1(1,Y)A_2(Y)\\
&=\pi|\varepsilon|^{-\f13}e^{-\mathrm i\frac{\pi}{3}}u'(Y_c)^{-\f23}\int_{-\infty}^{\infty}Ai(z)dzAi(e^{\mathrm i\frac{\pi}{6}}\kappa\eta(Y))\\
&\quad-\pi|\varepsilon|^{-\f13}e^{-\mathrm i\frac{5\pi}{6}}u'(Y_c)^{-\f23}\int_{-\infty}^{e^{\mathrm i\frac{\pi}{6}}\kappa\eta(Y)}Bi(z)dzAi(e^{\mathrm i\frac{\pi}{6}}\kappa\eta(Y))\\
&\quad+\pi|\varepsilon|^{-\f16}e^{-\mathrm i\frac{5\pi}{6}}u'(Y_c)^{-\f23}\int_{e^{\mathrm i\frac{\pi}{6}}\kappa\eta(Y)}^{+\infty}Ai(z)dzBi(e^{\mathrm i\frac{\pi}{6}}\kappa\eta(Y))+\mathcal O(|\varepsilon|^{-\f13}c_r)\\
&=\pi|\varepsilon|^{-\f13}e^{-\mathrm i\frac{5\pi}{6}}u'(Y_c)^{-\f23}\Big(\mathrm i Ai(e^{\mathrm i\frac{\pi}{6}}\kappa\eta(Y))+Gi(e^{\mathrm i\frac{\pi}{6}}\kappa\eta(Y))\Big)+\mathcal O(|\varepsilon|^{-\f13}c_r)\\
&=\pi|\varepsilon|^{-\f13}e^{-\mathrm i\frac{\pi}{2}}u'(Y_c)^{-\f23}Hi(e^{-\mathrm i\frac{\pi}{2}}\kappa\eta(Y))+\mathcal O(|\varepsilon|^{-\f13}c_r).
\end{align*}

The proof is completed.
	\end{proof}

\begin{lemma}[Lemma A.5 in \cite{MWWZ}]\label{lem:airy-0}
	Let $|\varepsilon|,|c|\ll1\ll\kappa$ and $c_r>0$. Suppose $|\kappa\eta_i|<\delta_0$. Then there holds
	\begin{align*}
	&\left|\tilde{\mathcal A}(1,0)-\kappa^{-1}\mathcal A(1,\kappa\eta(0))\right|+\kappa\left|\tilde{\mathcal A}(2,0)-\kappa^{-2}\mathcal A(2,\kappa\eta(0))\right|\leq C\kappa^{-1}(|c|+\kappa^{-1}),
		\end{align*}
	and
	\begin{align*}
	\left|\frac{\tilde{\mathcal A}(2,0)}{\tilde{\mathcal A}(1,0)}-\kappa^{-1}\frac{\mathcal A(2,\kappa\eta(0))}{\mathcal A(1,\kappa\eta(0))}\right|\leq C\kappa^{-1}(\kappa^{-1}+|c|).
	\end{align*}
	Moreover, if $|\kappa\eta(0)|>M$, then we have
	\begin{align*}
	\frac{\tilde{\mathcal A}(2,0)}{\tilde{\mathcal A}(1,0)}
	=&-e^{\mathrm i(\f{\pi}{4}-\f{\theta_0}{2})}|\varepsilon|^\f12c_r^{-\f12}+o(|\varepsilon|^\f12c_r^{-\f12}+|\varepsilon|^{\f13}),\quad
	\frac{Ai(e^{\mathrm i\frac{\pi}{6}}\kappa\eta(0))\tilde{\mathcal A}(2,0)}{\tilde{\mathcal A}(1,0)^2}=1+\mathcal O(|\varepsilon|^\f12c_r^{-\f32}).
	\end{align*}
\end{lemma}

\section*{Acknowledgments}
 Q. Chen is supported by NSF of China under Grant 12288201. D. Wu is supported by NSF of China under Grant 12471196.  Z. Zhang is partially supported by  NSF of China  under Grant 12171010 and 12288101. 

\end{document}